\documentclass[11pt, reqno]{amsart}
\usepackage[T1]{fontenc}
\usepackage[utf8]{inputenc}
\usepackage[english]{babel}
\usepackage{amsmath,amsfonts,amssymb,amsthm,mathrsfs}
\usepackage{accents,mathtools,extarrows}
\usepackage{yhmath}
\usepackage{graphicx}
\usepackage{rotating}
\usepackage{wasysym}
\usepackage{textcomp}
\usepackage{hyperref,cleveref}
\usepackage{xfrac}
\usepackage{enumitem}
\usepackage{tikz-cd}									
\usetikzlibrary{matrix,arrows,calc,positioning,decorations.pathmorphing}
\tikzset{
  shift left/.style ={commutative diagrams/shift left={#1}},
  shift right/.style={commutative diagrams/shift right={#1}}
}
\usepackage{float}
\usepackage{todonotes}

\theoremstyle{plain}
\newtheorem{theorem}{Theorem}[section]
\newtheorem{lemma}[theorem]{Lemma}
\newtheorem{proposition}[theorem]{Proposition}
\newtheorem{corollary}[theorem]{Corollary}
\newtheorem*{theorem*}{Theorem}
\newtheorem*{proposition*}{Proposition}
\newtheorem*{corollary*}{Corollary}

\theoremstyle{definition}
\newtheorem{definition}[theorem]{Definition}
\AtBeginEnvironment{definition}{\pushQED{\qed}}
\AtEndEnvironment{definition}{\popQED\enddefinition}
\newtheorem{example}[theorem]{Example}
\AtBeginEnvironment{example}{\pushQED{\qed}}
\AtEndEnvironment{example}{\popQED\endexample}
\newtheorem{remark}[theorem]{Remark}
\AtBeginEnvironment{remark}{\pushQED{\qed}}
\AtEndEnvironment{remark}{\popQED\endremark}
\newtheorem{construction}[theorem]{Construction}
\AtBeginEnvironment{construction}{\pushQED{\qed}}
\AtEndEnvironment{construction}{\popQED\endconstruction}
\newtheorem{observation}[theorem]{Observation}
\AtBeginEnvironment{observation}{\pushQED{\qed}}
\AtEndEnvironment{observation}{\popQED\endobservation}
\newtheorem{observation/construction}[theorem]{Obervation/Construction}
\AtBeginEnvironment{observation/construction}{\pushQED{\qed}}
\AtEndEnvironment{observation/construction}{\popQED\endobservation}
\newtheorem{notation}[theorem]{Notation}
\AtBeginEnvironment{notation}{\pushQED{\qed}}
\AtEndEnvironment{notation}{\popQED\endnotation}
\newtheorem{assumption}[theorem]{Assumption}
\AtBeginEnvironment{assumption}{\pushQED{\qed}}
\AtEndEnvironment{assumption}{\popQED\endassumption}
\newtheorem*{remark*}{Remark}
\AtBeginEnvironment{remark*}{\pushQED{\qed}}
\AtEndEnvironment{remark*}{\popQED\endremark}

\newtheorem*{thexreftheorem}{Theorem \thexref}

\newtheorem*{thexrefcorollary}{Corollary \thexcoro}
\newcommand{\Z}{\mathbb Z}

\newcommand{\R}{\mathbb R}
\newcommand{\C}{\mathbb C}

\renewcommand{\epsilon}{\varepsilon}
\renewcommand{\phi}{\varphi}

\begin{document}

\title[Moduli stack of stable curves from a stratified homotopy viewpoint]{Moduli stack of stable curves from a stratified homotopy viewpoint}
\author{Mikala Ørsnes Jansen}
\date{\today}
\address{Department of Mathematical Sciences, University of Copenhagen, 2100 Copenhagen, Denmark.}\email{mikala@math.ku.dk}
\thanks{The author was supported by the European Research Council (ERC) under the European Unions Horizon 2020 research and innovation programme (grant agreement No 682922) and the Danish National Research Foundation through the Copenhagen Centre for Geometry and Topology (DNRF151).}

\begin{abstract}
In 1984, Charney and Lee defined a category of stable curves and exhibited a rational homology equivalence from its geometric realisation to (the analytification of) the moduli stack of stable curves, also known as the Deligne--Mumford--Knudsen compactification. We strengthen this result by showing that, in fact, this category captures the stratified homotopy type of the moduli stack. In particular, it classifies constructible sheaves via an exodromy equivalence.
\end{abstract}

\maketitle

\tableofcontents

\section{Introduction}

\textbf{Background and aim.} Let $\overline{\mathcal{M}}{}_{g,n}$ denote the moduli stack of stable genus $g$ nodal curves with $n$ marked points, also know as the \textit{Deligne--Mumford--Knudsen compactification} (when $n=0$, we omit the $n$). In 1984, Charney and Lee defined a category $\mathfrak{SC}_g$, called the \textit{category of stable curves}, and proved that there is a rational homology equivalence from the geometric realisation $|\mathfrak{SC}_g|$ to (the analytification of) the coarse moduli space of $\overline{\mathcal{M}}{}_g$ (\cite{CharneyLee84}). In 2008, Ebert and Giansiracusa provided an integral refinement of this theorem: generalising the category $\mathfrak{SC}_g$ to include marked points, they both generalised and strengthened the result by showing that, in fact, the geometric realisation of the resulting category $\mathfrak{SC}_{g,n}$ has the same \textit{homotopy type} as the underlying analytic stack $\overline{\mathcal{M}}{}^{\operatorname{an}}_{g,n}$. Moreover, both $\mathfrak{SC}_{g,n}$ and $\overline{\mathcal{M}}{}^{\operatorname{an}}_{g,n}$ come equipped with natural stratifications and the established homotopy equivalence is functorial with respect to inclusions of strata and closures of strata (\cite{EbertGiansiracusa}).

The purpose of this paper is to further strengthen these results: we prove that the category $\mathfrak{SC}_{g,n}$ captures not merely the homotopy type of the stack but the \textit{stratified homotopy type}. The fact that $\mathfrak{SC}_{g,n}$ is a $1$-category may be the interpreted as saying that $\overline{\mathcal{M}}_{g,n}$ is a $K(\pi,1)$ in a stratified sense. The most important feature of the stratified homotopy type is the resulting classification of constructible sheaves, that is, sheaves which are locally constant along each stratum. As a consequence, we obtain an identification of the category of constructible sheaves on $\overline{\mathcal{M}}{}_{g,n}$ with the presheaf category on $\mathfrak{SC}_{g,n}$; moreover, we can identify the derived constructible category with a derived functor category. To the best of the author's knowledge, such a classification is new and it may provide a useful tool for studying cohomology of $\overline{\mathcal{M}}_{g,n}$ with coefficients which are not necessarily constant nor even locally constant. 

Our result cements the fact that the intuition displayed by Charney--Lee in 1984 goes far beyond the established rational homology equivalence (as is also evident when reading their paper); their understanding of the Deligne--Mumford--Knudsen compactification precedes the development of stratified homotopy theory (\cite{Treumann, LurieHA}) and the present paper should be viewed as a mere footnote to this story. We simply establish what was anticipated by the methods used in \cite{CharneyLee84}.

\medskip

\textbf{Stratified homotopy theory.} Let us give a brief overview of stratified homotopy theory to provide some intuition of the invariant in question. It is a classical result in algebraic topology that there is an equivalence between locally constant sheaves on a sufficiently nice topological space and representations of its fundamental groupoid, the so-called monodromy equivalence. In unpublished work, MacPherson made the observation that analogously to how locally constant sheaves are classified by paths in the topological space, the constructible sheaves on a stratified space can be classified by so-called \textit{exit paths}. Constructible sheaves are sheaves which are locally constant along each stratum, and exit paths are stratum preserving paths, that is, an exit path can move into ``more open'' strata, but not the other way. Treumann gave a $2$-categorical version of this result (\cite{Treumann}), and Lurie has developed the $\infty$-categorical setting (\cite[Appendix A]{LurieHA}). To a sufficiently nice stratified space, Lurie associates an \textit{exit path $\infty$-category}; this is a refinement of the fundamental $\infty$-groupoid incorporating the additional structure of a stratification. Roughly put, the exit path $\infty$-category captures the homotopy types of the individual strata and moreover the ``homotopy type of the glueing data'', that is, how these individual homotopy types fit together. Stratified homotopy theory has been used and developed in several directions in recent years: model structures and a stratified homotopy hypothesis (\cite{AyalaFrancisRozenblyum,AyalaFrancisTanaka,Douteau,Haine,Nand-Lal}); exodromy for schemes (\cite{BarwickGlasmanHaine, BarwickHaine19a}); and concrete calculations (\cite{Cepek,OrsnesJansen}).

In \cite{OrsnesJansen23}, we provide a generalisation of stratified homotopy types to stratified topological stacks: in the spirit of \cite{ClausenOrsnesJansen}, we opt for a non-constructive approach and define the exit path $\infty$-category of a stratified topological stack as the idempotent complete $\infty$-category classifying constructible sheaves (if it exists). Without getting into the finer technical details of the situation, a stratified topological $\infty$-stack is a space-valued sheaf $X\in \operatorname{Shv}(\mathfrak{T})$ over the site $\mathfrak{T}=\operatorname{LCHaus}^{2nd}$ of (locally compact Hausdorff second countable) topological spaces equipped with a morphism $X\rightarrow P$ to a poset $P$ satisfying the ascending chain condition (where $P$ is viewed as a topological space using the Alexandroff topology). We say that such a stratified topological $\infty$-stack $(X,P)$ \textit{admits an exit path $\infty$-category} if there is an idempotent complete $\infty$-category $\Pi(X,P)$ and an equivalence
\begin{align*}
\operatorname{Shv}_P^{\operatorname{cbl}}(X)\xrightarrow{\ \sim \ } \operatorname{Fun}(\Pi(X,P),\mathcal{S})
\end{align*}
between the $\infty$-category of $P$-constructible space-valued sheaves on $X$ and functors from $\Pi(X,P)$ into the $\infty$-category of spaces. We moreover require the pullback functor $\operatorname{Fun}(P,\mathcal{S})\rightarrow \operatorname{Shv}_P^{\operatorname{cbl}}(X)$ along the stratification map to preserve limits. In this case, we call $\Pi(X,P)$ the \textit{exit path $\infty$-category} of $(X,P)$. This definition can be spelt out in terms of more tangible technical assumptions on the $\infty$-category of constructible sheaves using the notion of atomic objects (see \Cref{definition exit path category}). Since the overall objective of this invariant is the classification of constructible sheaves above, we refrain from going into details in this paper and refer the reader to \cite{OrsnesJansen23} in which we also recall what we mean by sheaves on $\infty$-stacks and develop a number of calculational tools that will be applied here.

\medskip

\textbf{Main results.} Let $g,n$ be such that $2g-2+n>0$ and consider the (algebraic) moduli stack $\overline{\mathcal{M}}_{g,n}$ of stable $n$-pointed genus $g$ nodal curves. Let $\overline{\mathcal{M}}^{\operatorname{top}}_{g,n}$ denote the underlying topological stack and consider the canonical stratification of this over the poset $\mathcal{G}_{g,n}$ of stable $n$-pointed genus $g$ dual graphs --- that is, we stratify according to the homeomorphism types of the analytifications of the algebraic curves, equivalently, the boundary is stratified as a normal crossings divisor (\Cref{underlying stratified topological stack}). Let us briefly remark that the underlying topological stack is obtained by a natural extension of the functor sending an affine scheme $U$ of finite type over $\C$ to the topological space $U^{\operatorname{top}}$ underlying its analytification (see \Cref{underlying stratified topological stack} and \cite[\S 4]{OrsnesJansen23}).

Our main result is the following.

\begin{theorem*}[\ref{the exit path category of Mg bar}  and \ref{exodromy equivalence}]
The stratified topological stack $\overline{\mathcal{M}}^{\operatorname{top}}_{g,n}\rightarrow \mathcal{G}_{g,n}$ admits an exit path $\infty$-category (in the sense of \Cref{definition exit path category}) and this exit path $\infty$-category is equivalent to a $1$-category $\Pi_{g,n}$.

In particular, for any compactly generated $\infty$-category $\mathcal{C}$, there is an equivalence
\begin{align*}
\operatorname{Shv}^{\operatorname{cbl}}(\overline{\mathcal{M}}^{\operatorname{top}}_{g,n};\mathcal{C})\xrightarrow{\ \sim\ }\operatorname{Fun}(\Pi_{g,n},\mathcal{C}),
\end{align*}
between the $\infty$-category of sheaves on $\overline{\mathcal{M}}^{\operatorname{top}}_{g,n}$ that are constructible with respect to the fixed stratification over $\mathcal{G}_{g,n}$ and the $\infty$-category of functors $\Pi_{g,n}\rightarrow \mathcal{C}$.

Moreover, the exit path category identifies with (the opposite of) the $1$-category of stable curves introduced by Charney--Lee, $\Pi_{g,n}\simeq \mathfrak{SC}_{g,n}^{\operatorname{op}}$.
\end{theorem*}

The category $\Pi_{g,n}$ can be defined very explicitly; we call it the \textit{Charney--Lee category associated to a stable surface of type $(g,n)$} (see \Cref{CharneyLee categories} where we use the notation $\Pi_{g,n}=\operatorname{CL}_S$ for a surface $S$). It can in fact also be identified with a full subcategory of the orbit category of the mapping class group of $S$.

As mentioned, the fact that the exit path $\infty$-category of $\overline{\mathcal{M}}_{g,n}$ is just a $1$-category can be interpreted as saying that $\overline{\mathcal{M}}_{g,n}$ is a $K(\pi,1)$ in a stratified sense. It also has the very nice consequence that we can express the constructible derived category on $\overline{\mathcal{M}}_{g,n}$ as a derived functor category. For a given associative ring $R$, we consider the derived $\infty$-category of left $R$-modules, $\operatorname{LMod}_R\simeq \mathscr{D}(R)$, and the bounded constructible derived $\infty$-category
\begin{align*}
\mathscr{D}^b_{\operatorname{cbl}}(\operatorname{Shv}_1(\overline{\mathcal{M}}^{\operatorname{top}}_{g,n};R))\xrightarrow{\sim}\operatorname{Shv}^{\operatorname{cbl,cpt}}(\overline{\mathcal{M}}^{\operatorname{top}}_{g,n};\operatorname{LMod}_R);
\end{align*}
here 
the left hand side is the full subcategory of the derived $\infty$-category spanned by complexes of sheaves whose homology sheaves are constructible and whose stalk complexes are perfect, the right hand side is the full subcategory of constructible sheaves whose stalks are compact objects in $\operatorname{LMod}_R$, i.e. perfect complexes.

\begin{theorem*}[\ref{equivalence of derived infinity-categories}]
Let $R$ be an associative ring. There is an equivalence
\begin{align*}
\operatorname{Shv}^{\operatorname{cbl,cpt}}(\overline{\mathcal{M}}^{\operatorname{top}}_{g,n};\operatorname{LMod}_R)\xrightarrow{\sim} \mathscr{D}^{\operatorname{cpt}}\big(\operatorname{Fun}(\Pi_{g,n},\operatorname{LMod}_R^1)\big),
\end{align*}
identifying the bounded constructible derived $\infty$-category
with the full subcategory of the derived functor $\infty$-category spanned by complexes of functors $F_\bullet\colon \Pi_{g,n}\rightarrow \operatorname{LMod}_R^1$ into the $1$-category of left $R$-modules satisfying that $F_\bullet(x)$ is a perfect complex for every object $x$ of $\Pi_{g,n}$.
\end{theorem*}

There is also a version of this without the bounded condition, and, moreover, taking homotopy categories provides an analogous identification of the more classical $1$-categorical constructible derived categories (see \Cref{consequences}).

As immediate corollaries, we recover the results of Charney--Lee and Ebert--Giansiracusa. This follows from the following corollary where by the fundamental $\infty$-groupoid $\Pi(\overline{\mathcal{M}}^{\operatorname{top}}_{g,n})$ we simply mean the \textit{shape} of $\operatorname{Shv}(\overline{\mathcal{M}}^{\operatorname{top}}_{g,n})$; equivalently, it is the idempotent complete $\infty$-category classifying locally constant sheaves (see \cite[\S 3.3]{OrsnesJansen23}).

\begin{corollary*}[\ref{homotopy type}]
There is a weak homotopy equivalence $\Pi_{g,n}\rightarrow \Pi(\overline{\mathcal{M}}^{\operatorname{top}}_{g,n})$ to the fundamental $\infty$-groupoid of $\overline{\mathcal{M}}^{\operatorname{top}}_{g,n}$.
\end{corollary*}

\begin{remark*}
We also recover the functoriality established by Ebert--Giansiracusa (\Cref{EG-functoriality of weak equivalence}). Note that this functoriality is not quite sufficient to determine the exit path $\infty$-category as we might not account for the ``glueing data'' (see the example given in \cite[Remark 3.20]{OrsnesJansen}).
\end{remark*}

We also calculate the exit path $\infty$-categories of the following moduli stacks and obtain results analogous to the ones above:
\begin{itemize}
\item the moduli stack $\overline{\mathcal{M}}_{g,P}$ of stable $P$-pointed genus $g$ nodal curves for some finite set $P$ (the case $P=\{0,1,\ldots,n\}$ then recovers $\overline{\mathcal{M}}_{g,n}$). Working with general indexing sets enables us to keep track of nodes and marked points when we consider nodal surfaces part by part;
\item the moduli stack $\overline{\mathcal{M}}_{\mathbf{G}}$ of stable nodal curves with associated dual graph a specialisation of a given stable dual graph $\mathbf{G}$ --- this generalises the moduli stack $\overline{\mathcal{M}}_{g,P}$ above and allows us to make a cleaner inductive argument (\Cref{the exit path category of Mg bar});
\item the Harvey bordification $\overline{\mathcal{T}}_{\!\!S}$ of Teichmüller space associated to a stable nodal surface of finite type $S$ (\Cref{exit path category of Harvey bordification});
\item the Harvey compactification $[\Gamma_S\backslash \overline{\mathcal{T}}_{\!\!S}]$ of the open substack $[\Gamma_S\backslash \mathcal{T}_S]\cong \mathcal{M}^{\operatorname{top}}_{\mathbf{G}}\subseteq \overline{\mathcal{M}}_{\mathbf{G}}^{\operatorname{top}}$ where $\mathbf{G}$ is the dual graph associated to the stable nodal surface $S$ (\Cref{exit path category of Harvey compactification}).
\end{itemize}
 
Let us briefly remark here that the results and proof strategies of the present paper are analogous to the ones of \cite{ClausenOrsnesJansen} in which we determine the exit path $\infty$-category of the reductive Borel--Serre compactification of a locally symmetric space.

\medskip

\textbf{Outline of paper.} In Sections 2 through 4, we introduce and recall the necessary objects, constructions and properties enabling our final calculation. Let us stress that these sections contain no original material; we simply present the relevant background in a way that fits into the setting at hand. We do provide an alternative description of the category originally defined by Charney and Lee (\cite{CharneyLee84}, cf. \Cref{The category}), but this is a simple observation albeit absolutely essential to our calculational strategy. We do not pretend to provide a complete review of the background and refer instead to other sources for details and proofs. Our motivation (or excuse) for giving a detailed review is that, to the best of the author's knowledge, the moduli stack of stable curves has not previously been studied in the setting of stratified topological stacks; therefore we choose to introduce quite a lot of details in order to provide rigorous setup.

In Section 2, we introduce the moduli stacks in question: the Deligne--Mumford--Knudsen compactifications and their underlying stratified topological stacks. In Section 3, we introduce the category of stable curves, and in Section 4, we recall Teichmüller space and its Harvey bordification and make some preliminary observations.

The original material is presented in Sections 5 and 6 and follows the calculational strategy of \cite{ClausenOrsnesJansen}. In Section 5, we exploit various descent properties of sheaf categories to exhibit the inductive nature of the Deligne--Mumford--Knudsen compactifications and, more importantly perhaps, we provide a diagrammatic analogue of these inductive decompositions. The actual calculation is done in Section 6 where we combine the established inductive decompositions with certain permanence properties of exit path $\infty$-categories. The idea is to compatibly cover $\overline{\mathcal{M}}_{g,P}$  by stratified spaces whose exit path $\infty$-categories identify with their stratifying posets --- this should be interpreted as the analogue for stratified spaces of being contractible. We exploit the so-called Harvey bordification to construct this cover. The inductive nature of $\overline{\mathcal{M}}_{g,P}$ and the diagrammatic analogues allow us to show that sheaves on $\overline{\mathcal{M}}_{g,P}$ are determined by compatible sheaves on the cover (\Cref{sheaves on Mg bar as limit of simple pieces}). In the end, this allows us to calculate the exit path $\infty$-category as a colimit of posets in $\operatorname{Cat}_\infty$.

We include a small appendix on dual graphs for the reader unfamiliar with these details. We also include a rather comprehensive appendix identifying the Harvey compactification as constructed by Ivanov (\cite{Ivanov}) with the real oriented blow-up of the Deligne--Mumford--Knudsen compactification --- this is an observation originally due to Looijenga (\cite{Looijenga95}).

\medskip

\textbf{Acknowledgements.} The author would like to thank Søren Galatius and Dan Petersen for valuable comments and insights into the Deligne--Mumford--Knudsen compactification and related stacks. We thank Dustin Clausen for many useful conversations and also for the collaboration in \cite{ClausenOrsnesJansen}; the calculational strategy of that paper made the present calculation very manageable. 

\medskip

\textbf{Notation and conventions.} All schemes are assumed to be finite type over $\C$ unless otherwise stated.

We adopt the set-theoretic conventions of \cite[\S 1.2.15]{LurieHTT}

\section{The Deligne--Mumford--Knudsen compactification}\label{DMK section}

In this section, we introduce the moduli stacks of interest, namely the moduli stack of stable nodal curves, also known as the Deligne--Mumford--Knudsen compactification. We start by introducing the algebraic stacks and go on to consider their underlying topological stacks and their natural stratifications in terms of stable dual graphs. We refer to \cite{ArbarelloCornalbaGriffiths} for details.

\subsection{Stable curves and a range of moduli stacks}\label{stable curves dual graphs and moduli stacks}

What follows is a speedy recap of the objects under consideration; we refer to \cite[Chapter XII, \S 8]{ArbarelloCornalbaGriffiths} for details on ($1$-)stacks.

By a \textit{curve} we mean a reduced, proper, connected scheme $X$ of dimension $1$ over $\C$. By the \textit{genus $g$} of a curve $X$ we mean its arithmetic genus $g=\operatorname{dim}_{\C}H^1(X,\mathcal{O}_X)$. A \textit{node} (or an ordinary double point) of $X$ is a point $p\in X(\C)$ with the property that the completed local ring $\widehat{\mathcal{O}}_{X,p}$ is isomorphic to $\C[[x,y]]/(xy)$ as a complete $\C$-algebra. A curve is \textit{nodal} if its only singularities, if any, are nodes. Let $P$ be a finite set. A \textit{$P$-pointed nodal curve} is the data $(X;\{p_i\}_{i\in P})$ of a nodal curve $X$ and a collection of distinct smooth points $p_i\in X$, $i\in P$. We often write $(X;p_i)$ leaving $P$ implicit.

A $P$-pointed nodal curve $(X;p_i)$ is said to be \textit{stable} if its group of automorphisms is finite, where the automorphisms are required to fix the marked points pointwise. We will from now on assume that all curves considered are nodal and thus refer to stable nodal curves simply as stable curves; we say that such a curve is \textit{smooth} if we want to stress that it has no singularities. It is well-known that there exist stable $P$-pointed genus $g$ curves if and only if $2g-2+|P|>0$ (see e.g.~\cite[Chapter X, Lemma 3.1]{ArbarelloCornalbaGriffiths}); from now on, whenever consider a pair $(g,P)$, we will implicitly assume that $2g-2+|P|>0$.

\begin{remark}
A bijection $P\cong P'$ of finite sets turns a $P$-pointed curve into a $P'$-pointed curve, so we can always reduce to $P=\{1,\ldots,n\}$ for some $n$. It is useful, however, to be able to mark by a general finite set, for example when it comes to the so-called clutching morphisms (see e.g.~\cite[Chapter X, \S 7]{ArbarelloCornalbaGriffiths}). We will not need to go into the finer details of these morphisms, but they do play an important role in our calculational strategy, so we choose to work in this generality.
\end{remark}

\begin{definition}
Let $\overline{\mathcal{M}}_{g,P}$ denote the category whose objects are flat, proper morphisms $\xi\colon X\rightarrow T$ in $\operatorname{Sch}_{\C}^{ft}$ together with disjoint sections $\sigma_i\colon T\rightarrow X$, $i\in P$, such that the geometric fibres, together with their marked points given by the $\sigma_i$, are stable $P$-pointed genus $g$ curves. A morphism in $\overline{\mathcal{M}}_{g,P}$ from $(X'\xrightarrow{\xi'} T';\sigma_i')$ to $(X\xrightarrow{\xi} T; \sigma_i)$ in $\overline{\mathcal{M}}_{g,P}$ is given by a cartesian diagram respecting the sections:
\begin{center}
\begin{tikzpicture}
\matrix (m) [matrix of math nodes,row sep=2em,column sep=2em]
  {
X' & X \\
T' & T \\
  };
  \path[-stealth]
(m-1-1) edge (m-1-2) edge node[left]{$\xi'$} (m-2-1)
(m-2-1) edge (m-2-2)
(m-1-2) edge node[right]{$\xi$} (m-2-2)
;
\end{tikzpicture}
\end{center} 

We implicitly consider the category $\overline{\mathcal{M}}_{g,P}$ together with the functor $\overline{\mathcal{M}}_{g,P}\rightarrow \operatorname{Sch}_{\C}^{ft}$ sending a family $\xi\colon X\rightarrow T$ to the base $T$. An object $(\xi\colon X\rightarrow T; \sigma_i)$ of $\overline{\mathcal{M}}_{g,P}$ is called \textit{a family of stable $P$-pointed genus $g$ curves parametrised by $T$}. Let $\mathcal{M}_{g,P}$ denote the full subcategory of families whose geometric fibres are smooth curves.
\end{definition}

The following is well-known and can be found in e.g.~\cite[Chapter XII, Theorem 8.3]{ArbarelloCornalbaGriffiths}.

\begin{theorem}
$\overline{\mathcal{M}}_{g,P}$ and $\mathcal{M}_{g,P}$ are Deligne--Mumford stacks.
\end{theorem}

The stack $\overline{\mathcal{M}}_{g,P}$ is also known as the \textit{Deligne--Mumford--Knudsen compactification of} $\mathcal{M}_{g,P}$ or the \textit{moduli stack of stable $P$-pointed genus $g$ nodal curves}.

\medskip

It will be useful for us to generalise the definition of the moduli stack of stable curves introduced above --- our calculational strategy is to exploit the inductive nature of these stacks as will become apparent in \Cref{inductive decompositions}. In order to make this generalisation we will need the notion of dual graphs. We make a brief summary of the necessary observations and constructions here and refer the reader to \Cref{dual graphs} for details.

Recall that a $P$-pointed dual graph $\mathbf{G}=(G,\omega,m)$ consists of the following data:
\begin{itemize}
\item a multigraph $G=(V,E,\iota)$ consisting of
\begin{itemize}
\item a set $V$ of vertices;
\item a set of half-edges $E=\coprod_{v\in V}E_v$ equipped with a partition indexed by $V$;
\item and an involution $\iota\colon E\rightarrow E$. The fixed points of $\iota$ are call \textit{legs} and the pairs of half-edges that are interchanged by $\iota$ are called \textit{edges};
\end{itemize}
\item a \textit{weight function} $\omega\colon V\rightarrow \Z_{\geq 0}$, $v\mapsto g_v$;
\item and a $P$-marking $m\colon P\rightarrow E$, i.e.~a bijective labelling of the legs.
\end{itemize}
A dual graph is \textit{stable} if $2g_v-2+|E_v|>0$ for all $v\in V$. The \textit{genus} of a dual graph is defined to be $g(G,\omega)=1-\chi(G)+\sum_{v\in V}g_v$.

Edge contraction defines a poset relation on the set of isomorphism classes of stable $P$-pointed genus $g$ dual graphs (see \Cref{contraction of graph} for details):
\begin{align*}
[\mathbf{G}']\leq [\mathbf{G}]\quad&\text{if } \mathbf{G}'\text{ is a \textit{specialisation} of }\mathbf{G}, \\
&\text{i.e.~if }\mathbf{G}\text{ is isomorphic to a \textit{contraction} of }\mathbf{G}':\ \mathbf{G}\cong \mathbf{G}'_I\text{ for some } I.
\end{align*}
We denote the resulting poset by $\mathcal{G}_{g,P}$. The largest element is the one vertex graph of weight $g$. Intuitively, the more edges a graph has, the smaller it is.

One can associate a stable $P$-pointed genus $g$ dual graph $\mathbf{G}(X;p_i)$ to a stable $P$-pointed genus $g$ nodal curve $(X;p_i)$: the vertices are the connected components of the normalisation of $X$ each with weight given by the genus of the component, the edges are given by the nodes and the legs are given by the marked points. See \Cref{dual graph associated to stable curve} for details.

\medskip

Now, let $\mathbf{G}=(G,\omega,m)$ be a stable $P$-pointed genus $g$ dual graph, write $g_v:=\omega(v)$ and consider the set $P_v$ of half-edges incident to $v$ for each vertex $v$ of $G$. Let
\begin{align*}
\overline{\mathcal{M}}_{\mathbf{G}}^\sim:=\prod_{v\in V} \overline{\mathcal{M}}_{g_v,P_v}
\end{align*}
denote the product of moduli stacks of stable curves corresponding to the type $(g_v,P_v)$ of each vertex, and let $\mathcal{M}_{\mathbf{G}}^\sim:=\prod_{v\in V} \mathcal{M}_{g_v,P_v}$ denote the analogously defined open substack. The group $\operatorname{Aut}(\mathbf{G})$ acts on $\overline{\mathcal{M}}_{\mathbf{G}}^\sim$ and we define quotient stacks
\begin{align*}
\overline{\mathcal{M}}_{\mathbf{G}}:=[\operatorname{Aut}(\mathbf{G})\backslash\overline{\mathcal{M}}_{\mathbf{G}}^\sim],\quad\text{and}\quad \mathcal{M}_{\mathbf{G}}:=[\operatorname{Aut}(\mathbf{G})\backslash \mathcal{M}_{\mathbf{G}}^\sim].
\end{align*}

\begin{remark}
By quotient stack, we mean the $\infty$-quotient (or homotopy quotient),
\begin{align*}
[G\backslash \mathcal{M}]=\varinjlim_{BG}\mathcal{M},
\end{align*}
in the $\infty$-category $\operatorname{Shv}(\operatorname{Aff}_{\C}^{ft})$ of $\infty$-stacks over the étale site $\operatorname{Aff}_{\C}^{ft}$ of affine schemes of finite type over $\C$. We note, however, that in the case at hand this is equivalent to the quotient stack for a finite group action defined in \cite[Chapter XII, Example 8.15]{ArbarelloCornalbaGriffiths}. Indeed, if $X\rightarrow \mathcal{M}$ is an étale atlas, then both the $\infty$-quotient and the $1$-quotient admit an atlas morphism from $X$ giving rise to equivalent \v{C}ech nerves (see \cite{NikolausSchreiberStevenson}).
\end{remark}

We also consider the substack $\overline{\mathcal{D}}_{\mathbf{G}}\subset \overline{\mathcal{M}}_{g,P}$ of families $(\xi\colon X\rightarrow T; \sigma_i)$ whose geometric fibres have dual graphs which are specialisations of $\mathbf{G}$, that is, for each geometric point $t\in T$, the isomorphism class $[\mathbf{G}(X_t;\sigma_i(t))]$ of the dual graph associated to the geometric fibre $(X_t;\sigma_i(t))$ belongs to the poset $(\mathcal{G}_{g,P})_{\leq [\mathbf{G}]}$. Finally, let $\mathcal{D}_{\mathbf{G}}\subset\overline{\mathcal{D}}_{\mathbf{G}}$ denote the open substack of families whose geometric fibres have dual graphs \textit{isomorphic} to $\mathbf{G}$.

We wish to compare the stacks $\overline{\mathcal{M}}_{\mathbf{G}}^\sim$, $\overline{\mathcal{M}}_{\mathbf{G}}$ and $\overline{\mathcal{D}}_{\mathbf{G}}$ and their corresponding open substacks. Note that they are all Deligne--Mumford stacks. 

\begin{remark}
We refer the reader to \cite[Chapter XII, \S 10]{ArbarelloCornalbaGriffiths} for more details on the above, but beware that our notation differs quite a bit. For the reader's convenience, we very briefly explain the differences:
\begin{itemize}
\item First of all, they denote the dual graph by $\Gamma$. We reserve this for mapping class groups as we will see shortly;
\item They write $\overline{\mathcal{M}}_{\mathbf{G}}$ for our $\overline{\mathcal{M}}_{\mathbf{G}}^\sim$. Our choice of notation means that $\overline{\mathcal{M}}_{\mathbf{G}}$ is in a natural way a compactification of the stratum $\mathcal{D}_{\mathbf{G}}$ (see \Cref{DMK compactification as normalisation} and \Cref{DMK compactification associated to dual graph} below). This makes the inductive nature of these compactifications clearer.
\item They do not have a separate notation for the quotient by $\operatorname{Aut}(\mathbf{G})$, but they do define a stack $\mathcal{E}_{\mathbf{G}}$ which is equivalent to our $\overline{\mathcal{M}}_{\mathbf{G}}$. Theirs is perhaps a more satisfying definition as it is more intrinsic, giving a moduli description in terms of families of $P$-pointed genus $g$ curves equipped with a weak $\mathbf{G}$-marking. We choose our definition simply to avoid having to go into the details of weak $\mathbf{G}$-markings, but we refer the interested reader to \cite[p.314]{ArbarelloCornalbaGriffiths}.
\item We are writing $\overline{\mathcal{D}}_{\mathbf{G}}$ for their $\mathcal{D}_{\mathbf{G}}$; our $\mathcal{D}_{\mathbf{G}}$ refers to the open substack of families of curves with dual graph \textit{isomorphic} to $\mathbf{G}$, not just to a specialisation of it. The stack $\overline{\mathcal{D}}_{\mathbf{G}}$ is the closure of $\mathcal{D}_{\mathbf{G}}$, and this notation fits better into the stratified setting.\qedhere
\end{itemize}
\end{remark}

The clutching morphisms give rise to a map
\begin{align*}
\overline{\mathcal{M}}_{\mathbf{G}}^\sim \longrightarrow \overline{\mathcal{D}}_{\mathbf{G}}
\end{align*}
(see \cite[Chapter X, \S 7 and Chapter XII, \S 10]{ArbarelloCornalbaGriffiths}). Intuitively, this map is given by forgetting the ``$\mathbf{G}$-marking'': a geometric point in $\overline{\mathcal{M}}_{\mathbf{G}}^\sim$ is a collection of $P_v$-pointed genus $n_v$ nodal curves which we can clutch together according to the edges of the dual graph $\mathbf{G}$ (we glue the pairs of marked points corresponding to an edge). The resulting curve $(X;p_i)$ is a $P$-pointed genus $g$ nodal curve and its associated dual graph is a specialisation of $\mathbf{G}$ via an explicit isomorphism $\mathbf{G}\cong \mathbf{G}(X;p_i)_I$ (also called a $\mathbf{G}$\textit{-marking}). The map $\overline{\mathcal{M}}_{\mathbf{G}}^\sim \rightarrow \overline{\mathcal{D}}_{\mathbf{G}}$ above simply forgets this isomorphism. The map factors through the stack quotient by the automorphism group of $\mathbf{G}$, so we get a commutative diagram of Deligne--Mumford stacks:
\begin{center}
\begin{tikzpicture}
\matrix (m) [matrix of math nodes,row sep=2em,column sep=2em]
  {
\mathcal{M}_{\mathbf{G}}^\sim & \mathcal{M}_{\mathbf{G}} & \mathcal{D}_{\mathbf{G}} \\
\overline{\mathcal{M}}_{\mathbf{G}}^\sim & \overline{\mathcal{M}}_{\mathbf{G}} & \overline{\mathcal{D}}_{\mathbf{G}} \\
  };
  \path[-stealth]
(m-1-1) edge (m-1-2)
(m-1-2) edge (m-1-3)
(m-2-1) edge (m-2-2)
(m-2-2) edge (m-2-3)
;
\path[right hook-stealth]
(m-1-1) edge (m-2-1)
(m-1-2) edge (m-2-2)
(m-1-3) edge (m-2-3)
;
\end{tikzpicture}
\end{center}

If $\mathbf{G}=(g,P)$ is the $P$-pointed dual graph with a single vertex of weight $g$ and no edges, then $\overline{\mathcal{D}}_{\mathbf{G}}=\overline{\mathcal{M}}_{g,P}$, $\operatorname{Aut}(\mathbf{G})=1$ and all the horizontal maps are equivalences (in fact identity morphisms). This is not, however, the case in general; we have the following identification.

\begin{proposition}\label{DMK compactification as normalisation}
The map $\overline{\mathcal{M}}_{\mathbf{G}} \rightarrow \overline{\mathcal{D}}_{\mathbf{G}}$ can be identified with the normalisation of $\overline{\mathcal{D}}_{\mathbf{G}}$. In particular, it is an isomorphism over the open substack $\mathcal{D}_{\mathbf{G}}$.
\end{proposition}
\begin{proof}
This is \cite[Chapter XII, Proposition 10.11]{ArbarelloCornalbaGriffiths} and we refer to \S 8 of the same source for details on the normalisation in the setting of Deligne-Mumford stacks. The final claim follows from the fact that $\mathcal{D}_{\mathbf{G}}$ is already normal and can also be seen directly from the proof of \cite[Chapter XII, Proposition 10.11]{ArbarelloCornalbaGriffiths}.
\end{proof}

\begin{definition}\label{DMK compactification associated to dual graph}
We call $\overline{\mathcal{M}}_{\mathbf{G}}$ the \textit{Deligne--Mumford--Knudsen compactification} of $\mathcal{M}_{\mathbf{G}}\cong \mathcal{D}_{\mathbf{G}}$.
\end{definition}

\subsection{The underlying stratified topological stacks}\label{underlying stratified topological stack}

In order to study the stack $\overline{\mathcal{M}}_{g,P}$ from the point of view of stratified homotopy theory, we have to consider the underlying topological stack equipped with its natural stratification. This is the object that we formally introduce in this section. We refer to \cite[specifically \S 4]{OrsnesJansen23} for details.

By a topological $\infty$-stack, we mean an object in the $\infty$-category of space-valued sheaves $\operatorname{Shv}(\mathfrak{T})$ on the site $\mathfrak{T}=\operatorname{LCHaus}^{2nd}$ of locally compact Hausdorff second-countable topological spaces. The stacks under consideration in this paper are all $1$-stacks, that is, $1$-truncated $\infty$-stacks, or more explicitly, a sheaf taking values in ($1$-)groupoids. We will need the full power of working within the $\infty$-category of $\infty$-stacks to do our calculations, however.

Consider the functor
\begin{align*}
(-)^{\operatorname{top}}\colon\operatorname{Aff}^{ft}_{\C}\rightarrow \mathfrak{T},\quad U\rightarrow U^{\operatorname{top}}
\end{align*}
sending an affine scheme of finite type over $\C$ to the topological space underlying its analytification (in other words, we also forget the structure sheaf of holomorphic functions defining the analytification in the category of locally ringed spaces). This functor sends étale morphisms to local homeomorphisms and it preserves fibre product. Hence, we can transfer the \v{C}ech nerve of an atlas of an algebraic stack $X\in \operatorname{Shv}(\operatorname{Aff}_{\C}^{ft})$ into the category of topological ($\infty$-)stacks and consider the corresponding stack given by the colimit; this is the \textit{underlying topological stack} $X^{\operatorname{top}}$. Alternatively, it can be defined more intrinsically as follows (\cite[Proposition 4.9]{OrsnesJansen23}): $X^{\operatorname{top}}$ is the sheafification of the presheaf
\begin{align*}
\mathfrak{T}^{\operatorname{op}}\rightarrow \mathcal{S},\quad S\mapsto \overline{X}(\operatorname{Spec}(C(S,\C))),
\end{align*}
sending a topological space $S\in \mathfrak{T}$ to the value of $\overline{X}$ on the ring of complex continuous functions on $S$, where $\overline{X}$ is the canonical extension of $X$ to the category $\operatorname{Aff}_\C$ of affine schemes not necessarily of finite type.

\begin{definition}
Let $\mathbf{G}$ be a $P$-pointed stable dual graph. The \textit{topological Deligne--Mumford--Knudsen compactification associated to $\mathbf{G}$} is the topological stack $\overline{\mathcal{M}}^{\operatorname{top}}_{\mathbf{G}}$ underlying the algebraic Deligne--Mumford stack $\overline{\mathcal{M}}_{\mathbf{G}}$.
\end{definition}

The topological Deligne--Mumford--Knudsen compactifications are naturally stratified over the poset of (isomorphism classes of) stable dual graphs. We introduce this stratification now. Essentially, the stratification map $\overline{\mathcal{M}}^{\operatorname{top}}_{g,P}\rightarrow \mathcal{G}_{g,P}$ is given by sending a geometric point, that is, a single stable nodal curve, to its associated dual graph. In order to define it formally, however, we will make use of the coarse moduli space.

Let $\overline{M}_{g,P}$ denote the (algebraic) coarse moduli space of $\overline{\mathcal{M}}_{g,P}$ (see \cite[Chapter XII]{ArbarelloCornalbaGriffiths}) and let $\overline{M}_{g,P}^{\operatorname{top}}$ denote the topological space underlying its analytification. This is the coarse moduli space of the topological stack $\overline{\mathcal{M}}^{\operatorname{top}}_{g,P}$ (\cite[Proposition 4.12]{OrsnesJansen23}) and as a set, it consists of all isomorphism classes of stable $P$-pointed genus $g$ curves over $\C$. We define a map
\begin{align*}
\overline{M}_{g,P}^{\operatorname{top}}\rightarrow \mathcal{G}_{g,P},\quad [(X;p_i)]\mapsto [\mathbf{G}(X;p_i)].
\end{align*}

This defines a stratification by the following lemma.

\begin{lemma}
The map $\overline{M}_{g,P}^{\operatorname{top}}\rightarrow \mathcal{G}_{g,P}$, $[(X;p_i)]\mapsto [\mathbf{G}(X;p_i)]$ is continuous.
\end{lemma}
\begin{proof}
This can be seen by exploiting the construction of the moduli space as an analytic space using Kuranishi families as in \cite[Chapter XII, \S 2]{ArbarelloCornalbaGriffiths} and analysing the inherited stratification of the base spaces of the Kuranishi families constructed in \cite[Chapter XI, Theorem 3.17]{ArbarelloCornalbaGriffiths}.
\end{proof}

\begin{definition}\label{stratified DMK base case}
Equip the topological stack $\overline{\mathcal{M}}^{\operatorname{top}}_{g,P}$ with the stratification map
\begin{align*}
\pi_{g,P}\colon \overline{\mathcal{M}}^{\operatorname{top}}_{g,P}\longrightarrow \overline{M}_{g,P}^{\operatorname{top}}\longrightarrow \mathcal{G}_{g,P}
\end{align*}
where the former map is the canonical one from the stack to its coarse moduli space and the latter is the stratification map defined above. This is the \textit{stratified topological Deligne--Mumford--Knudsen compactification of type} $(g,P)$.
\end{definition}

\begin{remark}
This stratification of the Deligne--Mumford--Knudsen compactification is of course well-known: the boundary $\partial \overline{\mathcal{M}}_{g,P}=\overline{\mathcal{M}}_{g,P}\smallsetminus \mathcal{M}_{g,P}$ is a normal crossings divisor and as such carries a natural stratification. This stratification together with the open substack $\mathcal{M}_{g,P}$ is exactly what we recover from the above stratification map. See \cite[Chapter X, \S 2 and \S 10]{ArbarelloCornalbaGriffiths}.
\end{remark}

We also stratify the more general Deligne--Mumford--Knudsen compactifications associated to a stable dual graph. Let $\mathbf{G}=(G,\omega,m)$ be a $P$-pointed stable dual graph, write $g_v=\omega(v)$ and let $P_v$ denote the set of half-edges incident to $v$ for each vertex $v$ of $G$. Consider the poset
\begin{align*}
\mathcal{G}_{\mathbf{G}}^\sim:=\prod_{v\in V}\mathcal{G}_{g_v,P_v}
\end{align*}
with the obvious action of $\operatorname{Aut}(\mathbf{G})$ on the index set $V$ and the labels $P_v$. Note that since $\operatorname{Aut}(\mathbf{G})$ is finite, it follows from \cite[Lemma 2.35]{ClausenOrsnesJansen} that the quotient set by $\operatorname{Aut}(\mathbf{G})$ has a natural poset structure identifying it with the corresponding quotient in the category of posets: for a pair of equivalence classes $\bar{x}$, $\bar{y}$, we have $\bar{x}\leq \bar{y}$ if and only if there is a $g\in \operatorname{Aut}(\mathbf{G})$ such that $gx\leq y$ in $\mathcal{G}_{\mathbf{G}}^\sim$. We denote this quotient by
\begin{align*}
\mathcal{G}_{\mathbf{G}}:=\operatorname{Aut}(\mathbf{G})\backslash\,\mathcal{G}_{\mathbf{G}}^\sim.
\end{align*}

We stratify the product stack using the stratification maps defined in \Cref{stratified DMK base case}:
\begin{align*}
\pi_{\mathbf{G}}^\sim:=\prod_{v\in V}\pi_{g_v,P_v}\colon \overline{\mathcal{M}}^{\sim,\operatorname{top}}_{\mathbf{G}}\longrightarrow \mathcal{G}_{\mathbf{G}}^\sim
\end{align*}
and note that this descends to define a stratification of the quotient stack by $\operatorname{Aut}(\mathbf{G})$:
\begin{align*}
\pi_{\mathbf{G}}\colon \overline{\mathcal{M}}^{\operatorname{top}}_{\mathbf{G}}\longrightarrow \mathcal{G}_{\mathbf{G}}
\end{align*}
The resulting stratified topological stack is called the \textit{stratified topological Deligne--Mumford--Knudsen compactification of type $\mathbf{G}$}.

\medskip

Let us finish off this section by relating the various stratified topological stacks introduced. There is a map of posets
\begin{align*}
\mathcal{G}_{\mathbf{G}}^\sim\rightarrow (\mathcal{G}_{g,P})_{\leq [\mathbf{G}]}
\end{align*}
given by clutching the graphs together along the edges of $\mathbf{G}$: a collection $[\mathbf{H}_v]_{v\in V}$ defines a stable $P$-pointed dual graph $\mathbf{H}$ whose set of vertices and half-edges is given by the union of those in the $\mathbf{H}_v$, but now the legs $l\in P_v$ and $l'\in P_w$ in $\mathbf{H}_v$, respectively $\mathbf{H}_w$, make up an edge in $\mathbf{H}$ if the corresponding half-edges make up an edge in $\mathbf{G}$. The weight function and the $P$-marking are directly inherited from the $\mathbf{H}_v$. This map factors through the poset $\mathcal{G}_{\mathbf{G}}$ and we see that we have a sequence of morphisms of stratified topological stacks for any given stable $P$-pointed genus $g$ dual graph $\mathbf{G}$:

\begin{center}
\begin{tikzpicture}
\matrix (m) [matrix of math nodes,row sep=2em,column sep=2em]
  {
\overline{\mathcal{M}}_{\mathbf{G}}^{\sim,\operatorname{top}} & \overline{\mathcal{M}}_{\mathbf{G}}^{\operatorname{top}} & \overline{\mathcal{D}}_{\mathbf{G}}^{\operatorname{top}} & \overline{\mathcal{M}}_{g,P}^{\operatorname{top}} \\
\mathcal{G}_{\mathbf{G}}^\sim & \mathcal{G}_{\mathbf{G}} & (\mathcal{G}_{g,P})_{\leq [\mathbf{G}]} & \mathcal{G}_{g,P} \\
  };
  \path[-stealth]
(m-1-1) edge (m-1-2) edge (m-2-1)
(m-1-2) edge (m-1-3) edge (m-2-2)
(m-1-3) edge (m-2-3)
(m-1-4) edge (m-2-4)
(m-2-1) edge (m-2-2)
(m-2-2) edge (m-2-3)
;
\path[right hook-stealth]
(m-1-3) edge (m-1-4)
(m-2-3) edge (m-2-4)
;
\end{tikzpicture}
\end{center}

From now on, we implicitly view these stacks as stratified topological stacks and we will for the most part omit the stratifying posets and stratification maps.

\begin{remark}
Note that we also have various morphisms of stratified topological stacks given by for example the clutching morphisms and the specialisation relation between dual graphs, but we will not need this here. This would be an interesting route to pursue, however, in trying to combine our calculations with the fact that the stacks $\overline{\mathcal{M}}_{g,P}$ make up a modular stack.
\end{remark}

\section{Stable nodal surfaces and associated objects of interest}

In this section, we introduce what we will call the Charney--Lee categories; this is a variant of the category of stable curves introduced in \cite{CharneyLee84}. We want to work with a different and more algebraic incarnation of this category that fits better into the setting of stratified homotopy theory. In order to define these, we introduce stable nodal surfaces, their mapping class groups and the poset of admissible curve systems. We also provide an alternative description of the stratification poset of stable dual graphs.

\subsection{Stable nodal surfaces and their mapping class groups}

We recall the definition of stable nodal surfaces of finite type and their mapping class groups. By a surface we mean an oriented smooth $2$-manifold.

\medskip

Let $P$ a finite set. By a $P$-pointed surface of \textit{finite type}, we mean a compact connected surface $S$ equipped with an embedding $P\hookrightarrow S$. We refer to the image of $P$ as the \textit{marked points} and denote these by $p_i$, $i\in P$, or simply by $P\subseteq S$. We denote the $P$-pointed surface of \textit{finite type} by $(S,P)$ leaving the embedding implicit. The homeomorphism type of the punctured surface $S\smallsetminus P$ is determined by the topological genus $g$ of $S$ and the number $|P|$ of marked points (\cite[\S 1.1.1]{FarbMargalit}). A $P$-pointed surface of finite type $(S,P)$ is \textit{stable} if it satisfies one of the following equivalent conditions:
\begin{enumerate}
\item the Euler characteristic of the complement of the marked points is negative, $\chi(S\smallsetminus P)<0$;
\item $2g+2-n>0$ where $g$ is the genus of $S$ and $n=|P|$;
\item $S\smallsetminus P$ admits a hyperbolic metric (\cite[Theorem 1.2]{FarbMargalit}).
\end{enumerate}

In order to make a clean inductive argument for our final calculation it will be beneficial for us to work more generally with stable \textit{nodal} surfaces. This is purely a technicality and a nodal surface is. as one would expect, just the result of glueing a finite collection of stable surfaces together along pairs of marked points, thus introducing \textit{nodes}.

\begin{definition}
By a \textit{topological nodal surface of finite type}, we mean a topological space $S$ such that every point has a either a fundamental system of neighbourhoods homeomorphic to the unit disk in $\R^2$ or a fundamental system of neighbourhoods homeomorphic to the double cone in $\R^3$, i.e.~the result of glueing two unit disks together at the origin. The points with the latter type of neighbourhoods are called \textit{nodes}. For a finite set $P$, a \textit{$P$-pointed topological nodal surface of finite type} is a compact connected topological nodal surface $S$ together with an embedding $P\hookrightarrow S\smallsetminus \{\text{nodes}\}$.
\end{definition}

The following construction is just a topological model for normalisation of the nodes.

\begin{construction}\label{parts of nodal surface}
Let $(S,P)$ be a $P$-pointed topological nodal surface. Write $V_S:=\pi_0(S\smallsetminus \{\text{nodes}\})$ for the connected components of the complement of the nodes and denote for each $v\in V_S$ the corresponding connected component by $S_v^\circ\subset S$.

Let $P_v$ denote the set of marked points and nodes belonging to the closure of $S_v^\circ$ in $S$ where we count twice the nodes for which any small enough neighbourhood in $S_v^\circ$ has two components, i.e.~we count twice the nodes of $S$ connecting $S_v^\circ$ with itself. Note that $\coprod_{v\in V_S} P_v=P\sqcup N\sqcup N$, where $N$ is the set of nodes of $S$.

Now, for each $v\in V_S$, we let $(S_v,P_v)$ be the (non-singular) $P_v$-pointed topological surface obtained by adding the nodal points $p_n\in P_v\smallsetminus P$ to $S_v^\circ$ in the obvious manner: each connected component of a small enough neighbourhood of a node is a once-punctured disk and we simply replace it by a disk and label the ``puncture'' by the corresponding nodal point $p_n$.

We call the $(S_v,P_v)$ the \textit{parts} of $(S,P)$.
\end{construction}

\begin{definition}
A \textit{$P$-pointed nodal surface of finite type} is a $P$-pointed topological nodal surface of finite type $(S,P)$ together with a smooth structure on each part $(S_v,P_v)$, $v\in V_S$ (such that they become $P$-pointed surfaces of finite type). We say that $(S,P)$ is \textit{stable} if each part is stable. A \textit{diffeomorphism} of $P$-pointed nodal surfaces of finite type is a homeomorphism which restricts to a diffeomorphism on the parts.
\end{definition}

\begin{remark}
We require a smooth structure on the parts simply to have some control over how automorphisms behave around the nodes. A homeomorphism of a stable nodal surface $(S,P)$ above will always be isotopic to one which is smooth on the parts, and thus we avoid dealing with for example infinite twists around a node.

We will often write \textit{stable nodal surface} instead of \textit{stable nodal $P$-pointed surface of finite type}. We will say that a stable nodal surface is \textit{non-singular} if we want to stress that the set of nodes is empty. To ease notation, we may omit the marked points from our notation and simply refer to a stable nodal surface $(S,P)$ by $S$.
\end{remark}

The \textit{mapping class group} of a stable nodal surface $(S,P)$ is the group
\begin{align*}
\Gamma_S=\Gamma_{(S,P)}:=\pi_0(\operatorname{Diffeo}^+_P(S)) \cong\pi_0(\operatorname{Homeo}^+_P(S)),
\end{align*} 
of isotopy classes of orientation preserving diffeomorphisms of $S$ fixing the set of marked points pointwise; equivalently we may consider isotopy classes of homeomorphisms. Note that such homeomorphisms must fix the set of nodes, though not necessarily pointwise. We also consider the following variant, which we call the \textit{pure mapping class group}: it is the subgroup $\widetilde{\Gamma}_S\subset \Gamma_S$ of mapping classes represented by homeomorphisms which preserve the connected components of $S\smallsetminus \{\text{nodes}\}$ and fix the set of nodes pointwise.

\begin{remark}
We remark that in the literature the term mapping class group often refers to the mapping class group allowing permutations of the marked points and the term pure mapping class group requires the marked points to be fixed (\cite[Chapter 4]{FarbMargalit}). Since we label our marked points by a fixed finite set $P$ instead of thinking of them as an unlabelled set of punctures, it is natural for us to require these marked points to be fixed. Moreover, as we have to deal with nodal points that become marked points upon considering the parts of the nodal surface, it is also natural for us to consider the more restrictive version fixing the nodes and the parts. It also fits into the setup of \cite{ArbarelloCornalbaGriffiths}. We hope that our terminology will not be a source of confusion.
\end{remark}

\begin{observation}\label{product decomposition of pure mapping class group}
If $(S,P)$ is non-singular, then $\widetilde{\Gamma}_S=\Gamma_S$, and in general the pure mapping class group decomposes as a product of the mapping class groups of the parts
\begin{equation*}
\widetilde{\Gamma}_S\cong\prod_{v \in V_S}\Gamma_{S_v}.\qedhere
\end{equation*}
\end{observation}

Note that the construction of a dual graph associated to a stable curve only depends on the underlying stable nodal surface.

\begin{observation/construction}[Dual graph associated to a stable nodal surface of finite type]
Given a stable nodal surface $(S,P)$, we can proceed as in \Cref{dual graph associated to stable curve} and associate to it a $P$-pointed dual graph $\mathbf{G}(S,P)$ with vertices given by the parts of $(S,P)$, edges given by the nodes and legs given by the marked points (the whole construction goes through by replacing every instance of connected component with part). The weight is now given by the topological genus and since each part is a non-singular stable surface, the dual graph $\mathbf{G}(S,P)$ is stable.
\end{observation/construction}

We generalise the definition of genus to nodal surfaces using the construction of dual graphs.
\begin{definition}
The \textit{genus} of a nodal surface $S$ is defined as the genus of its associated dual graph $\mathbf{G}(S)$.
\end{definition}

Dual graphs enable us to identify explicitly the difference between the pure and non-pure mapping class groups:

\begin{lemma}\label{pure vs non-pure mapping class groups}
For any stable nodal surface $(S,P)$, there is a short exact sequence
\begin{align*}
1\longrightarrow \widetilde{\Gamma}_S\longrightarrow\Gamma_S\longrightarrow \operatorname{Aut}(\mathbf{G}_S)\longrightarrow 1
\end{align*}
where $\mathbf{G}_S=\mathbf{G}(S,P)$ is the dual graph associated to $(S,P)$.
\end{lemma}
\begin{proof}
This is immediate.
\end{proof}

\subsection{Curve posets}

We recall admissible curve systems, introduce the curve poset and compare it with the poset of dual graphs.

By an \textit{admissible curve system} on a stable nodal surface $(S,P)$, we mean a collection $C$ of disjoint simple closed curves in $S\smallsetminus (\{\text{nodes}\}\cup P)$ such that no single curve bounds a disk nor a once-punctured disk, and no two curves bound a cylinder.

\begin{definition}
Let $(S,P)$ be a stable nodal surface. We define the \textit{curve poset} of $(S,P)$ to be the poset $\mathfrak{C}_S=\mathfrak{C}_{(S,P)}$ of isotopy classes of admissible curve systems on $(S,P)$ with partial order given by reverse inclusion: $\sigma\leq \tau$ if and only if the isotopy classes of disjoint simple closed curves in $\tau$ are a subset of those in $\sigma$. We include the empty curve system $\emptyset$ as an element in $\mathfrak{C}_S$. 
\end{definition}

\begin{remark}
The realisation of $(\mathfrak{C}_S\smallsetminus \emptyset)^{op}$ is the classical \textit{curve complex} associated to $(S,P)$, that is, the simplicial complex with vertices the isotopy classes of simple closed curves not bounding a disk nor a once-punctured disk, and where distinct vertices $c_0,\ldots,c_n$ span an $n$-simplex if the representing curves can be homotoped to be pairwise disjoint. The curve poset $\mathfrak{C}_S$ is then the opposite of what is sometimes called the \textit{augmented curve complex} (e.g.~\cite{ChenLooijenga}). We choose to work with reverse inclusion as our partial order because this better aligns with the setting of stratified homotopy theory. Comparing with the situation of the reductive Borel--Serre compactification, the classical curve complex corresponds to the Tits building, whereas the curve poset corresponds to the poset of parabolic subgroups.
\end{remark}

\begin{construction}
Given an admissible curve system $C$ on a stable nodal surface $(S,P)$, we can consider the stable nodal surface $(S/C,P)$ consisting of the topological space $S/C$ given by collapsing each simple closed curve of $C$ to a point together with the embedding $P\hookrightarrow S\twoheadrightarrow S/C$. The smooth structure on the parts of $(S,P)$ descends to define a smooth structure on the parts of $(S/C,P)$.

Note that the genus of $S/C$ is equal to that of $S$ as the dual graph of $S/C$ is a specialisation of that of $S$.
\end{construction}

We make a few simple observations.

\begin{observation}\label{observations about curve posets}
Let $(S,P)$ be a stable nodal surface with parts $(S_v,P_v)$, $v\in V_S$.
\begin{enumerate}
\item The curve poset decomposes according to the parts
\begin{align*}
\mathfrak{C}_S\cong \prod_{v\in V_S} \mathfrak{C}_{S_v}.
\end{align*}
\item If $\sigma\in \mathfrak{C}_{S}$ is represented by an admissible curve system $C$, then we have an isomorphism of posets
\begin{align*}
\mathfrak{C}_{S/C}\xrightarrow{\cong}(\mathfrak{C}_{S})_{\leq \sigma},\quad\tau\mapsto \tau \cup \sigma
\end{align*}
where we identify $\tau$ with its preimage under the quotient map $S\twoheadrightarrow S/C$.\qedhere
\end{enumerate}
\end{observation}

The mapping class group $\Gamma_{S}$ acts on the curve poset $\mathfrak{C}_{S}$ in the obvious way:
\begin{align*}
\overline{\gamma}.\sigma=[\gamma(C)],\quad\quad\text{for}\quad \overline{\gamma}\in \Gamma_S,\ \sigma=[C]\in \mathfrak{C}_S. 
\end{align*}
By \cite[Lemma 2.35]{ClausenOrsnesJansen}, the quotient set $\Gamma_S\backslash \mathfrak{C}_S$ has a natural structure of a poset identifying with the quotient in the category of posets: for a pair of equivalence classes $\bar{\sigma}$, $\bar{\tau}$, we have $\bar{\sigma}\leq \bar{\tau}$ if there is a $\overline{\gamma}\in \Gamma_S$ such that $\overline{\gamma}.\sigma\leq \tau$ in $\mathfrak{C}_S$.

We want to compare the poset $\Gamma_S\backslash \mathfrak{C}_S$ with the poset of stable dual graphs introduced in the previous section. Let $(S,P)$ be a stable nodal surface of genus $g$ with associated dual graph $\mathbf{G}=\mathbf{G}(S,P)$ and parts $(S_v,P_v)$, $v\in V_S$. We have a map of posets
\begin{align*}
\mathfrak{C}_{S}\longrightarrow \mathcal{G}_{\mathbf{G}}^\sim,\quad (\sigma_v)_{v\in V_S}\mapsto [\mathbf{G}(\sigma_v)]_{v\in V_S},
\end{align*}
where (by a slight abuse of notation) $[\mathbf{G}(\sigma)]$ denotes the isomorphism class of the dual graph associated to the nodal surface $(S/C,P)$ for some representative $C$ of $\sigma$ and where we are using part 1 of \Cref{observations about curve posets} to view the curve poset as a product. This map factors through the quotient $\widetilde{\Gamma}_S\backslash \mathfrak{C}_S$, and further quotienting by $\operatorname{Aut}(\mathbf{G})$ defines a map of posets $\Gamma_{S}\backslash \mathfrak{C}_{S}\rightarrow\mathcal{G}_{\mathbf{G}}$.

\begin{proposition}\label{iso of posets}
Let $(S,P)$ be a stable nodal surface with associated dual graph $\mathbf{G}=\mathbf{G}(S,P)$. The map $\Gamma_{S}\backslash \mathfrak{C}_{S}\rightarrow\mathcal{G}_{\mathbf{G}}$ defined above is an isomorphism of posets.
\end{proposition}
\begin{proof}
It suffices to prove this for a non-singular stable surface $(S,P)$ of genus $g$. The general case then follows from the identification $\operatorname{Aut}(\mathbf{G})\cong \widetilde{\Gamma}_S\backslash \Gamma_S$ (\Cref{pure vs non-pure mapping class groups}) and the product decompositions of $\widetilde{\Gamma}_S$ and $\mathfrak{C}_S$ (\Cref{product decomposition of pure mapping class group} and \Cref{observations about curve posets}).

To see that the map is surjective, let $\mathbf{H}\in \mathcal{G}_{g,P}$ and for each vertex $v$ of $\mathbf{H}$, let $g_v$ denote the weight of $v$ and $P_v$ the set of half-edges incident to $v$. For every $v$, let $(S_v,P_v)$ be a non-singular stable surface of genus $g_v$. We glue (or clutch) these together according to the graph $\mathbf{H}$: identify the points $p_i\in S_v$ and $p_j\in S_w$ if the corresponding half-edges $i$ and $j$ make up an edge in $\mathbf{H}$. Now, we can resolve the nodes of the resulting stable nodal surface $(S_{\mathbf{H}},P)$ by replacing a sufficiently small conical neighbourhood of each node by an annulus. The resulting surface $(\widetilde{S}_{\mathbf{H}},P)$ is a non-singular stable surface of genus $g$ and as such it is homeomorphic to our original surface $(S,P)$. Hence, there is a map
\begin{align*}
(S,P)\xrightarrow{\ \cong\ } (\widetilde{S}_{\mathbf{H}},Z)\longrightarrow (S_\mathbf{H},Z)
\end{align*}
and the preimage of the nodes of $S_\mathbf{H}$ is an admissible system of curves on $S$ whose isotopy class is mapped to $\mathbf{H}$ by the map $\mathfrak{C}_S\rightarrow \mathcal{G}_{\mathbf{G}}$ of posets (since the associated dual graph only depends on the underlying homeomorphism type of the nodal surface and not on the smooth structure on the parts).

To see that the map is injective, let $\sigma=[C]$ and $\tau=[D]$ in $\Gamma_S\backslash \mathfrak{C}_S$ and note that an isomorphism of graphs $\mathbf{G}(\sigma)\xrightarrow{\cong} \mathbf{G}(\tau)$ will give rise to a homeomorphism of the corresponding nodal surfaces, $f\colon (S/C,P)\rightarrow (S/D,P)$. Now, we claim that such a homeomorphism lifts to a homeomorphism $\hat{f}$ of $S$ mapping $C$ to $D$ and such that the diagram below commutes up to isotopy where $\rho_C$ and $\rho_D$ denote the maps collapsing the admissible curve systems $C$, respectively $D$, to node points. Indeed, we may choose a representative in the isotopy class of $f$ which is a diffeomorphism on the parts and this can be lifted locally in a conical neighbourhood of a node $n\in S/C$ to a homeomorphism between cylindrical neighbourhoods of the corresponding closed simple curves, $(f\circ \rho_C)^{-1}(n)$ and $\rho_D^{-1}(n)$ in $S$. Outside of these local neighbourhoods $f$ defines a homeomorphism $S\smallsetminus (f\circ\rho_C)^{-1}(\{\text{nodes}\})\xrightarrow{\ \cong\ } S\smallsetminus D$.
\begin{center}
\begin{tikzpicture}
\matrix (m) [matrix of math nodes,row sep=2em,column sep=2em]
  {
S & S \\
S/C & S/D \\
  };
  \path[->,dashed]
  (m-1-1) edge node[above]{$\hat{f}$} (m-1-2)
  ;
  \path[-stealth]
(m-1-1) edge node[left]{$\rho_C$} (m-2-1)
(m-2-1) edge node[below]{$f$} (m-2-2)
(m-1-2) edge node[right]{$\rho_D$} (m-2-2)
;
\end{tikzpicture}
\end{center}
This completes the proof.
\end{proof}

We also proved the following.

\begin{corollary}
Let $(S,P)$ be a stable nodal surface with associated dual graph $\mathbf{G}=\mathbf{G}(S,P)$. The map $\widetilde{\Gamma}_{S}\backslash \mathfrak{C}_{S}\rightarrow\mathcal{G}_{\mathbf{G}}^\sim$ is an isomorphism of posets.
\end{corollary}

\begin{remark}
We could exhibit the stratification of the coarse moduli space using the poset $\Gamma_S\backslash \mathfrak{C}_S$. Indeed, the moduli space is homeomorphic to the quotient of the so-called \textit{augmented Teichmüller space} by an action of the mapping class group (\cite[Chapter 2]{Abikoff80}) and the augmented Teichmüller space is naturally stratified over the curve poset $\mathfrak{C}_S$. We have, however, actively chosen to avoid the augmented Teichmüller space as it does not fit into the analytic (or algebraic) setting nor is it a well-behaved stratified topological space. The problem is essentially that it is not locally compact. One could work around this as is done in \cite{OrsnesJansen}, but for a more satisfying setup, we choose an alternative route. We also find that it is also more appealing to define the stratification only in terms of the information given by the moduli space without having to pick a stable nodal surface $(S,P)$. In order to calculate the stratified homotopy type, we will need to make a choice of this kind, however --- essentially, it amounts to picking a basepoint in the open substack $\mathcal{M}_{\mathbf{G}}$ of smooth curves.
\end{remark}

\subsection{The Charney--Lee categories}\label{CharneyLee categories}

In this section, we associate a category $\operatorname{CL}_{S}$ to any stable nodal surface $(S,P)$.

\begin{notation}
Let $(S,P)$ be a stable nodal surface with mapping class group $\Gamma_{S}$ and curve poset $\mathfrak{C}_S$. For any given isotopy class of a single simple closed curve $c\in \mathfrak{C}_{S}$ denote by $t_c\in \Gamma_{S}$ the \textit{Dehn twist} about $c$ (see e.g.~\cite[Chapter 3]{FarbMargalit}). For any $\sigma\in \mathfrak{C}_{S}$, let $\Gamma_{S}(\sigma)\subset \Gamma_{S}$ denote the subgroup of mapping classes preserving $\sigma$, and let $\Delta_{\sigma}\subset \Gamma_{S}(\sigma)$ denote the free abelian subgroup generated by the Dehn twists about the curves making up $\sigma$.
\end{notation}

\begin{construction}\label{The category}
Let $(S,P)$ be a stable nodal surface and consider the action of the mapping class group $\Gamma_{S}$ on the curve poset $\mathfrak{C}_{S}$, and moreover the action by conjugation on the collection of subgroups $\Delta_\sigma\subset \Gamma_{S}$, $\sigma\in \mathfrak{C}_{S}$. The following conditions are satisfied:
\begin{enumerate}[label=(\roman*)]
\item $\Delta_\sigma\subset \Delta_\tau$ for all $\tau\leq \sigma$;
\item $\gamma\Delta_\sigma\gamma^{-1}=\Delta_{\gamma.\sigma}$ for all $\sigma\in \mathfrak{C}_{S}$, $\gamma\in \Gamma_{S}$ (\cite[Fact 3.7]{FarbMargalit}).
\end{enumerate}
Consider the category $\operatorname{CL}_{S}=\operatorname{CL}_{(S,P)}$ with objects the elements of $\mathfrak{C}_{S}$ and hom-sets
\begin{align*}
\operatorname{Hom}(\sigma,\tau)=\{\gamma\in \Gamma_{S}\mid \gamma.\sigma\leq \tau\}/\Delta_\sigma,
\end{align*}
where $\Delta_\sigma$ acts by right multiplication, and with composition given by multiplication of representatives in $\Gamma_{S}$. Properties (i) and (ii) imply that this is well-defined (see also \cite[\S 6]{OrsnesJansen}).

We call this the \textit{Charney--Lee category} associated to $(S,P)$. If $S$ is non-singular of genus $g$, we will often write $\operatorname{CL}_{g,P}=\operatorname{CL}_S$.
\end{construction}

The following observation, which is also made in \cite[\S 2.9]{CharneyLee84}, already reveals a little bit about the category $\operatorname{CL}_S$: the automorphism group of an object $\sigma=[C]$ identifies with the mapping class group of the stable nodal surface $S/C$.

\begin{lemma}\label{SES of groups quotient by Dehn twists}
Let $(S,P)$ be a stable nodal surface and let $\sigma=[C]\in \mathfrak{C}_S$. There is a short exact sequence of groups
\begin{align*}
1\rightarrow \Delta_{\sigma}\rightarrow \Gamma_{S}(\sigma)\rightarrow \Gamma_{S/C}\rightarrow 1.
\end{align*}
\end{lemma}
\begin{proof}
Surjectivity follows from the following observation made in the proof of \Cref{iso of posets}: a homeomorphism $S/C\rightarrow S/C$ lifts (up to isotopy) to a homeomorphism $S\rightarrow S$ mapping the curve system $C$ to itself.

Since the mapping class group of an annulus is generated by the Dehn twist around the annulus (\cite[Proposition 2.4]{FarbMargalit}), any two such lifts will differ by an element in $\Delta_\sigma$, which proves exactness at the middle term.
\end{proof}

\begin{remark}\label{category of stable curves remark}
We could define the category above using different subgroups of the mapping class group. We make a few remarks and observations about this here.
\begin{enumerate}
\item Let $\operatorname{CL}_S^\sim$ denote the category defined as above but using the pure mapping class group $\widetilde{\Gamma}_S$ instead. This decomposes according to the parts:
\begin{align*}
\operatorname{CL}_S^\sim\simeq \prod_{v\in V_S} \operatorname{CL}_{S_v}.
\end{align*}
\item Let $\sigma=[C]\in \mathfrak{C}_S$. Consider the category defined as above for the action of the group $\Gamma_S(\sigma)$ on the poset $(\mathfrak{C}_S)_{\leq \sigma}$. By the lemma above, this recovers the Charney--Lee category $\operatorname{CL}_{S/C}$ associated to the quotient $(S/C,P)$. Note, however, that this does \textit{not} in general agree with the full subcategory of $\operatorname{CL}_S$ spanned by the objects $(\mathfrak{C}_S)_{\leq \sigma}$. This is in fact one of the main technical differences between the situation at hand, namely the Deligne--Mumford--Knudsen compactification, and the reductive Borel--Serre compactification as dealt with in \cite{ClausenOrsnesJansen}. See also \Cref{key technical difference RBS versus DMK}.
\item In view of the previous observation, we could have avoided considering nodal surfaces by instead restricting our attention to the subgroups $\Gamma_S(\sigma)$ for non-singular $S$. We feel that it makes for a cleaner inductive argument to consider the nodal surfaces.\qedhere
\end{enumerate}
\end{remark}

We can identify the Charney--Lee category as an orbit category --- this is analogous to the situation in \cite[\S 6.2]{OrsnesJansen}.

\begin{observation}\label{observations about our 1-categories of stable curves}
Let $(S,P)$ be a stable nodal surface with mapping class group $\Gamma=\Gamma_{S}$, and let $\mathcal{O}(\Gamma)$ denote the orbit category of $\Gamma$. There is a fully faithful embedding $\operatorname{CL}_{S}^{op}\rightarrow \mathcal{O}(\Gamma)$ given by
\begin{align*}
\sigma \mapsto \Gamma/\Delta_\sigma\quad\text{ and }\quad([\gamma]\colon \sigma\rightarrow \tau)\mapsto( [\gamma]\colon\Gamma/\Delta_\tau\rightarrow  \Gamma/\Delta_\sigma),
\end{align*}
where we use that $\operatorname{Hom}_\Gamma(\Gamma/H, \Gamma/K)\xrightarrow{\ \cong\ } \{\gamma\in \Gamma\mid \gamma^{-1}H\gamma\leq K\}/K$ by sending a $\Gamma$-equivariant map to its value on the identity coset. The essential image of this functor is the full subcategory $\mathcal{O}_\Delta(\Gamma)\subset \mathcal{O}(\Gamma)$ of orbits whose isotropy group is a free abelian group generated by a collection of disjoint Dehn twists.
\end{observation}

It follows more or less directly from this observation that the Charney--Lee category $\operatorname{CL}_{S}$ associated to a non-singular stable $P$-pointed surface of genus $g$ is equivalent to (the opposite of) the category of stable curves introduced by Charney--Lee and Ebert--Giansiracusa (\cite{CharneyLee84}, \cite{EbertGiansiracusa}). Let us briefly recall this category although it will not be used explicitly in this paper.

\begin{definition}\label{category of stable curves}
Let $g\geq 0$ and $P$ a finite set such that $2g-2+|P|>0$. The \textit{category of stable ($P$-pointed) curves} $\mathfrak{SC}_{g,P}$ is the category with objects the stable $P$-pointed nodal surfaces of genus $g$ and morphisms given by isotopy classes of deformations.

Here a \textit{deformation} is a surjective continuous map $f\colon (S,P)\rightarrow (S',P)$ such that
\begin{itemize}
\item $f$ preserves the marked points;
\item the preimage of a node in $S'$ is either a node in $S$ or a simple closed curve in $S\smallsetminus (\{\text{nodes}\}\cup P)$ not bounding a disk nor a once-puntured disk. In other words, $f^{-1}(\text{nodes})=\text{nodes}\cup C$ for some admissible curve system $C$ on $S$;
\item on the complement of $C$ as above, $f$ restricts to an orientation preserving homeomorphism
\begin{equation*}
S\smallsetminus (\{\text{nodes}\}\cup C)\xrightarrow{\cong}S'\smallsetminus \{\text{nodes}\}.\qedhere
\end{equation*}
\end{itemize}
\end{definition}

We remark here that Ebert--Giansiracusa call this category the Charney--Lee category and denote it by $\mathcal{CL}_{g,n}$ in \cite{EbertGiansiracusa}. In view of the following identification, \textit{our} Charney--Lee category identifies with the opposite of \textit{theirs}; we hope this will not be a source of confusion. The identification is a simple generalisation of the short exact sequence of \Cref{SES of groups quotient by Dehn twists}.

\begin{proposition}\label{category of stable curves vs Charney-Lee category}
If $(S,P)$ is a non-singular stable $P$-pointed surface of genus $g$, then there is an equivalence $\operatorname{CL}_{S}\simeq\mathfrak{SC}_{g,P}^{op}$.
\end{proposition}
\begin{proof}
This is a consequence of \Cref{observations about our 1-categories of stable curves} above and the analogous observation made in \cite{EbertGiansiracusa}. It can also be proved as follows: fixing a hyperbolic metric on $(S,P)$, we can define a functor
\begin{align*}
\Phi\colon\operatorname{CL}_{S}^{op}\rightarrow\mathfrak{SC}_{g,P}
\end{align*}
sending $\sigma\in \mathfrak{C}_S$ to the stable nodal surface $(S/C_{\sigma}, P)$ obtained by collapsing the unique admissible curve system $C_\sigma$ representing $\sigma$ by geodesics (\cite[\S 1.2]{FarbMargalit}). On morphisms, we send $[\gamma]\colon \sigma\rightarrow \tau$ to the deformation given by the lower horizontal map in the diagram below where $\rho_\tau$ and $\rho_\sigma$ denote the maps collapsing the geodesic representatives of $\tau$, respectively $\sigma$, to node points.
\begin{center}
\begin{tikzpicture}
\matrix (m) [matrix of math nodes,row sep=2em,column sep=2em]
  {
S & S \\
S/C_\tau & S/C_\sigma \\
  };
  \path[-stealth]
(m-1-1) edge node[above]{$\gamma^{-1}$} (m-1-2) edge node[left]{$\rho_\tau$} (m-2-1)
(m-2-1) edge node[below]{$\Phi([\gamma])$} (m-2-2)
(m-1-2) edge node[right]{$\rho_\sigma$} (m-2-2)
;
\end{tikzpicture}
\end{center}
Clearly, $\Phi$ is essentially surjective. It is fully faithful by the observations made in the proof of \Cref{SES of groups quotient by Dehn twists}.
\end{proof}

\section{The Harvey bordification of Teichmüller space}

We recollect the definition of Teichmüller space and the bordification of it as a manifold with corners, also known as the Harvey bordification (following \cite{Ivanov}). The Harvey bordification will play an essential role in our calculational strategy, allowing us to construct a cover of $\overline{\mathcal{M}}_{\mathbf{G}}^{\operatorname{top}}$ by ``simpler'' stratified spaces. We will make various simple observations and identifications needed for our calculation.

\subsection{Teichmüller space}\label{Teichmuller space}

We review the \textit{Teichmüller space} $\mathcal{T}_{S}=\mathcal{T}_{(S,P)}$ associated to a stable nodal surface $(S,P)$ --- we refer to \cite{ArbarelloCornalbaGriffiths}, \cite{FarbMargalit} or \cite{HinichVaintrob} for details.

We recall first of all the definition of a stable Riemann surface with nodes. By a \textit{Riemann surface with nodes}, we mean a compact, connected, complex analytic space such that every point has a neighbourhood isomorphic either to the disk in $\C$ or to the cone $\{xy=0\}$ in $\C^2$; the latter points are called \textit{nodes}. 

Let $(S,P)$ be a stable nodal surface and consider the set $\mathcal{T}_{S}$ of equivalence classes of pairs $(X,f)$, where $X$ is a Riemann surface with nodes, and $f\colon S\smallsetminus P\rightarrow X$ an orientation preserving homeomorphism --- this is also called a \textit{marking of the curve}. Two such pairs $(X,f)$ and $(Y,g)$ are equivalent if there exists an isometry $\phi\colon X\rightarrow Y$ and an orientation preserving diffeomorphism $\gamma\colon (S,P)\rightarrow (S,P)$ which is isotopic to the identity and preserves the marked points pointwise and such that the following diagram commutes (note that this in particular implies that the $P$-labelling of the ``punctures'' of $X$ and $Y$ is preserved by the isometry).
\begin{center}
\begin{tikzpicture}
\matrix (m) [matrix of math nodes,row sep=2em,column sep=1em]
  {
S\smallsetminus P & & S\smallsetminus P \\
X & & Y \\
  };
  \path[-stealth]
(m-1-1) edge node[left]{$f$} (m-2-1) edge node[above]{$\gamma$} (m-1-3)
(m-1-3) edge node[right]{$g$} (m-2-3)
(m-2-1) edge node[below]{$\phi$} (m-2-3)
;
\end{tikzpicture}
\end{center} 
An equivalence class $[X,f]$ is referred to as a \textit{marked stable curve}.

The topology on $\mathcal{T}_{S}$ can be defined in various different ways. We shall content ourselves with the fact that if $C$ is a maximal admissible curve system on $(S,P)$, then the \textit{Fenchel--Nielsen coordinates} define a homeomorphism
\begin{align*}
\mathcal{T}_{S}\xrightarrow{\ \cong \ } (\R_{>0}\times \R)^C,
\end{align*}
with a factor for every closed simple curve in $C$. If $S$ has genus $g$, $k$ nodes and $|P|=n$, then the number of curves in $C$ is $m=3g-3+n-k$.

Let us very briefly recall the idea behind the Fenchel--Nielsen coordinates (see \cite[\S 10.6]{FarbMargalit} for details). One must first fix a choice of coordinate system of curves on $(S,P)$ consisting of the following data:
\begin{enumerate}
\item a \textit{pants decomposition} of $(S,P)$, that is, a choice $C=\{c_1,\ldots,c_m\}$ of disjoint closed simple curves such that $S\smallsetminus (C\cup \{\text{nodes}\}\cup P)$ is a disjoint union of pairs of pants (allowing also punctures instead of boundary components); in other words, $C$ is a maximal admissible curve system;
\item a choice of orientation of each curve $c_i$ in $C$ and for each pair of pants a choice of three arcs connecting the three boundary components (or punctures).
\end{enumerate}
The Fenchel--Nielsen coordinates of a marked hyperbolic surface $[X,f]$ is the tuple
\begin{align*}
(\ell_1,\ldots,\ell_m,\tau_1,\ldots,\tau_m),
\end{align*}
where $\ell_i$ is the length of the unique geodesic isotopic to $f(c_i)$ in $X$ and $\tau_i$ is the so-called twist parameter, which is more involved to define and specifies how the pairs of pants should be glued together: it measures the signed horizontal displacement of the endpoints of an arc crossing the geodesic representative of $f(c_i)$ when compared with the unique shortest arc connecting the same boundary components, that is, it tells us how the two pairs of pants on either side of the geodesic representative of $f(c_i)$ should be ``twisted'' before they are glued together.

\begin{observation}\label{decomposition of Teichmuller space according to the parts}
Let $(S,P)$ be a stable nodal surface with parts $(S_v, P_v)$, $v\in V_S$. Then the Teichmüller space associated to $(S,P)$ decomposes as a product:
\begin{align*}
\mathcal{T}_{S}\cong \prod_{v\in V_S} \mathcal{T}_{S_v}.
\end{align*}
Note that fixing a complex hyperbolic metric on $(S,P)$ amounts to choosing a basepoint in $\mathcal{T}_{S}$.
\end{observation}

The mapping class group $\Gamma_{S}$ acts on $\mathcal{T}_{S}$ as follows:
\begin{align*}
\overline{\gamma}.[X,f]=[X,f\circ \gamma^{-1}].
\end{align*}

We say that an action of a discrete group $G$ on a topological space $X$ is \textit{properly discontinuous} if every point $x\in X$ has a neighbourhood $U$ such that $gU\cap U= \emptyset$ for all but finitely many $g\in G$. The following is an important and well-known result and can be found in \cite[p. 80]{Abikoff80} (the result as stated is for non-singular $S$, but it extends immediately to the general case by \Cref{decomposition of Teichmuller space according to the parts} and the fact that $\operatorname{Aut}(\mathbf{G}(S,P))\cong \widetilde{\Gamma}_S\backslash \Gamma_S$ is finite).

\begin{proposition}
The action of $\Gamma_{S}$ on $\mathcal{T}_{S}$ is properly discontinuous.
\end{proposition}

\begin{remark}
For a topological space $X$ equipped with an action by a group $G$, we note that the quotient $\infty$-stack $[G\backslash X]=\varinjlim_{BG}X$ is equivalent to quotient $1$-stack of principal $G$-bundles $\pi\colon E\rightarrow B$ together with a $G$-equivariant map $\sigma\colon E\rightarrow X$ (see e.g.~\cite[p.286]{ArbarelloCornalbaGriffiths}) as both admit an atlas map from $X$ giving rise to the same \v{C}ech nerve (\cite{NikolausSchreiberStevenson}).
\end{remark}

The Teichmüller space $\mathcal{T}_{S}$ admits a complex structure and the resulting stack quotient $[\Gamma_{S}\backslash \mathcal{T}_{S}]$ is isomorphic to the analytic moduli stack $\mathcal{M}_{\mathbf{G}}^{an}$, where $\mathbf{G}=\mathbf{G}(S,P)$ is the associated dual graph (see \cite[\S 1.2]{HinichVaintrob} for the non-singular case which again generalises by \Cref{decomposition of Teichmuller space according to the parts} and the identification $\operatorname{Aut}(\mathbf{G}(S,P))\cong \widetilde{\Gamma}_S\backslash \Gamma_S$). Morally, the map $\mathcal{T}_{S}\rightarrow \mathcal{M}_{\mathbf{G}}^{an}$ just forgets the marking by sending $[X,f]$ to $X$. It follows that the same identification holds for the underlying topological stacks.

\begin{proposition}\label{moduli stack as quotient of Teichmuller}
Let $(S,P)$ be a stable nodal surface with associated dual graph $\mathbf{G}=\mathbf{G}(S,P)$. The topological stack quotient $[\Gamma_{S}\backslash \mathcal{T}_{S}]$ is isomorphic to the topological moduli stack $\mathcal{M}_{\mathbf{G}}^{\operatorname{top}}$.
\end{proposition}

\subsection{The bordification: construction, stratification and quotient}\label{Harvey bordification}

As mentioned, the Harvey bordification will play a crucial role in our calculations (it plays the role of the Borel--Serre compactification in \cite{ClausenOrsnesJansen}). We will need some details of the construction in order to review the desired properties, but we will not provide a complete description and refer the reader to \cite{Ivanov} for details and specifics (see also \cite{Harvey81} and \cite[Chapter 3]{Harer}). Note that in \cite{Ivanov}, the construction is done explicitly for non-singular surfaces, but it generalises immediately to the nodal case (see also \Cref{decomposition of Harvey bordification according to the parts}).

Let $(S,P)$ be a stable nodal surface. The \textit{Harvey bordification} $\overline{\mathcal{T}}_{\!\!S}=\overline{\mathcal{T}}_{\!\!(S,P)}$ is a smooth manifold with corners, whose interior identifies with the Teichmüller space $\mathcal{T}_S=\overline{\mathcal{T}}_{\!\!S}\smallsetminus \partial \overline{\mathcal{T}}_{\!\!S}$. It is equipped with an action of the mapping class group $\Gamma_S$ which extends the natural action of $\Gamma_S$ on $\mathcal{T}_S$. The most important properties of the bordification $\overline{\mathcal{T}}_{\!\!S}$ and the action of $\Gamma_S$ are summarised in the following theorem.

\begin{theorem}\ 
\begin{enumerate}
\item The action of $\Gamma_S$ on $\overline{\mathcal{T}}_{\!\!S}$ is properly discontinuous;
\item The quotient topological space $\Gamma_S\backslash \overline{\mathcal{T}}_{\!\!S}$ is compact;
\item $\overline{\mathcal{T}}_{\!\!S}$ is contractible; 
\end{enumerate}
\end{theorem}
\begin{proof}
Part 1 is \cite[Theorem 6.5]{Ivanov}, part 2 is \cite[Theorem 6.7]{Ivanov}, and part 3 follows directly from the fact that its interior, the Teichmüller space, is contractible.
\end{proof}

We will need a few additional observations. As a set, $\overline{\mathcal{T}}_{\!\!S}$ consists of the isotopy classes of singular hyperbolic structures on $(S,P)$ (\cite[Definitions 4.1 and 5.1]{Ivanov}): a \textit{singular hyperbolic structure} on $(S,P)$ consists of a pair $(C,X,f)$, where $C$ is an admissible system of curves on $(S,P)$, $X$ is a stable nodal Riemann surface and $f$ is an orientation preserving homeomorphism $f\colon S \smallsetminus (C\cup P) \rightarrow X$; this triple should moreover satisfy the following condition:

\begin{assumption}[Singular hyperbolic structure]\label{singular hyperbolic structure assumption}
For each curve $c$ in $C$, there is a neighbourhood $U\subseteq S$ of $c$ (disjoint from the remaining curves in $C$ and the nodes and marked points of $S$), and a homeomorphism $\phi\colon U\rightarrow S^1\times (-1,1)$  such that the restriction
\begin{align*}
\phi\vert_{U\smallsetminus c}\colon U\smallsetminus c\longrightarrow S^1\times (-1,1)\smallsetminus S^1\times \{0\}
\end{align*}
is an isometry when 
\begin{itemize}
\item $S^1\times (-1,1)\smallsetminus S^1\times \{0\}$ is equipped with the metric $t^2d\theta^2 + dt^2/t^2$ where $\theta$ is the standard angular parameter on $S^1$ and $t\in (-1,1)\smallsetminus\{0\}$;
\item $U\smallsetminus c$ is equipped with the hyperbolic metric of $f(U\smallsetminus c)\subset X$.
\end{itemize}
Note that the homeomorphism $\phi$ is unique up to a rotation. See \cite[Definition 4.1]{Ivanov} for more details.
\end{assumption}

Two such structures $(C,X,f)$, $(D,Y,g)$ are \textit{isotopic} if there is an isometry $\psi\colon X\rightarrow Y$ and an orientation preserving diffeomorphism $\gamma\colon (S,P)\rightarrow (S,P)$ which is isotopic to the identity, preserves the marked points pointwise, satisfies $\gamma(C)=D$, and such that the diagram below commutes. Moreover, $\gamma$ should respect the parametrisations of \Cref{singular hyperbolic structure assumption} above, that is, for every $c\in C$ and $U$ and $\phi$ as above, the parametrisation
\begin{align*}
\phi\circ \gamma^{-1}\colon \gamma(U)\rightarrow S^1\times (-1,1)
\end{align*}
of the neighbourhood $\gamma(U)$ of the curve $\gamma(c)$ in $D$ restricts to an isometry on the complement of $\gamma(c)$ and $S^1\times \{0\}$, when $\gamma(U)\smallsetminus\gamma(c)$ is equipped with the hyperbolic metric of $g(\gamma(U)\smallsetminus\gamma(c))\subset Y$.
\begin{center}
\begin{tikzpicture}
\matrix (m) [matrix of math nodes,row sep=2em,column sep=1em]
  {
S\smallsetminus (C\cup P) & & S\smallsetminus (D\cup P) \\
X & & Y \\
  };
  \path[-stealth]
(m-1-1) edge node[left]{$f$} (m-2-1) edge node[above]{$\gamma$} (m-1-3)
(m-1-3) edge node[right]{$g$} (m-2-3)
(m-2-1) edge node[below]{$\psi$} (m-2-3)
;
\end{tikzpicture}
\end{center}

The following observation is immediate.

\begin{observation}\label{decomposition of Harvey bordification according to the parts}
Let $(S,P)$ be a stable nodal surface with parts $(S_v,P_v)$, $v\in V_S$. The Harvey bordification $\overline{\mathcal{T}}_{\!\!S}$ decomposes as a product:
\begin{align*}
\overline{\mathcal{T}}_{\!\!S}\cong \prod_{v\in V_S} \overline{\mathcal{T}}_{\!\!S_v}.
\end{align*}
This fits with the analogous decompositions of the pure mapping class group and the curve poset (\Cref{product decomposition of pure mapping class group} and \Cref{observations about curve posets})
\end{observation}

The Harvey bordification is naturally stratified over the curve poset.

\begin{construction}[Stratification]\label{stratification of Harvey bordification}
Let $(S,P)$ be a stable nodal surface. The \textit{stratification map} is given by
\begin{align*}
s_S=s_{(S,P)}\colon \overline{\mathcal{T}}_{S}\rightarrow \mathfrak{C}_{S}, \quad [C,X,f]\mapsto [C],
\end{align*}
where $[C]$ is the isotopy class of the curve system $C$. Using the charts on $\overline{\mathcal{T}}_{S}$ from \cite[Definition 5.1]{Ivanov}, it can be verified directly that $s_S$ is continuous, and moreover, that this stratification coincides with the natural stratification of $\overline{\mathcal{T}}_{S}$ as a manifold with corners.
\end{construction}

The action of the mapping class group $\Gamma_{S}$ on $\overline{\mathcal{T}}_{S}$ is given by
\begin{align*}
\overline{\gamma}.[C,X,f]=[\gamma(C),X,f\circ \gamma^{-1}],\quad\quad\text{for }\overline{\gamma}\in \Gamma_S,\ [C,X,f]\in \overline{\mathcal{T}}_{S}.
\end{align*}

Thus we see that together with the action of the mapping class group on the curve poset, we have an action of $\Gamma_{S}$ on the stratified space $\overline{\mathcal{T}}_{S}\rightarrow \mathfrak{C}_{S}$, that is, compatible actions such that the stratification map is equivariant (see e.g.~\cite{OrsnesJansen}). From now on, we implicitly consider the Harvey bordification as a stratified topological space, omitting the stratification map from our notation unless explicitly needed.

Let us remark on the topology of the Harvey bordification, or more precisely, on the charts equipping it with the structure of a manifold with corners.

\begin{remark}\label{charts on Harvey bordification}
Let $(S,P)$ be a stable nodal surface, let $\sigma\in \mathfrak{C}_{S}$ and write  $U(\sigma)=s_S^{-1}((\mathfrak{C}_S)_{\geq \sigma})$ for the open star neighbourhood of the $\sigma$-stratum. For any minimal element $\mu\in \mathfrak{C}_S$ with $\mu\leq \sigma$ (i.e.~$\mu$ is a maximal curve system containing $\sigma$), the charts on $\overline{\mathcal{T}}_{\!\!S}$ from \cite[Definition 5.1]{Ivanov} provide stratified diffeomorphisms (with corners)
\begin{align*}
U(\sigma)\xrightarrow{\ \cong \ }\R^\sigma_{\geq 0}\times \R^{\mu\smallsetminus \sigma}_{>0}\times \R^\mu
\end{align*}
where the product $\R^\sigma_{\geq 0}$ has a factor for every simple closed curve in $\sigma$ and likewise for the other two products, and where the stratification on the right is given by the map
\begin{align*}
\R^\sigma_{\geq 0}\rightarrow (\mathfrak{C}_S)_{\geq \sigma}
\end{align*}
on the first factor sending a point $(r_c)_{c\in \sigma}$ to the curve system consisting the simple closed curves $c\in \sigma$ for which $r_c=0$.
\end{remark}

Since $[\Gamma_{S}\backslash \mathcal{T}_{S}]\cong\mathcal{M}^{\operatorname{top}}_{\mathbf{G}}$ (\Cref{moduli stack as quotient of Teichmuller}), the quotient of the Harvey bordification can be viewed as a compactification of the topological moduli stack $\mathcal{M}^{\operatorname{top}}_{\mathbf{G}}$. We record this as the following definition.

\begin{definition}
Let $(S,P)$ be a stable nodal surface with associated dual graph $\mathbf{G}=\mathbf{G}(S,P)$. The \textit{Harvey compactification} of the moduli stack of smooth stable curves $\mathcal{M}^{\operatorname{top}}_{\mathbf{G}}$ is the stack quotient $[\Gamma_{S}\backslash \overline{\mathcal{T}}_{\!\!S}]$. It is naturally stratified over the quotient poset $\Gamma_{S}\backslash \mathfrak{C}_{S}$.
\end{definition}

The following proposition relates the Harvey and Deligne--Mumford--Knudsen compactifications of $\mathcal{M}_{\mathbf{G}}^{\operatorname{top}}$. We will in fact only be needing the subsequent \Cref{proper map}, so we stick to the non-singular case for simplicity. Consider the complex analytic stack $\overline{\mathcal{M}}_{g,P}^{\operatorname{an}}$ and let $\operatorname{Bl}_\partial(\overline{\mathcal{M}}_{g,P}^{\operatorname{an}})$ denote the real oriented blow-up of $\overline{\mathcal{M}}_{g,P}^{\operatorname{an}}$ along the boundary $\partial\overline{\mathcal{M}}_{g,P}^{\operatorname{an}}$ (see \cite[Chapter X, \S 9]{ArbarelloCornalbaGriffiths} for details on real oriented blow-ups). The stack $\operatorname{Bl}_\partial(\overline{\mathcal{M}}_{g,P}^{\operatorname{an}})$ is by definition a real analytic stack, but we consider it in the category of stacks over smooth manifolds with corners; this is purely because the Harvey bordification as constructed by Ivanov is a priori only a smooth manifold with corners (see \cite[p. 1179]{Ivanov}).

\begin{theorem}[\Cref{Harvey compactification is real oriented blow-up}, \cite{Looijenga95}]\label{Harvey as real oriented blow-up}
Let $(S,P)$ be a non-singular stable surface of genus $g$. In the category of stacks over smooth manifolds with corners, the Harvey compactification can be identified with the real oriented blow-up of $\overline{\mathcal{M}}_{g,P}^{\operatorname{an}}$:
\begin{align*}
\operatorname{Bl}_\partial(\overline{\mathcal{M}}_{g,P}^{\operatorname{an}}) \xrightarrow{\ \cong \ } \left[\Gamma_{S}\backslash \overline{\mathcal{T}}_{S}\right]
\end{align*}
\end{theorem}

The resulting map $\left[\Gamma_{S}\backslash \overline{\mathcal{T}}_{S}\right] \longrightarrow \overline{\mathcal{M}}_{\mathbf{G}}^{\operatorname{top}}$ is essentially given by sending $[C,X,f]$ to the curve $[X]$.

\begin{corollary}\label{proper map}
For any stable nodal surface $(S,P)$ with associated dual graph $\mathbf{G}=\mathbf{G}(S,P)$, the map $\left[\Gamma_{S}\backslash \overline{\mathcal{T}}_{S}\right] \longrightarrow \overline{\mathcal{M}}_{\mathbf{G}}^{\operatorname{top}}$ 
is proper and restricts to an isomorphism over the open stratum $[\Gamma_{S}\backslash \mathcal{T}_{S}] \xrightarrow{\cong} \mathcal{M}_{\mathbf{G}}^{\operatorname{top}}$.
\end{corollary}
\begin{proof}
Indeed, the map arises as the quotient by the finite group $\operatorname{Aut}(\mathbf{G})$ of the map
\begin{align*}
\left[\widetilde{\Gamma}_{S}\backslash \overline{\mathcal{T}}_{S}\right] \longrightarrow \overline{\mathcal{M}}_{\mathbf{G}}^{\sim,\operatorname{top}}
\end{align*}
and this is proper as it identifies with the underlying map of topological spaces given by the real oriented blow-up by \Cref{Harvey as real oriented blow-up} and \Cref{decomposition of Harvey bordification according to the parts}. The final claim is \Cref{moduli stack as quotient of Teichmuller}.
\end{proof}

\begin{remark}
This identification of the Harvey compactification as a real oriented blow-up is due to Looijenga in \cite{Looijenga95}. The details are essentially spelt out in \cite[Chapter XV, \S 8]{ArbarelloCornalbaGriffiths}, but there, as well as in \cite{Looijenga95}, the Harvey bordification plays an auxiliary role as they are ultimately interested in a different bordification, namely the augmented Teichmüller space. We have chosen to review this identification with our particular interests in mind and with Ivanov's construction as our initial model --- we do this in \Cref{appendix}. We also need to be slightly careful to make the identification in the stack setting and not just for the coarse moduli spaces. One could argue that we might as well have taken the real oriented blow-up and its universal cover as our concrete models for the Harvey compactification, respectively bordification, and not gone into the details of Ivanov's construction --- we would thus avoid having to make this explicit identification. We feel, however, that Ivanov's construction makes the topological space, the stratifications and most importantly the quotients dealt with in the following section more accessible. These quotients will be essential to the inductive strategy of our final calculation.
\end{remark}

\subsection{Twists of singular hyperbolic structures}\label{twist of singular hyperbolic structures}

In this section we introduce \textit{twists of singular hyperbolic structures}. This will be an important ingredient in our final calculation.

\begin{notation}
For any stable nodal surface $(S,P)$ and any $\sigma\in \mathfrak{C}_{S}$, let 
\begin{align*}
X^{S}_\sigma:=s_S^{-1}(\sigma),\quad\text{and}\quad \overline{X}^{S}_\sigma:=s_S^{-1}((\mathfrak{C}_{S})_{\leq \sigma})
\end{align*}
denote the $\sigma$-stratum, respectively the closure of the $\sigma$-stratum, in $\overline{\mathcal{T}}_{\!\!S}$. We keep note of the stable nodal surface $S$ in the superscript as we will be varying this in our calculations (as usual, we omit the marked points for notational ease). As a set, the closure of the $\sigma$-stratum decomposes as a disjoint union
\begin{equation*}
\overline{X}^{S}_\sigma = \coprod_{\tau\leq \sigma}X^S_\tau
\end{equation*}
of $\tau$-strata for $\tau\leq \sigma$.
\end{notation}

Let $(S,P)$ be a stable nodal surface. Let $\tau\in \mathfrak{C}_{S}$ and write $\R^\tau$ for the additive group with a factor for every simple closed curve in $\tau$. For any $\sigma\leq \tau$, there is an action of $\R^{\tau}$ on the stratum $X^{S}_\sigma$ given by ``twists'': intuitively, we twist the hyperbolic structure around the curves of $\tau\subset \sigma$ using the parametrisations of \Cref{singular hyperbolic structure assumption}. The following identification of this action suffices for our needs. Let $\sigma\in \mathfrak{C}_{S}$ and let $\mu\in \mathfrak{C}_{S}$ be a minimal element such that $\mu\leq \sigma$. We have an $\R^\sigma$-equivariant diffeomorphism
\begin{align}\label{equivariant homeomorphisms}
X^{S}_\sigma\xrightarrow{\ \cong\ } \R_{>0}^{\mu\smallsetminus \sigma}\times \R^\mu,
\end{align}
where the action on the right hand side is given by the canonical action on the second factor via the inclusion $\R^\sigma\hookrightarrow \R^\mu$ (\cite[Lemma 4.5 and Definition 5.1]{Ivanov}).

\begin{remark}
We refer to \cite[4.3]{Ivanov} for the explicit definition of this action. Note that there it is defined explicitly for a non-singular stable surface; it generalises easily to the nodal case. In fact, the action is slightly more general: $\R^\tau$ acts on the stratum $X^{S}_\sigma$ for any $\sigma$ such that $\sigma\cup \tau$ is represented by an admissible curve system, in other words, is a well-defined element in $\mathfrak{C}_{S}$. We will not need this generality, however.
\end{remark}

\begin{remark}\label{Remark: Fenchel-Nielsen coordinates on strata}
The map (\ref{equivariant homeomorphisms}) above should be compared with the Fenchel--Nielsen coordinates of the Teichmüller space associated to a stable nodal curve $(S/C, P)$ with $C$ a representative of $\sigma\in \mathfrak{C}_{S}$: fixing a maximal curve system $\mu$ containing $\sigma$, the Fenchel--Nielsen coordinates define a homeomorphism
\begin{align*}
\mathcal{T}_{S/C}\xrightarrow{\ \cong \ }(\R_{>0}\times \R)^{\mu\smallsetminus \sigma}.
\end{align*}

The map (\ref{equivariant homeomorphisms}) is given by length and twist parameters in a similar fashion, but there is an extra twist parameter for every simple closed curve in $\sigma$. This is a consequence of \Cref{singular hyperbolic structure assumption}: the curves of $\sigma$ will correspond to nodes in the hyperbolic structure (with ``length 0''), but we require the existence of neighbourhoods of these same curves capturing the metric on the complement of the curve; this amounts to remembering the twist needed to glue the two ends together. See also Remark 5.5 in \cite{Ivanov} on the compatibilities of the topologies on $\mathcal{T}_{S}$ and $\overline{\mathcal{T}}_{S}$.
\end{remark}

If $\sigma\leq \tau\leq \tau'$, then the action of $\R^{\tau'}$ on $X^{S}_\sigma$ is the restriction of the action of $\R^\tau$ under the natural inclusion $\R^{\tau'}\hookrightarrow \R^\tau$ induced by the containment $\tau'\subseteq \tau$. For any given $\sigma\leq \tau$, the action of $\R^\tau$ on each $X^S_\theta\subseteq \overline{X}^S_\sigma$. $\theta\leq \sigma$ defines a continuous action on the closure $\overline{X}_\sigma^S$ of the $\sigma$-stratum --- this can be verified by considering the charts on $\overline{\mathcal{T}}_{\!\!S}$ (\Cref{charts on Harvey bordification}). The following observation will be essential to our calculations later on.

\begin{lemma}\label{maps from twisted arrow category}
Let $(S,P)$ be a stable nodal surface. Given $\sigma,\sigma',\tau, \tau'\in \mathfrak{C}_{S}$ and $\gamma\in \Gamma_{S}$ such that
\begin{align*}
\gamma.\sigma'\leq \sigma\leq \tau\leq \gamma.\tau',
\end{align*}
the induced map $\gamma^*\colon \overline{X}^{S}_{\sigma'}\rightarrow \overline{X}^{S}_{\sigma}$, $x\mapsto\gamma.x$, descends to define a map of the quotients as in the diagram below:
\begin{center}
\begin{tikzpicture}
\matrix (m) [matrix of math nodes,row sep=2em,column sep=2em]
  {
\overline{X}^{S}_{\sigma'} & \overline{X}^{S}_{\sigma} \\
\overline{X}^{S}_{\sigma'}/\R^{\tau'} & \overline{X}^{S}_{\sigma}/\R^\tau \\
  };
  \path[-stealth]
(m-1-1) edge node[above]{$\gamma^*$} (m-1-2)
(m-2-1) edge (m-2-2)
(m-1-1) edge (m-2-1)
(m-1-2) edge (m-2-2)
;
\end{tikzpicture}
\end{center} 
\end{lemma}
\begin{proof}

This is easily seen by analysing the composite $\overline{X}^{S}_{\sigma'}\rightarrow \overline{X}^{S}_{\sigma}\rightarrow \overline{X}^{S}_{\sigma}/\R^\tau$ stratumwise and exploiting the equivariant homeomorphisms (\ref{equivariant homeomorphisms}) above: for any $\theta\leq \sigma'$, we have a commutative diagram as below where $\mu$ is some minimal element with $\mu\leq \theta$.

\begin{center}
\begin{tikzpicture}
\matrix (m) [matrix of math nodes,row sep=2em,column sep=2em]
  {
X^{S}_{\theta} & X^{S}_{\gamma.\theta} & X^{S}_{\gamma.\theta}/\R^{\gamma.\tau'} & X^{S}_{\gamma.\theta}/\R^\tau \\
\R_{>0}^{\mu\smallsetminus \theta}\times \R^{\mu} & \R_{>0}^{\gamma.\mu\smallsetminus \gamma.\theta}\times \R^{\gamma.\mu} & \R_{>0}^{\gamma.\mu\smallsetminus \gamma.\theta}\times \R^{\gamma.\mu\smallsetminus \gamma.\tau'}& \R_{>0}^{\gamma.\mu\smallsetminus \gamma.\theta}\times \R^{\gamma.\mu\smallsetminus \tau} \\
  };
  \path[-stealth]
(m-1-1) edge node[above]{$\gamma_*$} node[below]{$\cong$} (m-1-2) edge node[right]{$\cong$} (m-2-1)
(m-2-1) edge (m-2-2)
(m-1-2) edge node[right]{$\cong$} (m-2-2)
(m-1-3) edge node[right]{$\cong$} (m-2-3)
(m-1-4) edge node[right]{$\cong$} (m-2-4)
;
\path[->>]
(m-1-2) edge (m-1-3)
(m-1-3) edge (m-1-4)
(m-2-2) edge (m-2-3)
(m-2-3) edge (m-2-4)
;
\end{tikzpicture}
\end{center} 

Since the lower line factors through $\R_{>0}^{\mu\smallsetminus \theta}\times \R^{\mu\smallsetminus \tau'}$, the upper line factors through $X^{S}_\theta/\R^{\tau'}$ as desired.
\end{proof}

Finally, we will need the following identifications for our inductive strategy later on.

\begin{proposition}\label{identifying quotients by additive action as smaller Harvey bordifications}
Let $(S,P)$ be a stable nodal surface and let $\tau\in \mathfrak{C}_{S}$ be represented by an admissible curve system $C$. There is a fibre sequence
\begin{align*}
\R^\tau\rightarrow \overline{X}^{S}_\tau\rightarrow \overline{\mathcal{T}}_{S/C},
\end{align*}
identifying the quotient $\overline{X}^{S}_\tau/\R^\tau$ with the Harvey bordification $\overline{\mathcal{T}}_{S/C}$ associated to the stable nodal curve $(S/C,P)$. More generally, if $\sigma \leq \tau$, then the quotient $\overline{X}^{S}_\sigma/\R^\tau$ naturally identifies with the closure of the $\sigma$-stratum $\overline{X}^{S/C}_\sigma$ in $\overline{\mathcal{T}}_{S/C}$.
\end{proposition}
\begin{proof}
This is contained in the proof of \cite[Theorem 6.8]{Ivanov}.
\end{proof}

\section{Inductive decompositions}\label{inductive decompositions}

We now have all the necessary ingredients and we commence the work towards the final calculation. In this section we exploit proper descent of sheaf categories to exhibit the inductive nature of the Deligne--Mumford--Knudsen compactifications and relate them to the Harvey compactifications. This is an immediate consequence of the observations made so far, and the real work of this section will be to supply diagrammatic analogues of these inductive decompositions.

For the whole of this section, we fix a stable nodal surface $(S,P)$ with associated dual graph $\mathbf{G}=\mathbf{G}(S,P)$, and we fix a complex hyperbolic structure on $S\smallsetminus P$, i.e.~a basepoint in $\mathcal{T}_{S}$. The reader is welcome to think of a non-singular surface $S$ --- it makes no difference to the arguments at hand and the only change of notation is that for a non-singular $S$ of genus $g$, we would tend to replace subscript $\mathbf{G}$ by subscript $g,P$.

\subsection{Categories of sheaves}

Having fixed a complex hyperbolic structure on $S\smallsetminus P$, for any $\sigma\in \mathfrak{C}_S$, there is a unique admissible curve system $C_\sigma$ representing $\sigma$ by geodesics (\cite[\S 1.2]{FarbMargalit}). We write
\begin{align*}
(S/\sigma,P):=(S/C_\sigma,P)
\end{align*}
for the resulting stable nodal surface. Recall the isomorphism of posets
\begin{align*}
\Gamma_S\backslash \mathfrak{C}_S\xrightarrow{\cong} \mathcal{G}_{\mathbf{G}},\quad [\theta]\mapsto [\mathbf{G}(\theta)]
\end{align*}
where $G(\theta)=\mathbf{G}(S/\theta,P)$ (\Cref{iso of posets}). For $\theta\in \mathfrak{C}_S$, write
\begin{align*}
\mathcal{D}_{\theta}:=\mathcal{D}_{\mathbf{G}(\theta)}^{\operatorname{top}}\quad\text{and}\quad\overline{\mathcal{D}}_{\theta}:=\overline{\mathcal{D}}_{\mathbf{G}(\theta)}^{\operatorname{top}} 
\end{align*}
for the topological stacks underlying the substacks $\mathcal{D}_{\mathbf{G}(\theta)}\subset \overline{\mathcal{M}}_{\mathbf{G}}$, respectively $\overline{\mathcal{D}}_{\mathbf{G}(\theta)}\subset \overline{\mathcal{M}}_{\mathbf{G}}$ (\Cref{stable curves dual graphs and moduli stacks}). Note that $\mathcal{D}_{\theta}$ and $\overline{\mathcal{D}}_{\theta}$ identify with the following fibre products, 
\begin{align*}
\overline{\mathcal{D}}_{\theta}\xrightarrow{\sim}\overline{\mathcal{M}}_{\mathbf{G}}^{\operatorname{top}}\mathop{\times}_{\mathcal{G}_{\mathbf{G}}}\{[\mathbf{G}(\theta)]\}\quad\text{and}\quad \overline{\mathcal{D}}_{\theta}\xrightarrow{\sim}\overline{\mathcal{M}}_{\mathbf{G}}^{\operatorname{top}}\mathop{\times}_{\mathcal{G}_{\mathbf{G}}}(\mathcal{G}_{\mathbf{G}})_{\leq [\mathbf{G}(\theta)]};
\end{align*}
that is, the $[\mathbf{G}(\theta)]$-stratum, respectively the closure of the $[\mathbf{G}(\theta)]$-stratum of $\overline{\mathcal{M}}_{\mathbf{G}}^{\operatorname{top}}$.

Let $\theta\in \mathfrak{C}_S$ and consider the following commutative diagram of stratified topological stacks (the upper diagram is the diagram of topological stacks, the lower is the corresponding diagram of stratifying posets; we omit the actual stratification maps).
\begin{center}
\begin{tikzpicture}
\matrix (m) [matrix of math nodes,row sep=2em,column sep=2em]
  {
\overline{\mathcal{T}}_{S/\theta} & \left[\widetilde{\Gamma}_{S/\theta}\backslash \overline{\mathcal{T}}_{S/\theta}\right] & \overline{\mathcal{M}}_{\mathbf{G}(\theta)}^{\sim,\operatorname{top}} & \\
& \left[\Gamma_{S/\theta}\backslash \overline{\mathcal{T}}_{S/\theta}\right] & \overline{\mathcal{M}}^{\operatorname{top}}_{\mathbf{G}(\theta)} & \overline{\mathcal{D}}_{\theta} \\
\mathfrak{C}_{S/\theta} & \widetilde{\Gamma}_{S/\theta}\backslash \mathfrak{C}_{S/\theta} & \mathcal{G}_{\mathbf{G}(\theta)}^{\sim}& \\
& \Gamma_{S/\theta}\backslash \mathfrak{C}_{S/\theta} & \mathcal{G}_{\mathbf{G}(\theta)} & (\mathcal{G}_{\mathbf{G}})_{\leq \mathbf{G}(\theta)} \\
  };
  \path[-stealth]
(m-1-1) edge (m-1-2) edge (m-2-2)
(m-1-2) edge (m-1-3) edge (m-2-2)
(m-1-3) edge (m-2-3) edge (m-2-4)
(m-2-2) edge (m-2-3)
(m-2-3) edge (m-2-4)
(m-3-1) edge (m-3-2) edge (m-4-2)
(m-3-2) edge (m-3-3) edge (m-4-2)
(m-3-3) edge (m-4-3) edge (m-4-4)
(m-4-2) edge (m-4-3)
(m-4-3) edge (m-4-4)
;
\end{tikzpicture}
\end{center}

\begin{lemma}\label{useful proper map Mg bar}
For any $\theta\in \mathfrak{C}_S$, the composite of the lower horizontal maps in the diagram above
\begin{align*}
[\Gamma_{S/\theta}\backslash \overline{\mathcal{T}}_{S/\theta}]\longrightarrow \overline{\mathcal{M}}^{\operatorname{top}}_{\mathbf{G}(\theta)}\longrightarrow \overline{\mathcal{D}}_{\theta}
\end{align*}
is proper and restricts to an isomorphism over the open stratum $[\Gamma_{S/\theta}\backslash \mathcal{T}_{S/\theta}]\xrightarrow{\cong}\mathcal{D}_{\theta}$.
\end{lemma}
\begin{proof}
This follows from \Cref{proper map} and \Cref{DMK compactification as normalisation}.
\end{proof}

We write $\partial\overline{\mathcal{D}}_{\theta}=\overline{\mathcal{D}}_{\theta}\smallsetminus\mathcal{D}_{\theta}$ for the boundary of $\overline{\mathcal{D}}_{\theta}$. The following should be compared with Corollary 4.12 of \cite{ClausenOrsnesJansen}:

\begin{proposition}\label{preliminary decomposition of sheaves on Mg bar}
For any $\theta\in \mathfrak{C}_S$, we have:
\begin{enumerate}
\item an equivalence of categories of sheaves
\begin{align*}
\displaystyle\operatorname{Shv}(\overline{\mathcal{D}}_{\theta})\xrightarrow{\ \sim\ }\operatorname{Shv}(\partial \overline{\mathcal{D}}_{\theta}) \mathop{\times}_{\operatorname{Shv}(\partial(\Gamma_{S/\theta}\backslash \overline{\mathcal{T}}_{S/\theta}))}\operatorname{Shv}(\Gamma_{S/\theta}\backslash \overline{\mathcal{T}}_{S/\theta}),
\end{align*}
and similarly for constructible sheaves;
\item an equivalence of categories of sheaves
\begin{align*}
\displaystyle\operatorname{Shv}(\partial\overline{\mathcal{D}}_{\theta})\xrightarrow{\ \sim\ }\varprojlim_{[\sigma]\in (\Gamma_S\backslash \mathfrak{C}_S)_{\leq[\theta]}^{op}\smallsetminus [\theta]}\operatorname{Shv}(\overline{\mathcal{D}}_{\sigma}),
\end{align*}
and similarly for constructible sheaves.
\end{enumerate}
\end{proposition}
\begin{proof}
Part 1 for sheaves follows from part 1 of the lemma above and proper descent (\cite[Corollary 2.18]{OrsnesJansen23}). To deduce the claim for constructible sheaves, it suffices to note that a sheaf on $\overline{\mathcal{D}}_{\theta}$ is constructible if and only if its pullback to the other three terms is: this is clear, since the only stratum not contained in the boundary is the open stratum $\mathcal{D}_\theta$ over which the projection map from $\Gamma_{S/\theta}\backslash \overline{\mathcal{T}}_{S/\theta}$ is an isomorphism. Part 2 for sheaves follows directly from descent for closed covers (\cite[Corollary 2.19]{OrsnesJansen23}), and to deduce the claim for constructible sheaves it suffices to note that every stratum is contained in its closure.
\end{proof}

\begin{remark}
Intuitively, this result is telling us that the Deligne--Mumford--Knudsen compactification $\overline{\mathcal{M}}_{g,P}$ can by ``built'' from the Harvey compactification $[\Gamma_S\backslash \overline{\mathcal{T}}_{\!\!S}]$ (for $S$ of type $(g,P)$) and the ``smaller'' $\overline{\mathcal{D}}_\sigma$, $\sigma<\emptyset$ in $\mathfrak{C}_S$. As $\overline{\mathcal{D}}_\sigma$ does not necessarily identify with one of the stacks $\overline{\mathcal{M}}_{\mathbf{G}}$ (but only normalises to one), we need the statement for general $\theta$ and not just $\theta=\emptyset$. See also \Cref{key technical difference RBS versus DMK}.
\end{remark}

\subsection{A diagrammatic analogue}

We now provide diagrammatic analogues of the inductive decompositions of sheaf categories established in the previous section.

Recall the \textit{twisted arrow category} $\operatorname{Tw}(\mathcal{C})$ of a category $\mathcal{C}$: its objects are the maps $x\rightarrow y$ in $\mathcal{C}$ and a morphism $(f\colon x\rightarrow y)\rightarrow (f'\colon x'\rightarrow y')$ is a factorisation of $f'$ through $f$, that is, a pair of maps $a\colon x'\rightarrow x$ and $b\colon y\rightarrow y'$ such that $f'=bfa$.

We still fix a stable nodal surface $(S,P)$ with dual graph $\mathbf{G}$ and a fixed hyperbolic structure. Consider the Charney--Lee category $\operatorname{CL}_S$ of $(S,P)$ and its associated twisted arrow category $\operatorname{Tw}(\operatorname{CL}_S)$. Let us first of all make an observation providing an alternative description of this twisted arrow category that will be useful when identifying subcategories as in the proposition below and also when defining functors out of it (the proof is completely analogous to the similar description in \cite[Lemma 4.17]{ClausenOrsnesJansen}).

\begin{lemma}\label{alternative description of twisted arrow category}
The category $\operatorname{Tw}(\operatorname{CL}_S)$ is equivalent to the category whose objects are pairs of objects $\sigma\leq \tau$ in $\mathfrak{C}_S$, and whose set of maps $(\sigma\leq \tau)\rightarrow (\sigma'\leq \tau')$ is given by
\begin{align*}
\Delta_\tau\backslash \{\gamma\in \Gamma_S\mid \gamma.\sigma'\leq \sigma, \gamma.\tau'\geq \tau\}
\end{align*}
with composition induced by multiplication in $\Gamma_S$.

Concretely, the equivalence is given by the functor that on objects sends $\sigma\leq \tau$ to the map $\sigma\xrightarrow{[\operatorname{id}]}\tau$ in $\operatorname{CL}_S$, and on maps sends $[\gamma]\colon (\sigma\leq \tau)\rightarrow (\sigma'\leq \tau')$ to the pair of maps $\sigma'\xrightarrow{\operatorname{[\gamma]}}\sigma$ and $\tau\xrightarrow{\operatorname{[\gamma^{-1}]}}\tau'$ in $\operatorname{CL}_S$.
\end{lemma}
\begin{proof}
The functor in question is well-defined since if $[\gamma]\colon (\sigma\leq \tau)\rightarrow (\sigma'\leq \tau')$ as in the statement, then
\begin{align*}
\Delta_{\tau}\subset \Delta_{\sigma}\subset \Delta_{\gamma.\sigma'}=\gamma\Delta_{\sigma'}\gamma^{-1},
\end{align*}
which implies that $\Delta_{\tau}\gamma\subset
\gamma\Delta_{\sigma'}$. To prove that it is an equivalence, we need to show 1) that every object $[\gamma]\colon \sigma\rightarrow \tau$ in $\operatorname{Tw}(\operatorname{CL}_S)$ is isomorphic to one whose defining map is given by $\gamma=\operatorname{id}$, and 2) that maps in $\operatorname{Tw}(\operatorname{CL}_S)$ between objects of this type are given by the set in the statement.

The first claim is easy, since any map $[\gamma]\colon \sigma\rightarrow \tau$ in $\operatorname{CL}_S$ factors as a composite as below where the first map is an isomorphism
\begin{align*}
\sigma\xrightarrow{[\gamma]}\gamma.\sigma\xrightarrow{\operatorname{[id]}}\tau.
\end{align*}
For the second claim, note that maps
\begin{align*}
(\sigma\xrightarrow{\operatorname{[id]}} \tau)\longrightarrow (\sigma'\xrightarrow{\operatorname{[id]}} \tau')
\end{align*}
in $\operatorname{Tw}(\operatorname{CL}_S)$ are by definition given by
\begin{align*}
\{\gamma_a,\gamma_b\in \Gamma_S\mid \gamma_a.\sigma'\leq \sigma, \gamma_b.\tau\leq \tau', \gamma_b\gamma_a\in \Delta_{\sigma'}\}/\sim,
\end{align*}
where $(\gamma_a,\gamma_b)\sim (\rho_a,\rho_b)$ if and only if $\gamma_a\Delta_{\sigma'}=\rho_a\Delta_{\sigma'}$ and $\gamma_b\Delta_{\tau}=\rho_b\Delta_{\tau}$. It follows that we can uniquely specify $\gamma_a$ in terms of $\gamma_b$ by setting $\gamma_a={\gamma_b}^{\!-1}$, and the claim follows.
\end{proof}

From now on, we consider $\operatorname{Tw}(\operatorname{CL}_S)$ as described above. In the same vein, we make the following observation that we will need later on.

\begin{lemma}\label{equivalence of comma categories}
For any $\sigma\in \mathfrak{C}_S$, the natural functor
\begin{align*}
(\mathfrak{C}_S)_{/\sigma}\rightarrow (\operatorname{CL}_S)_{/\sigma}
\end{align*}
induced by $(\tau\leq \sigma)\mapsto (\tau\xrightarrow{\operatorname{[\operatorname{id}]}}\sigma)$ is an equivalence.
\end{lemma}
\begin{proof}
By direct calculation, the functor is essentially surjective and fully faithful.
\end{proof}

For $\theta\in \mathfrak{C}_S$, consider the action of the mapping class group $\Gamma_{S/\theta}$ on the curve poset $\mathfrak{C}_{S/\theta}$ and recall that the action category $\Gamma_{S/\theta}\backslash \backslash \mathfrak{C}_{S/\theta}$ is the category whose objects are the elements $\sigma\in \mathfrak{C}_{S/\theta}$ and whose morphisms $\sigma\rightarrow \tau$ are the $\gamma\in \Gamma_{S/\theta}$ such that $\gamma.\sigma\leq \tau$. \Cref{decomposition of twisted arrow category} below and its corollary allow us to decompose the category $\operatorname{Tw}(\operatorname{CL}_S)$ and functors out of it, providing a diagrammatic analogue of the inductive nature of the Deligne-Mumford compactification established in \Cref{preliminary decomposition of sheaves on Mg bar}. But first of all, we introduce some notation.

\begin{definition}\label{full subcategories of twisted arrow category}
Let $\theta\in \mathfrak{C}_S$ and denote by $\operatorname{Tw}(\operatorname{CL}_S)^\theta\subset \operatorname{Tw}(\operatorname{CL}_S)$ the full subcategory of those $(\sigma\leq \tau)$ satisfying $\tau\leq \theta$. We consider the following full subcategories of $\operatorname{Tw}(\operatorname{CL}_S)^\theta$:
\begin{enumerate}[label=(\roman*)]
\item For $[\nu]\in (\Gamma_S\backslash \mathfrak{C}_S)_{\leq  [\theta]}$, we write
\begin{align*}
\operatorname{Tw}(\operatorname{CL}_S)_{\leq [\nu]}^\theta
\end{align*}
for the full subcategory of $\operatorname{Tw}(\operatorname{CL}_S)^\theta$ on those $(\sigma\leq \tau)$ with $\sigma\leq \nu'$ for some representative $\nu'$ of $[\nu]$;
\item For $[\nu]\in (\Gamma_S\backslash \mathfrak{C}_S)_{\leq  [\theta]}$, we write
\begin{align*}
\operatorname{Tw}(\operatorname{CL}_{S})_{[\nu]}^\theta
\end{align*}
for the further full subcategory of those $(\sigma\leq \tau)$ with $\tau\leq \nu'$ for some representative $\nu'$ of $[\nu]$.
\item We view $(\Gamma_{S/\theta}\backslash \backslash \mathfrak{C}_{S/\theta})^{op}$ as a full subcategory of $\operatorname{Tw}(\operatorname{CL}_S)^\theta$ via the canonical identification $\mathfrak{C}_{S/\theta}\cong (\mathfrak{C}_S)_{\leq \theta}$ and the inclusion $\sigma\mapsto (\sigma\leq \theta)$ for $\sigma\in (\mathfrak{C}_S)_{\leq \theta}$. In particular, we view $B\Gamma_{S/\theta}$ as the full subcategory on the object $(\theta\leq \theta)$. 
\end{enumerate}
When $\theta=\emptyset$, we have $\operatorname{Tw}(\operatorname{CL}_S)^\emptyset=\operatorname{Tw}(\operatorname{CL}_S)$ and will also write
\begin{equation*}
\operatorname{Tw}(\operatorname{CL}_S)_{\leq [\nu]}^\emptyset=\operatorname{Tw}(\operatorname{CL}_S)_{\leq [\nu]}\quad\text{and}\quad
\operatorname{Tw}(\operatorname{CL}_S)_{[\nu]}^\emptyset=\operatorname{Tw}(\operatorname{CL}_S)_{[\nu]}.\qedhere
\end{equation*}
\end{definition}

We now review some properties of these subcategories and the various inclusions in order to establish the desired decompositions

\begin{proposition}\label{decomposition of twisted arrow category}
Let $\theta\in\mathfrak{C}_S$.
\begin{enumerate}
\item The full subcategories $\operatorname{Tw}(\operatorname{CL}_S)^\theta_{\leq [\nu]}$ and $(\Gamma_{S/\theta}\backslash\backslash \mathfrak{C}_{S/\theta})^{op}$ are right closed in $\operatorname{Tw}(\operatorname{CL}_S)^{\theta}$;
\item The union of the subcategories $\operatorname{Tw}(\operatorname{CL}_S)_{\leq[\nu]}^\theta$ for $[\nu]< [\theta]$ is equal to the complement
\begin{align*}
\operatorname{Tw}(\operatorname{CL}_S)^\theta\smallsetminus B\Gamma_{S/\theta};
\end{align*}
\item For any $\theta\in \mathfrak{C}_S$, the inclusion
\begin{align*}
B\Gamma_{S/\theta}\hookrightarrow \Gamma_{S/\theta}\backslash\backslash \mathfrak{C}_{S/\theta}^{op}
\end{align*}
admits a right adjoint.
\item For any $\theta\in \mathfrak{C}_S$ and any $[\nu]\leq [\theta]$, the inclusion
\begin{align*}
\operatorname{Tw}(\operatorname{CL}_{S})_{[\nu]}^\theta\hookrightarrow\operatorname{Tw}(\operatorname{CL}_S)_{\leq[\nu]}^\theta
\end{align*}
is a $\varprojlim$-equivalence.
\item For any $\nu\leq \theta$, the inclusion
\begin{align*}
\operatorname{Tw}(\operatorname{CL}_{S})^\nu\hookrightarrow \operatorname{Tw}(\operatorname{CL}_{S})_{[\nu]}^\theta
\end{align*}
is an equivalence.
\end{enumerate}
\end{proposition}
\begin{proof}
Parts 1 and 2 are obvious. For part 3, the right adjoint is given by sending $(\sigma\leq \theta)$ to the unique object $(\theta\leq \theta)$ in $B\Gamma_{S/\theta}$ and by the inclusion on hom-sets
\begin{align*}
\operatorname{Hom}(\sigma\leq \theta,\sigma'\leq \theta)=\{\gamma\in \Gamma_{S/\theta}\mid \gamma.\sigma'\leq \sigma\}\hookrightarrow \Gamma_{S/\theta}.
\end{align*}
For part 5, simply note that any object in $\operatorname{Tw}(\operatorname{CL}_{S})_{[\nu]}^\theta$ is isomorphic to an object $(\sigma\leq\tau)$ with $\tau\leq \nu$.

To prove part 4, we need to show that for any $(\sigma\leq \tau)$ in $\operatorname{Tw}(\operatorname{CL}_S)_{\leq[\nu]}^\theta$, the left fibre $F_{/(\sigma\leq \tau)}$ is contractible in the sense that $|F_{/(\sigma\leq \tau)}|\simeq \ast$ (Joyal's Theorem A, see \cite[Theorem 2.19]{ClausenOrsnesJansen}).

Fix $(\sigma\leq \tau)$ in $\operatorname{Tw}(\operatorname{CL}_S)_{\leq[\nu]}^\theta$, and note that $F_{/(\sigma\leq \tau)}$ is the usual $1$-categorical comma category, as we are taking the left fibre of a functor of $1$-categories.  First of all, we observe that $F_{/(\sigma\leq \tau)}$ is equivalent to the full subcategory $F_{/(\sigma\leq \tau)}^{\operatorname{id}}$ spanned by the objects
\begin{align*}
((\sigma'\leq \tau'), (\sigma'\leq \tau')\xrightarrow{[\operatorname{id}]}(\sigma\leq \tau))
\end{align*}
whose defining map is induced by the identity. To see this, one needs to check that for any given object $((\sigma'\leq \tau'),(\sigma'\leq \tau')\xrightarrow{[\gamma]}(\sigma\leq \tau))$ in $F_{/(\sigma\leq \tau)}$, the object $(\gamma^{-1}.\sigma'\leq \gamma^{-1}.\tau')$ belongs to $\operatorname{Tw}(\operatorname{CL}_S)_{[\nu]}^\theta$: indeed,
\begin{align*}
\gamma^{-1}.\tau'\leq \tau\leq \theta\quad\text{and}\quad \gamma^{-1}.\tau'\leq \gamma^{-1}\nu' \quad\text{for some representative }\nu'\text{ of }[\nu]\text{ with }\tau'\leq \nu'.
\end{align*}

We now construct a sufficiently contractible cover of $F_{/(\sigma\leq \tau)}^{\operatorname{id}}$. Recall that $\Gamma_S(\nu)\subset \Gamma_S$ denotes the stabiliser of $\nu$. For any $[\gamma]\in \Gamma_S/\Gamma_S(\nu)$, consider the full subcategory $A_{[\gamma]}$ of $F_{/(\sigma\leq \tau)}^{\operatorname{id}}$ spanned by the objects
\begin{align*}
((\sigma'\leq \tau'),(\sigma'\leq \tau')\xrightarrow{[\operatorname{id}]}(\sigma\leq \tau))\quad\text{with}\quad \tau'\leq \gamma.\nu.
\end{align*}

Note that $A_{[\gamma]}$ is left closed and that $A_{[\gamma]}\neq \emptyset$ if and only if $\sigma\leq \gamma.\nu$, in which case $A_{[\gamma]}$ has a terminal object
\begin{align*}
((\sigma\leq \tau\cup \gamma.\nu),(\sigma\leq \tau\cup \gamma.\nu)\xrightarrow{[\operatorname{id}]} (\sigma\leq \tau)).
\end{align*}
The curve system $\tau\cup \gamma.\nu$ is a well-defined element of $\mathfrak{C}_S$ exactly because $\sigma$ contains both $\tau$ and $\gamma.\nu$.

More generally, for any subset $\{[\gamma_i]\}_{i\in I}\subset \Gamma_S/\Gamma_S(\nu)$, the intersection $A_I=\bigcap_{i\in I}A_{[\gamma_i]}$ is non-empty if and only if $\sigma\leq \gamma_i.\nu$ for all $i\in I$, in which case $A_I$ has a terminal object
\begin{align*}
((\sigma\leq \tau\cup\bigcup_{i\in I} \gamma_i.\nu),(\sigma\leq \tau\cup \bigcup_{i\in I} \gamma_i.\nu)\xrightarrow{[\operatorname{id}]} (\sigma\leq \tau)).
\end{align*}
In particular, any collection of non-empty $A_{[\gamma]}$'s have non-empty intersection. In other words, we have constructed a cover of $F_{/(\sigma\leq \tau)}$ by left closed full subcategories whose intersections are all non-empty and have terminal objects. It follows by descent for left closed covers (\cite[Corollary 2.32]{ClausenOrsnesJansen}) that $|F_{/(\sigma\leq \tau)}|\simeq \ast$ as desired: indeed, the collection $\mathcal{P}$ of non-empty intersections $A_I$ viewed as a poset under inclusion has an initial object and thus
\begin{align*}
|F_{/(\sigma\leq \tau)}|\simeq \mathop{\varinjlim}_{I\in \mathcal{P}} |A_I|\simeq  \mathop{\varinjlim}_{\mathcal{P}} \ast \simeq |\mathcal{P}|\simeq \ast
\end{align*}
which is what we set out to prove.
\end{proof}

This proposition allows us to decompose limits as in the corollary below.

\begin{corollary}\label{functors out of twisted arrow category}
Let $\theta\in \mathfrak{C}_S$. For any functor $F\colon \operatorname{Tw}(\operatorname{CL}_S)^\theta\rightarrow \mathcal{C}$ to an arbitrary $\infty$-category $\mathcal{C}$ with all limits, we have
\begin{enumerate}
\item 
$\displaystyle
\varprojlim F\xrightarrow{\ \sim\ } \varprojlim F\vert_{\operatorname{Tw}(\operatorname{CL}_S)^\theta\smallsetminus B\Gamma_{S/\theta}}\mathop{\times}_{\varprojlim F\vert_{(\Gamma_{S/\theta}\backslash\backslash \mathfrak{C}_{S/\theta})^{op}\smallsetminus B\Gamma_{S/\theta}}} \varprojlim F\vert_{(\Gamma_{S/\theta}\backslash\backslash \mathfrak{C}_{S/\theta})^{op}},
$
 
\  
 
and \ 
$
\varprojlim F\vert_{(\Gamma_{S/\theta}\backslash\backslash \mathfrak{C}_{S/\theta})^{op}}\xrightarrow{\ \sim \ } \varprojlim F\vert_{B\Gamma_{S/\theta}}
$.
\item[] 
\item 
$\displaystyle
\varprojlim F\vert_{\operatorname{Tw}(\operatorname{CL}_S)^\theta\smallsetminus B\Gamma_{S/\theta}} \xrightarrow{\ \sim \ } \varprojlim_{[\nu]\in (\Gamma_S\backslash \mathfrak{C}_S)_{\leq [\theta]}^{op}\smallsetminus [\theta]}\varprojlim F\vert_{\operatorname{Tw}(\operatorname{CL}_S)_{\leq [\nu]}^\theta},
$
 
\  
 
and \ 
$
\varprojlim F\vert_{\operatorname{Tw}(\operatorname{CL}_S)^\theta_{\leq [\nu]}}\xrightarrow{\ \sim \ } \varprojlim F\vert_{\operatorname{Tw}(\operatorname{CL}_S)^\nu}
$ for all $\nu\in (\mathfrak{C}_S)_{\leq \theta}$.
\end{enumerate}
\end{corollary}
\begin{proof}
This follows from parts 1 and 2 of \Cref{decomposition of twisted arrow category} above and descent for left closed covers applied to $(\operatorname{Tw}(\operatorname{CL_S})^\theta)^{op}$, and the corresponding decomposition of limits as in \cite[Proposition 2.1]{ClausenOrsnesJansen}. The additional equivalence in part 1 is a consequence of part 3 and the fact that a left adjoint is a $\varprojlim$-equivalence; the additional equivalence in part 2 is a consequence of parts 4 and 5.
\end{proof}

\begin{remark}\label{key technical difference RBS versus DMK}
The proof of part 4 of \Cref{decomposition of twisted arrow category} above is perhaps the key technical difference between the case of the Deligne--Mumford--Knudsen compactification and the reductive Borel--Serre compactification. In \cite[Proposition 4.21 (4)]{ClausenOrsnesJansen}, the analogous statement says that the map in question is a left adjoint. The argument above would in fact go through in that case, but the resulting cover of the fibre has only one element; in other words, the fibre has a terminal object, equivalently the functor is a left adjoint. What we are seeing diagrammatically here is the failure of the closure of a stratum in $\overline{\mathcal{M}}_{\mathbf{G}}$ to really be a ``smaller'' instance of a Deligne--Mumford--Knudsen compactification; instead, one needs to pass to its normalisation (\Cref{DMK compactification as normalisation}). In the case of the reductive Borel--Serre compactification, the closure of a stratum \textit{is} a smaller reductive Borel--Serre compactification, namely the one associated to the Levi quotient of the corresponding parabolic subgroup. In terms of the action of the mapping class group $\Gamma_S$ on the Harvey bordification of Teichmüller space $\overline{\mathcal{T}}_{\!\!S}$, this difference exhibits itself as follows: an element $\gamma\in \Gamma_S$ may act on the closure of the $\sigma$-stratum $\overline{X}_\sigma^S\subset \overline{\mathcal{T}}_{\!\!S}$ even though it does not belong to the stabiliser $\Gamma_S(\sigma)$. On the $\sigma$-stratum $X_\sigma^S$ itself, only the stabiliser will act, but elements not belonging to the stabiliser may identify points on the boundary, so we cannot just consider the action of $\Gamma_S(\sigma)$ on $\overline{X}_\sigma^S$, we must remember the bigger mapping class group. For the Charney--Lee categories, this fact translates into the observation made in \Cref{category of stable curves remark} (3): the Charney--Lee category $\operatorname{CL}_{S/\sigma}$ associated to the stable nodal surface $(S/\sigma, P)$ does \textit{not} agree with the full subcategory of $\operatorname{CL}_S$ spanned by the objects $(\mathfrak{C}_S)_{\leq \sigma}$.
\end{remark}

\section{Determining the stratified homotopy type}

We are now all set to identify the stratified homotopy type of the Deligne--Mumford--Knudsen compactifications. The idea is to build the stack out of pieces that are as \textit{simple} as possible from the point of view of stratified homotopy theory in the following sense: the exit path $\infty$-category of each ``building block'' should identify with its stratifying poset, the analogue of being contractible for stratified spaces. When we here say build, we mean in terms of limits of sheaf categories (cf. \Cref{sheaves on Mg bar as limit of simple pieces}). This will exhibit the exit path $\infty$-category as a colimit of posets which we can directly identify.

\subsection{The exit path category as a colimit of posets}

We refer the reader to \cite{OrsnesJansen23} for the details on exit path $\infty$-categories of stratified topological ($\infty$-)stacks, but let us recall the technical definition here.

\begin{definition}\label{definition exit path category}
Let $P$ be a poset satisfying the ascending chain condition. We say that a stratified étale $\infty$-stack $(X,P)$ \textit{admits an exit path $\infty$-category} if the following conditions hold:
\begin{enumerate}
\item The full subcategory of $P$-constructible $\operatorname{Shv}^{\operatorname{cbl}}_P(X)\subset\operatorname{Shv}(X)$ is closed under all limits and colimits.
\item The $\infty$-category $\operatorname{Shv}^{\operatorname{cbl}}_P(X)$ is generated under colimits by a set of atomic objects (that is, objects $x$ such that the functor $\operatorname{Map}(x,-)$ preserves all colimits).
\item The pullback $\pi^*\colon \operatorname{Fun}(P, \mathcal{S})\rightarrow \operatorname{Shv}^{\operatorname{cbl}}_P(X)$ preserves all limits (and colimits, but that is automatic).
\end{enumerate}
If $(X, P)$ admits an exit path $\infty$-category, we define its \textit{exit path $\infty$-category} to be the opposite category of the full subcategory of atomic constructible sheaves:
\begin{equation*}
\Pi(X,P):=\left(\operatorname{Shv}^{\operatorname{cbl}}_P(X)^{atom}\right)^{\operatorname{op}}.\qedhere
\end{equation*}
\end{definition}

If $(X,P)$ admits an exit path $\infty$-category, then it follows that there is an induced ``exodromy'' equivalence (\cite[Proposition 3.9]{ClausenOrsnesJansen}, cf.~\cite{BarwickGlasmanHaine} for the terminology)
\begin{align*}
\operatorname{Fun}(\Pi(X,P),\mathcal{S})\xrightarrow{\ \sim \ } \operatorname{Shv}^{\operatorname{cbl}}_P(X).
\end{align*}

So essentially, the definition is the following: a stratified topological ($\infty$-)stack $(X,P)$ whose stratifying poset $P$ satisfies the ascending chain condition \textit{admits an exit path $\infty$-category} if there exists an idempotent complete $\infty$-category $\Pi(X,P)$ and an equivalence
\begin{align*}
\operatorname{Shv}^{\operatorname{cbl}}_P(X)\xrightarrow{\sim} \operatorname{Fun}(\Pi(X,P),\mathcal{S})
\end{align*}
between the $\infty$-category of space-valued $P$-constructible sheaves on $X$ and the $\infty$-category of functors $\Pi(X,P)\rightarrow \mathcal{S}$ (with a few additional technical assumptions).

\medskip

Let us first of all identify the stratified homotopy types of the Harvey bordification and compactification.

\begin{proposition}\label{exit path category of Harvey bordification}
For a stable nodal surface $(S,P)$, the Harvey bordification $\overline{\mathcal{T}}_{S}$ admits an exit path $\infty$-category which identifies with its stratifying poset $\mathfrak{C}_{S}$. In particular, for any $\sigma\in \mathfrak{C}_{S}$, the closure of the $\sigma$-stratum $\overline{X}_\sigma^{S}$ admits an exit path $\infty$-category which identifies with its stratifying poset $(\mathfrak{C}_S)_{\leq \sigma}$.
\end{proposition}
\begin{proof}
This is an application of \cite[Corollary 3.7]{OrsnesJansen}. It can also be proved in the spirit of \cite{ClausenOrsnesJansen} by following the strategy of Corollary 4.6 of that paper.
\end{proof}

The following corollary identifies the exit path $\infty$-category of the Harvey compactification (it is the stack analogue of \cite[Theorem 4.2]{OrsnesJansen} or \cite[Corollary 4.9]{ClausenOrsnesJansen} in the reductive Borel--Serre case). 

\begin{corollary}\label{exit path category of Harvey compactification}
For a stable nodal surface $(S,P)$, the Harvey compactification $[\Gamma_{S}\backslash\overline{\mathcal{T}}_{S}]$ admits an exit path $\infty$-category which identifies with the action category $\Gamma_{S}\backslash\backslash\mathfrak{C}_{S}$.
\end{corollary}
\begin{proof}
This is a direct consequence of \Cref{exit path category of Harvey bordification} above and \cite[Corollary 3.16]{OrsnesJansen23}.
\end{proof}

We now exploit the inductive nature of the Deligne--Mumford--Knudsen compactifications exhibited in the previous section to provide a refined decomposition of the relevant sheaf categories. The identification of the exit path $\infty$-categories will follow directly from this by applying \cite[Proposition 3.13]{OrsnesJansen23} (cf. \Cref{the exit path category of Mg bar} below).

For $(S,P)$ a stable nodal surface with a fixed hyperbolic metric, define a functor
\begin{align*}
\operatorname{Tw}(\operatorname{CL}_S)^{op}\rightarrow \operatorname{Strat}
\end{align*}
into stratified topological spaces (viewed as stratified topological stacks) by sending $(\sigma\leq \tau)$ to
\begin{align*}
\overline{X}_\sigma^{S/\tau}\rightarrow (\mathfrak{C}_{S/\tau})_{\leq \sigma},
\end{align*}
i.e.~the closure of the $\sigma$-stratum of the Harvey bordification associated to $(S/\tau,P)$ stratified over $(\mathfrak{C}_{S/\tau})_{\leq \sigma}\cong (\mathfrak{C}_S)_{\leq \sigma}$. On morphisms, the functor sends $[\gamma]\colon (\sigma\leq \tau)\rightarrow (\sigma'\leq \tau')$ as described in \Cref{alternative description of twisted arrow category} to the induced map
\begin{align*}
\overline{X}_{\sigma'}^{S/\tau'}\rightarrow \overline{X}_\sigma^{S/\tau}
\end{align*}
given by \Cref{maps from twisted arrow category} and \Cref{identifying quotients by additive action as smaller Harvey bordifications}. Taking sheaves (or constructible sheaves) with pullback functoriality, we get a functor
\begin{align*}
\operatorname{Tw}(\operatorname{CL}_S)\rightarrow \operatorname{Cat}_\infty.
\end{align*}

The identification of the exit path $\infty$-category will be a consequence of the following theorem.

\begin{theorem}\label{sheaves on Mg bar as limit of simple pieces}
Let $(S,P)$ be a stable nodal surface with associated dual graph $\mathbf{G}=\mathbf{G}(S,P)$ and a fixed hyperbolic metric. Then
\begin{align*}
\operatorname{Shv}(\overline{\mathcal{M}}_{\mathbf{G}}^{\operatorname{top}})\xrightarrow{\ \sim\ } \mathop{\varprojlim}_{(\sigma\leq \tau)\in \operatorname{Tw}(\operatorname{CL}_{S})} \operatorname{Shv}(\overline{X}^{S/\tau}_\sigma)
\end{align*}
and similarly for constructible sheaves.
\end{theorem}

This theorem is a special case of the following more general statement for the closure of each stratum in $\overline{\mathcal{M}}^{\operatorname{top}}_{\mathbf{G}}$; the theorem above is then the case of the open stratum, i.e.~$\theta=\emptyset$. We have to prove such a refined statement because of the observations made in \Cref{key technical difference RBS versus DMK}.

\begin{theorem}\label{sheaves on Mg bar as limit of simple pieces - stratumwise}
Let $(S,P)$ be a stable nodal surface with a fixed hyperbolic metric and let $\theta\in \mathfrak{C}_{S}$. Then
\begin{align*}
\Psi_\theta\colon \operatorname{Shv}(\overline{\mathcal{D}}_{\theta})\xrightarrow{\ \sim\ } \mathop{\varprojlim}_{(\sigma\leq \tau)\in \operatorname{Tw}(\operatorname{CL}_{S})^\theta} \operatorname{Shv}(\overline{X}^{S/\tau}_\sigma)
\end{align*}
and similarly for constructible sheaves.
\end{theorem}

\begin{proof}
We will prove this by induction on the decreasing number of closed simple curves in $\theta$, i.e.~the depth of the strata. In the case of a stratum of maximal depth associated to a maximal curve system $\theta$, the claim follows from the following observations
\begin{itemize}
\item $\overline{\mathcal{D}}_{\theta}=\mathcal{D}_{\theta}\cong [\Gamma_{S/\theta}\backslash \mathcal{T}_{S/\theta}]$ and $\mathcal{T}_{S/\theta}=\overline{X}^{S/\theta}_\theta$;
\item the category $\operatorname{Tw}(\operatorname{CL}_S)^\theta$ identifies with the full subcategory $B\Gamma_{S/\theta}$ on the object $(\theta\leq\theta)$.
\end{itemize}

Now suppose $\theta$ consists of $M$ curves and assume that the claim holds for any isotopy class of admissible curve systems with at least $M+1$ curves. We consider the functor 
\begin{align*}
F\colon \operatorname{Tw}(\operatorname{CL}_S)^\theta\rightarrow \operatorname{Cat}_\infty
\end{align*}
sending $(\sigma\leq \tau)$ to $\operatorname{Shv}(\overline{X}^{S/\tau}_\sigma)$ with pullback functoriality and exploit the decompositions of limits of such functors established in the previous section. 

First of all, we observe that the map $\Psi_\theta$ restricts to an equivalence on the boundary $\partial \overline{\mathcal{D}}_\theta$ by restricting the functor $F$ to $\operatorname{Tw}(\operatorname{CL}_S)^\theta\smallsetminus B\Gamma_{S/ \theta}$: indeed, by comparing part 2 of \Cref{preliminary decomposition of sheaves on Mg bar} (descent for closed covers) with part 2 of \Cref{functors out of twisted arrow category} and exploiting the inductive hypothesis, we get the an equivalence 
\begin{align*}
\operatorname{Shv}(\partial\overline{\mathcal{D}}_{\theta})\xrightarrow{\ \sim\ } \mathop{\varprojlim}_{(\sigma\leq \tau)\in \operatorname{Tw}(\operatorname{CL}_S)^\theta\smallsetminus B\Gamma_{S/\theta}} \operatorname{Shv}(\overline{X}^{S/\tau}_\sigma).
\end{align*}

Now, we show that we get a similar decomposition on the boundary of the Harvey compactification $[\Gamma_{S/\theta}\backslash \overline{\mathcal{T}}_{S/\theta}]$ by restricting the functor $F$ to $(\Gamma_{S/\theta}\backslash\backslash \mathfrak{C}_{S/\theta})^{op}\smallsetminus B\Gamma_{S/\theta}$. By \cite[Corollary 2.34]{ClausenOrsnesJansen}, we have
\begin{align*}
(\Gamma_{S/\theta}\backslash \backslash \mathfrak{C}_{S/\theta})^{op}\smallsetminus B\Gamma_{S/\theta}\simeq \mathop{\varinjlim}_{B\Gamma_{S/\theta}}(\mathfrak{C}_{S/\theta}^{op}\smallsetminus \emptyset).
\end{align*}
Combining the subsequent limit decomposition (\cite[Proposition 2.1]{ClausenOrsnesJansen}) with the identification of the boundary as $\partial[\Gamma_{S/\theta}\backslash \overline{\mathcal{T}}_{S/\theta}]\cong [\Gamma_{S/\theta}\backslash  \partial \overline{\mathcal{T}}_{S/\theta}]$ and descent for the closed cover of the boundary by closures of strata, we get an equivalence
\begin{align*}
\operatorname{Shv}(\partial[\Gamma_{S/\theta}\backslash \overline{\mathcal{T}}_{S/\theta}])\xrightarrow{\ \sim \ }\mathop{\varprojlim}_{\sigma\in (\Gamma_{S/\theta}\backslash \backslash \mathfrak{C}_{S/\theta})^{op}\smallsetminus B\Gamma_{S/\theta}} \operatorname{Shv}(\overline{X}^{S/\theta}_\sigma).
\end{align*}

We also get an equivalence of sheaf categories on the Harvey compactification by restricting the functor $F$ to $(\Gamma_{S/\theta}\backslash\backslash \mathfrak{C}_{S/\theta})^{op}$ since the inclusion $B\Gamma_{S/\theta}\hookrightarrow (\Gamma_{S/\theta}\backslash\backslash \mathfrak{C}_{S/\theta})^{op}$ is a left adjoint and hence a $\varprojlim$-equivalence (\Cref{decomposition of twisted arrow category} part 3):
\begin{align*}
\operatorname{Shv}([\Gamma_{S/\theta}\backslash \overline{\mathcal{T}}_{S/\theta}])\xrightarrow{\ \sim\ } \mathop{\varprojlim}_{B\Gamma_{S/\theta}}\operatorname{Shv}(\overline{\mathcal{T}}_{S/\theta})\xrightarrow{\ \sim\ } \mathop{\varprojlim}_{\sigma\in (\Gamma_{S/\theta}\backslash\backslash \mathfrak{C}_{S/\theta})^{op}}\operatorname{Shv}(\overline{X}^{S/\theta}_\sigma).
\end{align*}

Finally, by comparing part 1 of \Cref{preliminary decomposition of sheaves on Mg bar} with part 1 of \Cref{functors out of twisted arrow category}, we see that the above equivalences assemble to yield the desired equivalence, without the constructibility condition.

To prove the claim for constructible sheaves, consider instead the functor of constructible sheaves
\begin{align*}
\operatorname{Tw}(\operatorname{CL}_S)\rightarrow \operatorname{Cat}_\infty,\quad (\sigma\leq \tau)\mapsto\operatorname{Shv}^{\operatorname{cbl}}(\overline{X}^{S/\tau}_\sigma)
\end{align*}
with pullback functoriality. We see that the preceding inductive argument goes through for this functor as well; the only additional observation to be made is that for an action of a group $G$ on a stratified space $X$, the equivalence $\operatorname{Shv}([G\backslash X])\simeq \varprojlim_{BG}\operatorname{Shv}(X)$ also goes through for constructible sheaves as the restriction of the quotient map $X\rightarrow [G\backslash X]$ to each stratum admits local sections (this was also used implicitly in \Cref{exit path category of Harvey compactification} above, see \cite[Corollary 3.16]{OrsnesJansen23}).
\end{proof}

We now identify the exit path $\infty$-category of the stratified topological Deligne--Mumford--Knudsen compactification (this calculation is completely analogous to the one in \cite[Corollary 4.20]{ClausenOrsnesJansen}). 

\begin{corollary}\label{the exit path category of Mg bar}
Let $\mathbf{G}$ be a stable $P$-pointed dual graph. The stratified topological stack $(\overline{\mathcal{M}}^{\operatorname{top}}_{\mathbf{G}},\mathcal{G}_{\mathbf{G}})$ admits an exit path $\infty$-category and this exit path $\infty$-category identifies with the Charney--Lee category $\operatorname{CL}_S$ for any stable nodal surface $(S,P)$ whose associated dual graph is isomorphic to $\mathbf{G}$.

In particular, if $\mathbf{G}=(g,P)$ is the one vertex graph of genus $g$, then the exit path $\infty$-category identifies with the category $\mathfrak{SC}_{g,P}^{\operatorname{op}}$ of stable curves and (opposite) deformations (\Cref{category of stable curves}).
\end{corollary}
\begin{proof}
Let $(S,P)$ be a stable nodal surface with dual graph $\mathbf{G}(S,P)\cong \mathbf{G}$. For any $\sigma\leq \tau$, the exit path $\infty$-category of $\overline{X}^{S/\tau}_\sigma$ identifies with its stratifying poset $(\mathfrak{C}_S)_{/\sigma}$ (\Cref{exit path category of Harvey bordification}). Thus, by \cite[Proposition 3.13]{OrsnesJansen23} and \Cref{sheaves on Mg bar as limit of simple pieces}, it suffices to calculate that in $\operatorname{Cat}_\infty$, we have
\begin{align*}
\varinjlim_{(\sigma\leq \tau)\in \operatorname{Tw}(\operatorname{CL}_S)^{op}}(\mathfrak{C}_S)_{/\sigma}\simeq \operatorname{CL}_S.
\end{align*}
For any $\sigma\in \mathfrak{C}_S$, the natural functor $(\mathfrak{C}_S)_{/\sigma}\rightarrow (\operatorname{CL}_S)_{/\sigma}$ sending $\tau\leq \sigma$ to $[\operatorname{id}]\colon \tau\rightarrow \sigma$ is an equivalence by \Cref{equivalence of comma categories}. The claim then follows from the following two facts:
\begin{enumerate}
\item for any $1$-category $\mathcal{C}$, the projection functor $\operatorname{Tw}(\mathcal{C})^{op}\rightarrow \mathcal{C}$ sending $x\rightarrow y$ to $x$ is a $\varinjlim$-equivalence (\cite[Lemma 4.16]{ClausenOrsnesJansen}),
\item for any small $\infty$-category $\mathcal{C}$, we have $\varinjlim_{c\in \mathcal{C}}\mathcal{C}_{/c}\xrightarrow{\sim}\mathcal{C}$ (\cite[Corollary 2.35]{ClausenOrsnesJansen}).
\end{enumerate}
The final claim for non-singular $S$ follows from the equivalence $\mathfrak{SC}_{g,P}^{\operatorname{op}}\simeq \operatorname{CL}_S$ (\Cref{category of stable curves vs Charney-Lee category}).
\end{proof}

\subsection{Consequences}\label{consequences}

In this section, we review some immediate consequences of our calculation. We also get analogous results for the Harvey bordification and compactification using \Cref{exit path category of Harvey bordification} and \Cref{exit path category of Harvey compactification}, but we leave the details to the reader.

By definition of exit path $\infty$-categories, we get an ``exodromy'' equivalence, that is, a classification of constructible sheaves as representations of the exit path $\infty$-category (\cite[Proposition 3.22]{OrsnesJansen23}).

\begin{theorem}\label{exodromy equivalence}
Let $(S,P)$ be a stable $P$-pointed nodal surface with dual graph $\mathbf{G}$ and let $\mathcal{C}$ be a compactly generated $\infty$-category. There is an equivalence of $\infty$-categories
\begin{align*}
\operatorname{Shv}^{\operatorname{cbl}}(\overline{\mathcal{M}}^{\operatorname{top}}_{\mathbf{G}};\mathcal{C})\xrightarrow{\ \sim\ }  \operatorname{Fun}(\operatorname{CL}_S,\mathcal{C}).
\end{align*}
In particular, for any $g$ such that $2g+2-|P|>0$, we have an equivalence
\begin{align*}
\operatorname{Shv}^{\operatorname{cbl}}(\overline{\mathcal{M}}^{\operatorname{top}}_{g,P};\mathcal{C})\xrightarrow{\ \sim\ }  \operatorname{Fun}(\mathfrak{SC}_{g,P}^{op}, \mathcal{C}).
\end{align*}
\end{theorem}

In the case at hand, the exodromy equivalence above allows us to identify the constructible derived category as a derived functor category as below --- this is a consequence of the fact that the exit path $\infty$-category is in fact a $1$-category. We refer to \cite{OrsnesJansen} and \cite{OrsnesJansen23} for details. For an associative ring $R$, consider the derived $\infty$-category of left $R$-modules $\operatorname{LMod}_R\simeq\mathscr{D}(R)$, the $1$-category of sheaves taking values in the $1$-category, $\operatorname{LMod}_R^1$, of left $R$-modules,
\begin{align*}
\operatorname{Shv}_1(\overline{\mathcal{M}}^{\operatorname{top}}_{\mathbf{G}};R),
\end{align*}
and finally consider the following full subcategories of the derived $\infty$-category of sheaves:
\begin{align*}
\mathscr{D}(\operatorname{Shv}_1(\overline{\mathcal{M}}^{\operatorname{top}}_{\mathbf{G}};R))\supseteq \mathscr{D}_{\operatorname{cbl}}(\operatorname{Shv}_1(\overline{\mathcal{M}}^{\operatorname{top}}_{\mathbf{G}};R))\supseteq \mathscr{D}^b_{\operatorname{cbl}}(\operatorname{Shv}_1(\overline{\mathcal{M}}^{\operatorname{top}}_{\mathbf{G}};R));
\end{align*}
the category in the middle is spanned by complexes whose homology sheaves are constructible (with respect to the fixed stratification over the poset of dual graphs), and on the right hand side, we restrict further to the bounded constructible derived category spanned by those complexes which moreover have perfect stalk complexes.

By the discussion in \cite[\S 3.4]{OrsnesJansen23}, the canonical functor
\begin{align*}
\mathscr{D}(\operatorname{Shv}_1(\overline{\mathcal{M}}^{\operatorname{top}}_{\mathbf{G}};R))\longrightarrow \operatorname{Shv}(\overline{\mathcal{M}}^{\operatorname{top}}_{\mathbf{G}};\operatorname{LMod}_R)
\end{align*}
restricts to equivalences
\begin{align*}
&\mathscr{D}_{\operatorname{cbl}}(\operatorname{Shv}_1(\overline{\mathcal{M}}^{\operatorname{top}}_{\mathbf{G}};R))\xrightarrow{\sim}\operatorname{Shv}^{\operatorname{cbl}}(\overline{\mathcal{M}}^{\operatorname{top}}_{\mathbf{G}};\operatorname{LMod}_R) \\
&\mathscr{D}^b_{\operatorname{cbl}}(\operatorname{Shv}_1(\overline{\mathcal{M}}^{\operatorname{top}}_{\mathbf{G}};R))\xrightarrow{\sim}\operatorname{Shv}^{\operatorname{cbl,cpt}}(\overline{\mathcal{M}}^{\operatorname{top}}_{\mathbf{G}};\operatorname{LMod}_R)
\end{align*}
where the right hand $\infty$-category on the lower line is the full subcategory of constructible compact-valued sheaves, that is, it is spanned by the constructible sheaves whose stalks are compact objects in $\operatorname{LMod}_R$, i.e.~perfect complexes.

\begin{theorem}\label{equivalence of derived infinity-categories}
Let $(S,P)$ be a stable nodal surface with dual graph $\mathbf{G}$ and let $R$ be an associative ring. There is an equivalence
\begin{align*}
\operatorname{Shv}^{\operatorname{cbl}}(\overline{\mathcal{M}}^{\operatorname{top}}_{\mathbf{G}};\operatorname{LMod}_R)\xrightarrow{\sim}\mathscr{D}\big(\operatorname{Fun}(\operatorname{CL}_S,\operatorname{LMod}_R^1)\big),
\end{align*}
which restricts to an equivalence
\begin{align*}
\operatorname{Shv}^{\operatorname{cbl,cpt}}(\overline{\mathcal{M}}^{\operatorname{top}}_{\mathbf{G}};\operatorname{LMod}_R)\xrightarrow{\sim} \mathscr{D}^{\operatorname{cpt}}\big(\operatorname{Fun}(\operatorname{CL}_S,\operatorname{LMod}_R^1)\big),
\end{align*}
where the right hand side denotes the full subcategory of complexes of functors complexes of functors $F_\bullet$ into the $1$-category of left $R$-modules satisfying that $F_\bullet(\sigma)$ is a perfect complex for every object $\sigma$ of $\operatorname{CL}_S$.
\end{theorem}
\begin{proof}
Since the exit path $\infty$-category of $\overline{\mathcal{M}}_\mathbf{G}^{\operatorname{top}}$ is a $1$-category, \cite[Proposition 3.23]{OrsnesJansen23} applies.
\end{proof}

Taking homotopy categories, one gets an analogous $1$-categorical statement (see \cite[Corollary 3.24]{OrsnesJansen23}).

As remarked in the introduction, our calculation is a refinement of the work of Charney--Lee (\cite{CharneyLee84}) and Ebert--Giansiracusa (\cite{EbertGiansiracusa}). This is a consequence of the following corollary (and the succeeding remark).

\begin{corollary}\label{homotopy type}
Let $(S,P)$ be a stable nodal surface with dual graph $\mathbf{G}$. The moduli stack $\overline{\mathcal{M}}^{\operatorname{top}}_{\mathbf{G}}$ admits a fundamental $\infty$-groupoid $\Pi_{\mathbf{G}}$ and the functor
\begin{align*}
\operatorname{CL}_S\rightarrow \Pi_{\mathbf{G}}
\end{align*}
is a weak homotopy equivalence. In particular, for any $g$ such that $2g+2-|P|>0$, there is a weak homotopy equivalence $\mathfrak{SC}_{g,P}\rightarrow \Pi(\overline{\mathcal{M}}^{\operatorname{top}}_{g,P})$.
\end{corollary}
\begin{proof}
First of all note that the stack admits an fundamental $\infty$-groupoid by combining \Cref{sheaves on Mg bar as limit of simple pieces} and \cite[Proposition 3.13]{OrsnesJansen23} with the fact that each $\overline{X}_\sigma^{S/\tau}$ admits a fundamental $\infty$-groupoid. The functor $\operatorname{CL}_S\rightarrow \Pi_{\mathbf{G}}$ is then a weak homotopy equivalence by \cite[Corollary 3.21]{OrsnesJansen23}. We could also use \cite[Proposition 3.13]{OrsnesJansen23} to directly calculate that 
\begin{align*}
\Pi_\mathbf{G}\simeq |\operatorname{CL}_S|,
\end{align*}
as each $\overline{X}_\sigma^{S/\tau}$ is contractible and $\varinjlim_{\mathcal{C}}\ast\simeq |\mathcal{C}|$ for any $\infty$-category $\mathcal{C}$ (\cite[Corollary 2.10]{ClausenOrsnesJansen}).
\end{proof}

\begin{remark}\label{EG-functoriality of weak equivalence}
By \cite[Proposition 3.13]{OrsnesJansen23}, the weak homotopy equivalence of the above corollary is functorial with respect to inclusions of strata and closures of strata in the sense of \cite[\S 7]{EbertGiansiracusa}.
\end{remark}

\appendix

\section{Dual graphs}\label{dual graphs}

In this appendix, we introduce the notion of a stable dual graph which enters into the story via the natural stratification of the Deligne--Mumford--Knudsen compactification.

By a graph, we mean a finite connected (multi-)graph, that is, we allow multiple edges and we will also allow half-edges. More precisely, a \textit{graph} $G=(V,E,\iota)$ consists of the following data:
\begin{itemize}
\item a set of \textit{vertices} $V$;
\item a set of \textit{half-edges} $E=\coprod_{v\in V}E_v$ equipped with a partition indexed by $V$;
\item an involution $\iota\colon E\rightarrow E$.
\end{itemize}
The pairs of half-edges interchanged by $\iota$ are called \textit{edges}, the set of half-edges fixed by $\iota$ are called \textit{legs}, and a half-edge $e\in E$ is said to be \textit{incident to $v$} if it belongs to $E_v$. We will generally omit $\iota$ from the notation.

\begin{definition}
A \textit{dual graph} is a pair $(G,\omega)$ where:
\begin{enumerate}[label=(\roman*)]
\item $G=(V,E)$ is a graph as above;
\item the \textit{weight function} $\omega\colon V\rightarrow \Z_{\geq 0}$, $v\mapsto g_v$, is any function on the set of vertices.
\end{enumerate}

The \textit{genus} of a dual graph $(G,\omega)$ is
\begin{align*}
g(G,\omega)=1-\chi(G)+\sum_{v\in V} g_v.
\end{align*}

For a finite set $P$, a \textit{$P$-pointed dual graph} is a triple $(G,\omega,m)$ where $(G,\omega)$ is a dual graph as above and the \textit{$P$-marking} $m\colon P\rightarrow E$ is a bijective labelling of the legs of the graph.

For any vertex $v\in V$, we write $n_v=|E_v|$, that is, the number of half-edges which are incident to $v$. A dual graph $(G,\omega,m)$ is said to be \textit{stable} if
\begin{align*}
2g_v-2 + n_v>0\quad \text{for every vertex }v\in V.
\end{align*}

An \textit{isomorphism} of $P$-pointed dual graphs is given by an isomorphism of graphs such that the bijection on vertex sets preserves the additional structure given by the weight functions and $P$-markings. We will also write $\mathbf{G}=(G,\omega,m)$ to simplify notation when we do not need to specify the data of the dual graph explicitly.
\end{definition}

\Cref{fig:stable dual graphs} shows examples of stable $\{a,b,c\}$-pointed dual graphs of genus $6$.

\begin{figure}[H]
\centering
\begin{tikzpicture}[node distance={15mm}, thick, main/.style = {draw, circle}] 
\node[main] (1) {$1$}; 
\node[main] (2) [right of=1] {$1$}; 
\node[main] (3) [below of=1] {$0$}; 
\node[main] (4) [below of=2] {$0$};
\node (5) [above right of=2] {$a$};
\node (6) [left of=3] {$b$};
\node (7) [below right of=4] {$c$};
\draw (1) -- (2); 
\draw (1) -- (3);
\draw (2) -- (4);
\draw (3) -- (4);
\draw (2) -- (5);
\draw (3) -- (6);
\draw (4) -- (7);
\draw (1) to [out=175,in=95,looseness=7] (1);
\draw (3) to [out=230,in=310,looseness=7] (3);
\draw (4) to [out=65,in=345,looseness=7] (4);
\end{tikzpicture}
\begin{tikzpicture}[node distance={15mm}, thick, main/.style = {draw, circle}] 
\node[main] (1) {$1$}; 
\node[main] (2) [right of=1] {$1$}; 
\node[main] (3) [below of=2] {$2$};
\node (4) [above right of=2] {$a$};
\node (5) [below left of=3] {$b$};
\node (6) [below right of=3] {$c$};
\draw (1) -- (2); 
\draw (1) -- (3);
\draw (2) -- (3);
\draw (2) -- (4);
\draw (3) -- (5);
\draw (3) -- (6);
\draw (1) to [out=175,in=95,looseness=7] (1);
\end{tikzpicture}\quad
\begin{tikzpicture}[node distance={15mm}, thick, main/.style = {draw, circle}] 
\node[main] (1) {$5$};
\node (2) [above right of=1] {$a$};
\node (3) [right of=1] {$b$};
\node (4) [below right of=1] {$c$};
\node (5) [below of=1] {$ $};
\draw (1) -- (2); 
\draw (1) -- (3);
\draw (1) -- (4);
\draw (1) to [out=175,in=95,looseness=7] (1);
\end{tikzpicture}\qquad\qquad
\begin{tikzpicture}[node distance={15mm}, thick, main/.style = {draw, circle}] 
\node[main] (1) {$6$};
\node (2) [above right of=1] {$a$};
\node (3) [right of=1] {$b$};
\node (4) [below right of=1] {$c$};
\node (5) [below of=1] {$ $};
\draw (1) -- (2); 
\draw (1) -- (3);
\draw (1) -- (4);
\end{tikzpicture}
\caption{Examples of stable $\{a,b,c\}$-pointed dual graphs of genus $6$. From left to right we refer to these by (i), (ii), (iii), (iv) in future examples.}
\label{fig:stable dual graphs}
\end{figure}
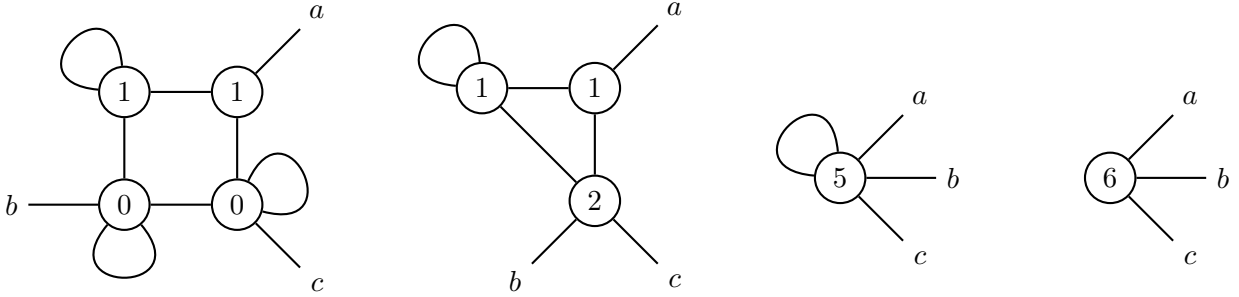

We go on to define a contraction of a dual graph. See also \Cref{contraction of graph intuition} below for intuition and examples.

\begin{construction}\label{contraction of graph}
Let $\mathbf{G}=(G,\omega,m)$ be a $P$-pointed dual graph with $G=(V,E,\iota)$ and let $I=(V,E')$ be a subgraph of $G$ containing all the vertices and no legs; $I$ is completely determined by the subset $E'\subset E$ of edges, that is, $E'$ is a subset of the half-edges containing no legs and such that if $e\in E'$ then also $\iota(e)\in E'$. Denote by $\mathbf{G}_I=(G_I,\omega_I,m_I)$ the dual graph obtained by collapsing the edges in $I$. More precisely $\mathbf{G}_I=(G_I,\omega_I,m_I)$ is defined as follows: let $W$ denote the set of connected components of $I$ and define the graph $G_I=(V_I,E_I)$ with
\begin{align*}
V_I=W,\quad\text{and}\quad E_I=E\smallsetminus E'
\end{align*}
such that the half-edges of $E_I$ that are incident to $w\in W$ are exactly those half-edges of $E\smallsetminus E'$ that are incident to some vertex of the connected component $I_w$ corresponding to $w$:
\begin{align*}
(E_I)_w=\{e\in E\smallsetminus E'\mid e\in E_v \text{ for some } v\in I_w\}.
\end{align*}
The weight function is given by
\begin{align*}
w_I\colon V_I\rightarrow \Z_{\geq 0},\quad w\mapsto g(I_w,\omega\vert_{I_w});
\end{align*}
that is, it sends $w$ to the genus of the component $I_w$ equipped with the restricted weight function of $G$. Finally, the marking $m_I\colon P\rightarrow E_I$ is the factorisation of $m$ through the inclusion $E\smallsetminus E'\subset E$.
\end{construction}

\begin{remark}\label{contraction of graph intuition}
If unfamiliar with the construction above, it may be helpful to think of the contraction in terms of the geometric realisations of the graphs: given a dual graph $\mathbf{G}=(G,\omega,m)$ and a subgraph $I$ of $G$ as above, there is a continuous map $|G|\rightarrow|G_I|$ given by contracting each component $|I_w|\subset |G|$ to the vertex $w$ in $|G_I|$. The labelling of the legs stays fixed since $I$ is assumed to contain no legs. In \Cref{fig:stable dual graphs}, we see that the graph (ii) is a contraction of the graph (i), where the subgraph $I$ is the one given by the lower level of the graph. In fact, each graph in \Cref{fig:stable dual graphs} is a contraction of the graphs to the left of it. See also \cite[Chapter XII, \S 10]{ArbarelloCornalbaGriffiths}.
\end{remark}

\begin{definition}
For $\mathbf{G}$ and $I$ as above, we call $\mathbf{G}_I$ a \textit{contraction} of $\mathbf{G}$. For a pair of $P$-pointed dual graphs $\mathbf{G}$, $\mathbf{G}'$, we say that $\mathbf{G}'$ is a \textit{specialisation} of $\mathbf{G}$ if $\mathbf{G}$ is isomorphic to a contraction of $\mathbf{G}'$. This defines a partial order on the set of isomorphism classes of $P$-pointed genus $g$ stable dual graphs:
\begin{align*}
[\mathbf{G}']\leq [\mathbf{G}]\quad\text{if } \mathbf{G}'\text{ is a specialisation of }\mathbf{G},\text{ i.e.~if }\mathbf{G}\cong \mathbf{G}'_I\text{ for some } I.
\end{align*}

We denote by $\mathcal{G}_{g,P}$ the \textit{poset of stable $P$-pointed genus $g$ dual graphs}.
\end{definition}

\begin{remark}
In \Cref{fig:stable dual graphs}, we have $(i)\leq (ii)\leq (iii)\leq (iv)$. Roughly put, the more edges the graph has, the smaller it is; the maximal element in $\mathcal{G}_{g,P}$ is the one vertex graph $(g,P)$ with $|P|$ labelled legs and whose weight function sends the single vertex to $g$. In the construction below, we see that we may interpret the notion of being smaller as being ``more nodal''.
\end{remark}

We can associate a stable dual graph to a stable nodal curve.

\begin{construction}[Dual graph associated to a stable curve]\label{dual graph associated to stable curve}
Let $(X; p_i)$ be a $P$-pointed nodal curve over $\C$ and let $\widetilde{X}\rightarrow X$ denote the normalisation of $X$. The \textit{dual graph} $\mathbf{G}(X;p_i)$ associated to $(X;p_i)$ is the graph $(G,\omega,m)$ defined as follows
\begin{itemize}
\item the set of vertices $V$ is the set of connected components of $\widetilde{X}$;
\item the set of half-edges is $E=\coprod_{v\in V} P_v$ where $P_v$ is the set of points of the connected component corresponding to $v$ that are mapped to a node or a marked point in $X$;
\item the involution $\iota \colon E\rightarrow E$ is given by
\begin{align*}
e\mapsto \begin{cases}
e & \text{if } e \text{ corresponds to a marked point}, \\
e' & \text{if }e \text{ and } e' \text{ correspond to the same node in } X.
\end{cases}
\end{align*}
Thus the legs of the graph correspond to the marked points and the edges to the nodes.
\item the weight function $\omega\colon V\rightarrow \Z_{\geq 0}$ sends a vertex $v$ to the arithmetic genus $g_v$ of the corresponding component of $\widetilde{X}$;
\item the marking $m\colon P\rightarrow E$ sends $i$ to the leg corresponding the marked point $p_i$.
\end{itemize}
The genus of the graph $\mathbf{G}(X;p_i)$ defined above is equal to the genus of $(X;p_i)$ (\cite[Chapter X, \S 2]{ArbarelloCornalbaGriffiths}). Moreover, by \cite[Chapter X, Lemma 3.1]{ArbarelloCornalbaGriffiths}, $\mathbf{G}(X;p_i)$ is stable if and only if $(X;p_i)$ is stable. 
\end{construction}

\begin{figure}[H]
\centering
\includegraphics[scale=0.2]{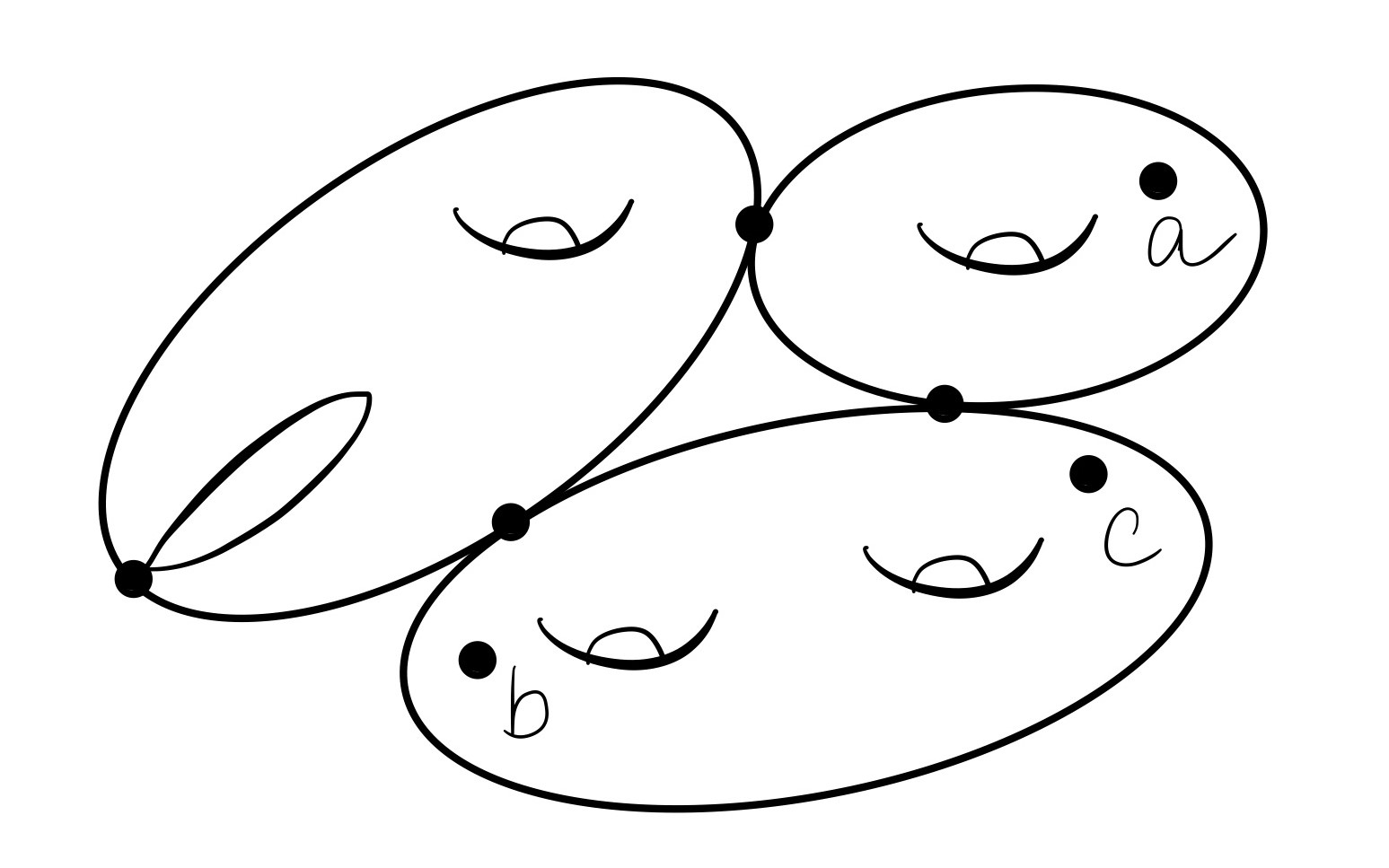}
\includegraphics[scale=0.2]{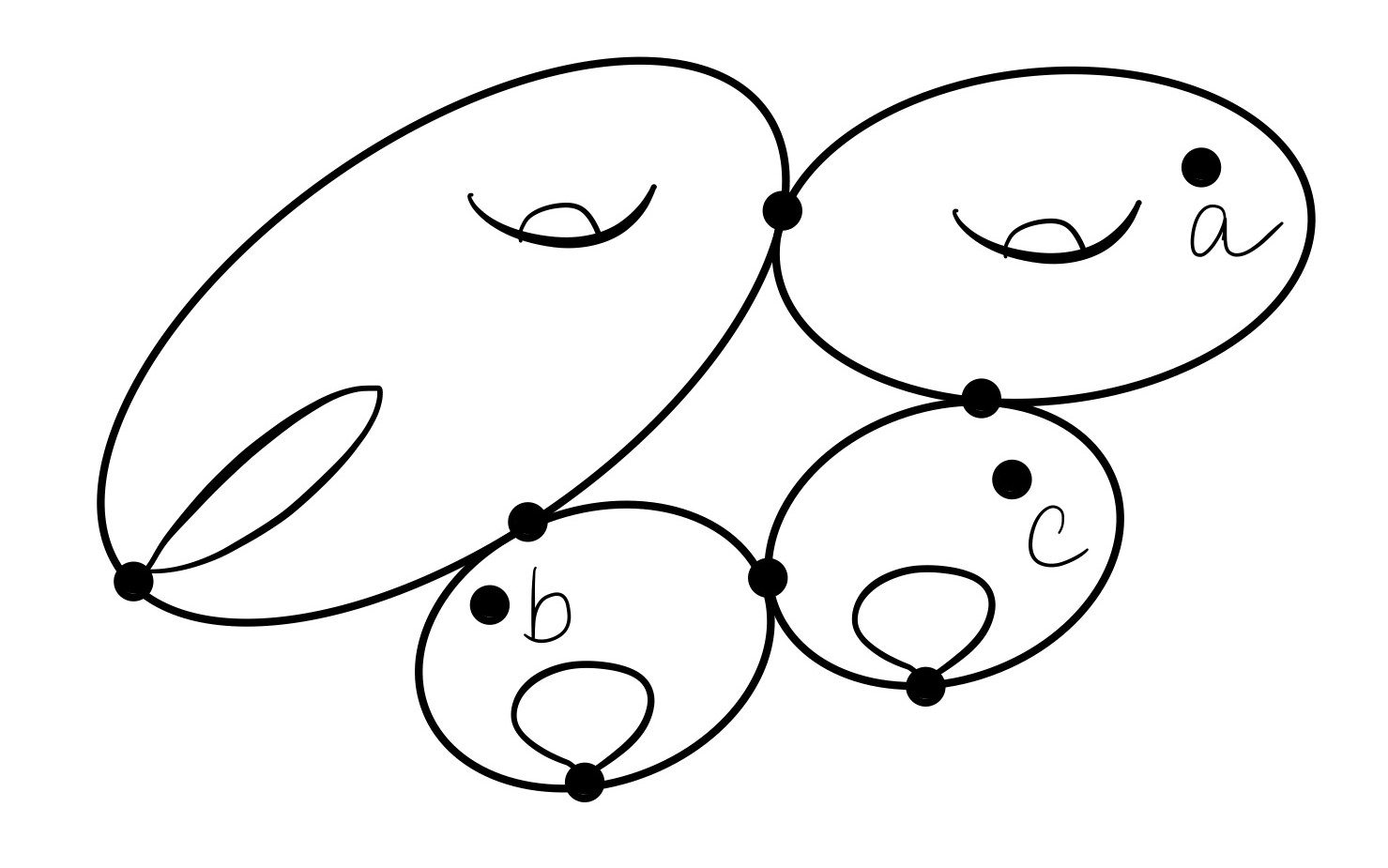}
\caption{Here are two stable $\{a,b,c\}$-pointed nodal curves of genus $6$. The dual graph associated to the curve on the left is graph (ii) of \Cref{fig:stable dual graphs} and the dual graph associated to the curve on the right is the graph (i).}
\label{fig:stable nodal curves}
\end{figure}

\section{The Harvey compactification as a real oriented blow-up}\label{appendix}

In this appendix, we identify the Harvey compactification of $\mathcal{M}_{g,P}$ as constructed by Ivanov \cite{Ivanov} with the real oriented blow-up of the Deligne--Mumford--Knudsen compactification of $\mathcal{M}_{g,P}$. Exhibiting the Harvey compactification as a real oriented blow-up is originally due to Looijenga \cite{Looijenga95} and the details are essentially given in \cite[Chapter XV\S 8]{ArbarelloCornalbaGriffiths}. In both of these sources, however, the Harvey bordification plays an auxiliary role as the bordification of interest is instead the so-called augmented Teichmüller space. For this reason, we choose to make the explicit identification here, taking Ivanov's model as our starting point. We also take care to identify the stacks and not just of the underlying coarse moduli spaces. We mostly follow \cite{ArbarelloCornalbaGriffiths}; more specifically \S 9 of Chapter X on real oriented blow-ups and \S 8 of Chapter 15 on the boundary of Teichmüller space.

We refer to \cite[pp.149-151]{ArbarelloCornalbaGriffiths} for a recap of real oriented blow-ups as needed in our situation (see also \cite{Gillam} for the general case and basic properties). We refrain from explicitly recollecting the construction here as we will in any case be referring to \cite{ArbarelloCornalbaGriffiths} for many of the finer details in this appendix. Given a smooth analytic space $X$ with a normal crossings divisor $D\subseteq X$, we denote the real oriented blow-up of $X$ at $D$ by $\operatorname{Bl}_DX$, or if $D$ is implicitly given (as for example the boundary), then we may write $\operatorname{Bl}_\partial X$. For us it suffices to deal explicitly with real oriented blow-ups of polydisks $B$ with $D$ a collection of coordinate hyperplanes (as spelt out in \cite{ArbarelloCornalbaGriffiths}). Since the real oriented blow-up is a locally defined construction, we can generalise to Deligne--Mumford stacks.

\begin{definition}
Given a smooth analytic Deligne--Mumford stack $\mathcal{M}$ with a normal crossings divisor $D\subseteq \mathcal{M}$ with groupoid presentation $X_1\rightrightarrows X_0$ and subsequent normal crossings divisors $D_1\rightrightarrows D_0$, we define the \textit{real oriented blow-up} $\operatorname{Bl}_D\mathcal{M}$ to be the stack presented by the groupoid
\begin{equation*}
\operatorname{Bl}_{D_1} X_1\rightrightarrows \operatorname{Bl}_{D_0} X_0.\qedhere
\end{equation*}
\end{definition}

Recall that the Harvey compactification of $\mathcal{M}_{g,P}$ is the quotient stack $[\Gamma_{S}\backslash \overline{\mathcal{T}}_{S}]$ of the Harvey bordification $\overline{\mathcal{T}}_{\!\!S}$ of Teichmüller space by the action of the mapping class group $\Gamma_S$. We refer to \Cref{Harvey bordification} and \cite{Ivanov} for details on the Harvey bordification of Teichmüller space and the action of the mapping class group on it, and we refer to \cite[p.286]{ArbarelloCornalbaGriffiths} for details on quotient stacks $[G\backslash X]$. We will consider the real oriented blow-up as a smooth manifold with corners. This is due to the fact that Ivanov's construction of the Harvey bordification of Teichmüller space is not real analytic (see comment on p.1179 of \cite{Ivanov}). The aim of this appendix is to prove the following theorem:

\begin{theorem}\label{Harvey compactification is real oriented blow-up}
Let $(S,P)$ be a non-singular stable surface of genus $g$ and let $\overline{\mathcal{M}}_{g,P}$ denote the algebraic stack of stable genus $g$, $P$-pointed curves. In the category of stacks over smooth manifolds with corners, the Harvey compactification can be identified with the real oriented blow-up of $\overline{\mathcal{M}}_{g,P}^{\operatorname{an}}$:
\begin{align*}
\operatorname{Bl}_\partial(\overline{\mathcal{M}}_{g,P}^{\operatorname{an}}) \xrightarrow{\ \cong \ } \left[\Gamma_{S}\backslash \overline{\mathcal{T}}_{S}\right]
\end{align*}
\end{theorem}

\begin{notation}
We attempt to follow the precise notation of \cite{ArbarelloCornalbaGriffiths} to ease transition between the two texts. However, as we will be referring to different sections of the book that use slightly different notation and we also wish to avoid multiple use of the same notation and too much discrepancy with the notation of the main body of this paper, we have had to make very slight alterations. We hope this will not be a source of confusion.
\end{notation}

\subsection{Going to the boundary of Teichmüller space}\label{going to the boundary}

In this section, we do the technical legwork preparing for our final identification, the important ingredient being a collection of embeddings $F\colon \widetilde{B}\rightarrow \overline{\mathcal{T}}_{\!\!S}$ from the universal cover $\widetilde{B}$ of the real oriented blow-up $\operatorname{Bl}_\partial B$ of the base of a ``small'' standard analytic Kuranishi family $\pi\colon X\rightarrow B$. These maps will depend on the family $\pi$ and on a marking of a concrete ``resolution'' of the central fibre. We define these maps and show that they are injective, smooth and moreover equivariant with respect to a natural action of Dehn twists.

Let $\pi\colon X\rightarrow B$ be a Kuranishi family of stable $P$-pointed curves with sections $\sigma_i\colon B\rightarrow X$, $i\in P$, and let $D\subseteq B$ be the divisor parametrising the singular fibres of $\pi$ (\cite[Chapter XI, \S 4]{ArbarelloCornalbaGriffiths}). We may assume that $\pi$ is \textit{small} in the sense that $B$ is a polydisc centered at $x=0$ with coordinates $t_1,\ldots,t_n$ and that the singular locus $D\subseteq B$ is given by union of coordinate hyperplanes i.e.~cut out by the equation $t_1\ldots t_k=0$ for some $k\leq n$ (\cite[Chapter XI, Theorem 3.17 and Corollary 4.7]{ArbarelloCornalbaGriffiths}). For a given $\epsilon>0$, we write $\Delta_\epsilon=\{t\in \C\mid |t|<\epsilon\}$, so we can denote our chosen polydisc by $B=\Delta_\epsilon^n$ for some $\epsilon>0$.

Consider the real oriented blow-up $\nu\colon \operatorname{Bl}_D(B)\rightarrow B$ and the pullback $\pi'\colon X' \rightarrow \operatorname{Bl}_D(B)$ of $\pi$ along $\nu$; we will also denote the induced sections of this family by $\sigma_i$, $i\in P$. Set $E:=\nu^{-1}(D)$, so that $\pi'$ coincides with $\pi$ over $\operatorname{Bl}_D(B)\smallsetminus E=B\smallsetminus D$. We can construct a real analytic family $\psi\colon Z\rightarrow \operatorname{Bl}_D(B)$ of compact orientable $P$-pointed genus $g$ differentiable surfaces together with a diagram
\begin{center}
\begin{tikzpicture}
\matrix (m) [matrix of math nodes,row sep=2em,column sep=2em]
  {
Z & X' \\
 & \operatorname{Bl}_D(B) \\
  };
  \path[-stealth]
(m-1-1) edge node[above]{$\lambda$} (m-1-2) edge node[below left]{$\psi$} (m-2-2)
(m-1-2) edge node[right]{$\pi'$} (m-2-2)
;
\end{tikzpicture}
\end{center} 
such that $\lambda$ restricts to the identity over $\operatorname{Bl}_D(B)\smallsetminus E$, i.e.~$\psi$ agrees with $\pi'$ away from $E$, and respects the $P$-labelled sections (\cite[Chapter X, Proposition 9.16]{ArbarelloCornalbaGriffiths}). We will need to take a closer look at some of the explicit details of the construction of this family in the proof of \Cref{singular structure} and in \ref{explicit trivialisation} below --- we will not, however, spell out the whole construction as this is very well exposed in \cite[pp.152-157]{ArbarelloCornalbaGriffiths} to which we refer the reader. Intuitively, we are resolving the singular fibres of $\pi$ by blowing up the nodes and thus producing a family of vanishing cycles; this is possible after taking the real oriented blow-up as we keep track of the directions in which we can move out of the singular locus $D$ into the smooth locus $B\smallsetminus D$.

Let $\zeta\colon\widetilde{B}\rightarrow \operatorname{Bl}_D(B)$ denote the universal cover of the real oriented blow-up and consider the pullback $\widetilde{\psi}\colon \widetilde{Z}\rightarrow \widetilde{B}$ of $\psi$ along $\zeta$ with $P$-labelled sections $\sigma_i$. Note that the space $\operatorname{Bl}_D(B)$ contracts to the central fibre of $\nu$, and since this is a torus (\cite[(9.13), p.150]{ArbarelloCornalbaGriffiths}), the universal cover $\widetilde{B}$ is contractible. It follows that we can trivialise the family $\widetilde{\psi}$: fix a stable $P$-pointed genus $g$ smooth surface $(S,P)$ and choose a smooth trivialisation $\alpha\colon S\times \widetilde{B}\xrightarrow{\cong} \widetilde{X}$ of $\widetilde{\psi}$ respecting the $P$-labelled sections.

Consider the following commutative diagram collecting the families just introduced and the maps between them:
\begin{center}
\begin{tikzpicture}
\matrix (m) [matrix of math nodes,row sep=2em,column sep=3em]
  {
S\times \widetilde{B} & \widetilde{Z} & Z & X' & X \\
\widetilde{B} & \widetilde{B} & \operatorname{Bl}_D(B) & \operatorname{Bl}_D(B) & B \\
  };
  \path[-stealth]
(m-1-1.356) edge node[above]{$\alpha$} node[below]{$\cong$} (m-1-2.192)
(m-1-1) edge (m-2-1)
(m-1-2) edge node[left]{$\widetilde{\psi}$} (m-2-2)
(m-1-2.348) edge (m-1-3)
(m-1-3) edge node[right]{$\psi$} (m-2-3) edge node[above]{$\lambda$} (m-1-4)
(m-1-4) edge (m-1-5) edge node[left]{$\pi'$} (m-2-4)
(m-1-5) edge node[right]{$\pi$} (m-2-5)
(m-2-2.345) edge node[below]{$\zeta$} (m-2-3)
(m-2-4) edge node[below]{$\nu$} (m-2-5)
;
\path[-]
(m-2-1.345) edge[double equal sign distance] (m-2-2.195)
(m-2-3) edge[double equal sign distance] (m-2-4)
(m-1-2) edge[white] node[black,near start]{\scalebox{1.5}{$\lrcorner$}} (m-2-3)
(m-1-4) edge[white] node[black,near start]{\scalebox{1.5}{$\lrcorner$}} (m-2-5)
;
\end{tikzpicture}
\end{center}

Denote the upper horizontal composite in the diagram above by $k\colon S\times \widetilde{B}\rightarrow X$. Let $t\in \widetilde{B}$, write $z=\nu\zeta(t)\in B$ and consider the map of fibres
\begin{align*}
k_t\colon (S,P)\rightarrow (X_z,\mathbf{x}_z)
\end{align*}
where $\mathbf{x}_z=\{\sigma_i(z)\}_{i\in P}$ denotes the set of marked points. Let $p_1,\ldots, p_\delta$ be the nodes of $X_z$ and let $L_t=\{L_{1,t},\ldots,L_{\delta,t}\}\subset S$ be the admissible curve system on $S$ given by the preimages of these nodes under $k_t$.

The following lemma allows us to define a map $\widetilde{B}\rightarrow \overline{\mathcal{T}}_{\!\!S}$. We refer to \Cref{Harvey bordification} and \cite[Definition 4.1]{Ivanov} for the definition of singular hyperbolic structures.

\begin{lemma}\label{singular structure}
The pair $(L_t,k_t)$ defines a singular hyperbolic structure on $S$.
\end{lemma}
\begin{proof}
To see this we have to verify \Cref{singular hyperbolic structure assumption} for each $L_{i,t}$, $i=1,\ldots,\delta$. Explicitly, we must find a neighbourhood $U_i$ of $L_{i,t}$ disjoint from the remaining curves of $L_t$ and the marked points of $S$, and a homeomorphism $d_i\colon U_i\rightarrow S^1\times (-1,1)$ such that the restriction
\begin{align*}
d_i\vert_{U_i\smallsetminus L_{i,t}}\colon U_i\smallsetminus L_{i,t}\longrightarrow S^1\times (-1,1)\smallsetminus S^1\times\{0\}
\end{align*}
is an isometry when
\begin{itemize}
\item $S^1\times (-1,1)\smallsetminus S^1\times \{0\}$ is equipped with the metric $\frac{d\theta^2 + ds^2}{s^2}$ where $\theta$ is the standard angular parameter on $S^1$ and $s\in (-1,1)\smallsetminus\{0\}$;
\item and $U_i\smallsetminus L_{i,t}$ is equipped with the hyperbolic metric of $k_t(U_i\smallsetminus L_{i,t})\subset X_z$, identifying $U_i\smallsetminus L_{i,t}$ with its image in $S/L_t$.
\end{itemize}

Of course, if $z\notin D$, then this is clear, since $X_z$ in non-singular and $k_t$ is a homeomorphism. So we assume that $z\in D$ and without loss of generality, we may assume that $t_j(z)=0$ for $j=1,\ldots,\delta$ and $t_j(z)\neq 0$ for $j=\delta+1,\ldots n$; recall that $D$ is cut out by the equation $t_1\cdots t_k=0$ for some $k\leq n$ and that we must have $\delta\leq k$. We take a closer look at the construction of the family $\psi\colon Z\rightarrow \operatorname{Bl}_D(B)$. Write $\tau=\zeta(t)\in \operatorname{Bl}_D(B)$ and factor $k_t$ as a composite:
\begin{align*}
(S,P)\xrightarrow{\alpha_t} (Z_\tau,\mathbf{z}_\tau)\xrightarrow{\lambda_\tau} (X_z,\mathbf{x}_z).
\end{align*}
By definition of the real oriented blow-up, $\tau=(z,\tau_1,\ldots,\tau_k)\in B\times (S^1)^k$ for some tuple of $\tau_j$ with $t_j(z)=|t_j(z)|\tau_j$ for $j=1,\ldots,k$ --- by assumption, $t_j(z)\neq 0$ for $j=\delta +1, \ldots,k$, so for these $j$, the direction $\tau_j$ is completely determined by $z$. For each $i\in \{1,\ldots,\delta\}$, let $x_i$ and $y_i$ be local coordinates on the two branches of $X_z$ at the node $p_i$ and let
\begin{align*}
V_i=\{(x_i,y_i)\in \Delta_\epsilon^2\mid x_iy_i=0\}\subseteq X_z
\end{align*}
be an $\epsilon$-neighbourhood of $p_i$ in these coordinates (we may take these small enough to be pairwise disjoint for all $i=1,\ldots,\delta
$). Consider the complex surface
\begin{align*}
A_i&=\{ (x_i,y_i,\xi_i,\eta_i)\in \Delta_\epsilon^2\times (S^1)^2\mid x_iy_i=0,x_i=|x_i|\xi_i,y_i=|y_i|\eta_i,\xi_i\eta_i=\tau_i \}.
\end{align*}
The surface $Z_\tau$ is obtained from $X_z$ by replacing the cone $V_i$ by the cylinder $A_i$ for all $i$, glueing along the complements of $\{x_i=y_i=0\}$ via the map $A_i\rightarrow V_i$, $(x_i,y_i,\xi_i,\eta_i)\mapsto (x_i,y_i)$ (\cite[Chapter X, Remark 9.17]{ArbarelloCornalbaGriffiths}). This replaces each node $p_i$ by a copy of $S^1$, namely the simple closed curve
\begin{align*}
\{ (0,0,\xi_i,\eta_i)\mid \xi_i\eta_i=\tau_i\}\subseteq A_i.
\end{align*}

There is a real analytic isomorphism
\begin{align*}
\phi_i\colon S^1\times (-1,1)\rightarrow A_i,\quad\quad (\rho,s)\mapsto \begin{cases}
(0,-s\epsilon\tau_i\overline{\rho}, \rho, \tau_i\overline{\rho} ) & s\leq 0 \\
(s\epsilon\rho,0,\rho,\tau_i\overline{\rho}) & s \geq 0
\end{cases}
\end{align*}
Using polar coordinates, we easily verify that this restricts to an isometry
\begin{align*}
S^1\times (-1,1)\smallsetminus S^1\times \{0\} \longrightarrow V_i\smallsetminus \{p_i\}
\end{align*}
when the left hand side is equipped with the metric $\frac{d\theta^2 + ds^2}{s^2}$ and the right hand side with the standard cuspidal hyperbolic metric on (the two copies of) the once-punctured disc. Setting $U_i:=\alpha_t^{-1}(A_i)$ and $d_i:=\phi_i^{-1}\circ\alpha_t\vert_{U_i}$ then verifies \Cref{singular hyperbolic structure assumption}.
\end{proof}

It follows from the above lemma, that we have a map
\begin{align*}
F_{\pi,\alpha}\colon \widetilde{B}\rightarrow \overline{\mathcal{T}}_{\!\!S},\quad t\mapsto s_t=[L_t,k_t]
\end{align*}
from $\widetilde{B}$ into the Harvey bordification of Teichmüller space $\mathcal{T}_S$.

\begin{remark}\label{dependence on single fibre}
Note that $F_{\pi,\alpha}$ depends on the choice of family $\pi\colon X\rightarrow B$ and the choice of trivialisation $\alpha\colon S\times \widetilde{B}\rightarrow \widetilde{Z}$. It turns out, however, that it only depends on $\pi$ and (the isotopy class of) the value of $\alpha$ on a single fibre. More precisely, let $\beta\colon S\times \widetilde{B}\rightarrow \widetilde{Z}$ be another trivialisation and suppose that for some $t_0\in \widetilde{B}$, the composite $\alpha_{t_0}^{-1}\circ \beta_{t_0}\colon S\rightarrow S$ is isotopic to the identity. For any $t\in \widetilde{B}$, choose a path $p\colon [0,1]\rightarrow \widetilde{B}$ from $t_0$ to $t$ and note that the homeomorphisms
\begin{align*}
\alpha_{p(s)}^{-1}\circ \beta_{p(s)}\colon S\rightarrow S,\quad\quad s\in [0,1]
\end{align*}
trace out an isotopy from $\alpha_{t_0}^{-1}\circ \beta_{t_0}$ to $\alpha_t^{-1}\circ \beta_t$. In other words, $\alpha_t^{-1}\circ \beta_t$ is isotopic to the identity for all $t\in \widetilde{B}$, and in particular, the associated singular structures
\begin{align*}
[(\lambda_{\zeta(t)}\circ \alpha_t)^{-1}\big(\{\operatorname{nodes}\}\big), \lambda_{\zeta(t)}\circ \alpha_t]\quad\text{and}\quad [(\lambda_{\zeta(t)}\circ \beta_t)^{-1}\big(\{\operatorname{nodes}\}\big), \lambda_{\zeta(t)}\circ \beta_t]
\end{align*}
are isotopic. In other words, $F_{\pi,\alpha}=F_{\pi,\beta}$.

So in fact, the map $F_{\pi,\alpha}$ depends only on the family $X$ and on a marking of, say, the central fibre $(S,P)\xrightarrow{\cong} (Z_0,\mathbf{z}_0)$ where $0$ abusively denotes the point $(0,1,\ldots,1)\in \operatorname{Bl}_DB\subset B\times (S^1)^k$. We will write out an explicit trivialisation below to better illustrate what is going on (\ref{explicit trivialisation}).

We will see, in fact, that once we mod out by the action of the mapping class group on $\overline{\mathcal{T}}_{\!\!S}$, any two choices of trivialisations will give rise to canonically isomorphic maps into $[\Gamma_S\backslash \overline{\mathcal{T}}_{\!\!S}]$ (\Cref{independence of trivialisation}); mapping further into the coarse moduli space, the dependency on $\alpha$ completely disappears.
\end{remark}

\begin{notation}
In view of the above remark, when we write $F_{\pi,\alpha}$, $\alpha$ may refer to either a trivialisation or simply a marking of the central fibre. Also, when no confusion can occur, we will alternately write $F$, $F_{\pi,\alpha}$ and $F_\alpha$ depending on which dependency we want to stress or ignore.
\end{notation}

For the remainder of this subsection, we work with respect to a fixed trivialisation $\alpha$. We have assumed that $B=\Delta_\epsilon^n$ is a polydisc of dimension $n$ and that the singular locus $D$ is given by the equation $t_1\cdots t_k=0$; in particular, the central fibre of $\pi\colon X\rightarrow B$ has $k$ nodes. For any subset $I\subseteq \{1,\ldots,k\}$, let $D_I\subset B$ denote the subspace of points $x\in B$ with $t_i(x)=0$ for all $i\in I$ and $t_i(x)\neq 0$ for $i\in \{1,\ldots,k\}\smallsetminus I$. This partition defines a natural stratification of $B$ over the power set of $\{1,\ldots,k\}$ (with reverse inclusion). Now, let $0\in \widetilde{B}\subseteq \R^k\times[0,\epsilon]^k\times \C^{n-k}$ denote the point $(0,\ldots,0)$ lying over the origin $0\in B$ and consider the admissible curve system $L_0=k_0^{-1}(\{\operatorname{nodes\ of} X_0\})=\{L_{0,1},\ldots,L_{0,k}\}$. We write $\mathbf{L}=\{L_1,\ldots,L_k\}$ for the corresponding isotopy class $[L_0]=\{[L_{0,1}],\ldots,[L_{0,k}]\}$; we write $\mathbf{L}_\alpha$ if we want to stress the dependence on $\alpha$.

\begin{lemma}\label{isotopy classes of vanishing cycles}
For any $t\in \widetilde{B}$, the isotopy class of the admissible curve system
\begin{align*}
L_t=k_t^{-1}(\{\operatorname{nodes\ of} X_{\nu\zeta(t)}\})
\end{align*}
is contained in $\mathbf{L}$. More precisely, if $\nu\zeta(t)\in D_I$, then $[L_t]=\{L_i\mid i\in I\}\subseteq \mathbf{L}$. In particular, if $\nu\zeta(t)=0$, then $[L_t]=\mathbf{L}$.
\end{lemma}
\begin{proof}
Let $r\colon X\rightarrow X_0$ denote the retraction to the central fibre of the Kuranishi family $\pi\colon X\rightarrow B$ as in \cite[Chapter X, Lemma 9.19]{ArbarelloCornalbaGriffiths} and write $r_z\colon X_z\rightarrow X_0$ for its restriction to the fibre $X_z=\pi^{-1}(z)$. By construction, the preimage of a node $p_i\in X_0$ under $r_z$ is a cycle if $t_i(z)\neq 0$ and a node if $t_i(z)=0$; in other words, the nodes of $X_z$ for $z\in D_I$ are exactly the preimages $r_z^{-1}(p_i)$, $i\in I$. Now, write $z=\nu\zeta(t)$ and suppose $z\in D_I$. For any choice of path $p\colon [0,1]\rightarrow \widetilde{B}$ from $t$ to $0$, the maps
\begin{align*}
r_{\nu\zeta(p(s))}\circ \lambda_{\zeta(p(s))}\circ \alpha_{p(s)}\colon (S,P)\rightarrow (X_0,\mathbf{x}_0),\quad\quad s\in [0,1]
\end{align*}
trace out an isotopy between $r_z\circ k_t$ and $k_0$. In particular, the admissible curve system
\begin{align*}
L_t=k_t^{-1}(\{\operatorname{nodes\ of}X_z\})=\{(r_z\circ k_t)^{-1}(p_i)\mid i\in I\}
\end{align*}
is isotopic to the admissible curve system $\{L_{0,i}\mid i\in I\}\subseteq L_0$.
\end{proof}

\begin{remark}
Note that we have not shown that the singular structures $[L_t,s_t]$ over the origin are independent of the choice of $t\in (\nu\zeta)^{-1}(0)$ since the isotopy in the proof above does not preserve the parametrisation of the annular neighbourhoods of the vanishing cycles. In contrast, the isotopy of \Cref{dependence on single fibre} is defined with respect to a fixed intermediate surface $Z_\tau$, so it does respect these parametrisations of the annular neighbourhoods.
\end{remark}

\begin{remark}
We stratify $\widetilde{B}$ by the preimages of the $D_I\subset B$, i.e.~over the power set of $\{1,\ldots,k\}$ with reverse inclusion. Recall that the Harvey bordification of Teichmüller space is naturally stratified over the curve poset of $S$. The lemma implies that $F$ is a stratum preserving map, sending a subset $I\subseteq \{1,\ldots,k\}$ to the set of vanishing cycles $\mathbf{L}_I=\{L_i\mid i\in I\}\subseteq \mathbf{L}$.
\end{remark}

\subsection{Explicit trivialisation}\label{explicit trivialisation}

In this section we construct an explicit trivialisation of the family $\widetilde{\psi}\colon \widetilde{Z}\rightarrow \widetilde{B}$. This will only be used briefly in what follows (cf. \Cref{injectivity}), but we think it may also shed some light on the behaviour of these maps as it allows us to analyse the generalised Fenchel--Nielsen coordinates on $\overline{\mathcal{T}}_{\!\!S}$ with relative ease (\cite[Definition 5.1]{Ivanov}). The trivialisation is similar to the one defined in \cite[pp.147-149]{ArbarelloCornalbaGriffiths}.

We will have to consider some explicit parts of the construction of the family $\psi\colon Z\rightarrow \operatorname{Bl}_DB$. As said, this is very well-described in \cite[pp.152-157]{ArbarelloCornalbaGriffiths} to which we refer the reader for details, so we only introduce a bare minimum here in order to make the necessary definitions and observations. For simplicity, we assume that we are dealing with a single node; for the general case, one just considers pairwise disjoint neighbourhoods of the nodes and proceeds in the same fashion. Thus we assume that we have a Kuranishi family $\pi\colon X\rightarrow B$ over a polydisc
\begin{align*}
B=\Delta_{\epsilon^2}^n=\{(t_1,\ldots,t_n)\in \C\mid |t_i|<\epsilon^2\},
\end{align*}
such that the central fibre $\pi^{-1}(0)$ has a single node $p$ and that the family is of the form $xy=t_1$ in a neighbourhood of $p$. It follows that the singular locus $D\subset B$ is the hyperplane cut out by $t_1=0$ and the pullback $\pi'\colon X'\rightarrow \operatorname{Bl}_DB$ is locally isomorphic to a neighbourhood of $\{x=y=0\}$ of the form 
\begin{align*}
W'=\{(x,y,t_2,\ldots,t_n,\tau)\in \Delta_\epsilon^2\times \Delta_{\epsilon^2}^n\times S^1\mid xy=|xy|\tau \}.
\end{align*}

The family $\psi\colon Z\rightarrow \operatorname{Bl}_DB$ is constructed from the pullback $\pi'\colon X'\rightarrow \operatorname{Bl}_DB$ by replacing $W'$ by
\begin{align*}
W=\{(x,y,t_2,\ldots,t_n,\xi,\eta)\in \Delta_\epsilon^2\times \Delta_{\epsilon^2}^n\times (S^1)^2\mid x=|x|\xi,y=|y|\eta \}.
\end{align*}
via the map $W\rightarrow W'$, $(x,y,t_2,\ldots,t_n,\xi,\eta)\mapsto (x,y,t_2,\ldots,t_n,\xi\eta)$ which is bijective on the complement of $\{x=y=0\}$. The fibre over a point $(x,y,t_2,\ldots,t_n,\tau)$ with $xy=0$ is a circle, namely the locus of those points $(x,y,t_2,\ldots,t_n,\xi,\eta)$ for which $\xi\eta=\tau$.

The projection $W\rightarrow \operatorname{Bl}_DB$ is given by
\begin{align*}
(x,y,t_2,\ldots,t_n,\xi,\eta) \mapsto (xy,t_2,\ldots,t_n,\xi\eta).
\end{align*}
Recall that the universal cover $\zeta\colon\widetilde{B}\rightarrow \operatorname{Bl}_DB$ is given explicitly in polar coordinates by
\begin{align*}
\widetilde{B}=\R\times [0,\epsilon^2)\times \Delta_{\epsilon^2}^{n-1}\longrightarrow \operatorname{Bl}_DB,\quad\quad (\theta, r, t_2,\ldots,t_n)\mapsto (re^{i\theta},t_2,\ldots,t_n,e^{i\theta}).
\end{align*}
Hence, the pullback $\widetilde{\psi}\colon \widetilde{Z}\rightarrow \widetilde{B}$ is locally isomorphic to a neighbourhood of $\{x=y=0\}$ of the form
\begin{align*}
\widetilde{W}=\{(x,y,t_2,\ldots,t_n,\xi,\eta,\theta)\in \Delta_\epsilon^2\times \Delta_{\epsilon^2}^n\times (S^1)^2\times \R\mid x=|x|\xi,y=|y|\eta, \xi\eta=e^{i\theta} \}
\end{align*} 
with the projection to $\widetilde{B}$ given by
\begin{align*}
(x,y,t_2,\ldots,t_n,\xi,\eta,\theta)\mapsto (\theta, |xy|, t_2,\ldots,t_n).
\end{align*}

For notational simplicity, we omit the $t_i$-coordinates from now on as they are just harmless parameters being carried around. Notice that since $\eta=\overline{\xi}e^{i\theta}$, we could also omit this from the tuple, but we keep this for clarity.

Let $Z_0=\widetilde{\psi}^{-1}(0,0)$ denote the fibre over $(0,0)\in \R\times [0,\epsilon^2)=\widetilde{B}$. We now define a trivialisation
\begin{center}
\begin{tikzpicture}
\matrix (m) [matrix of math nodes,row sep=2em,column sep=2em]
  {
Z_0\times \widetilde{B} & & \widetilde{Z} \\
 & \widetilde{B} & \\
  };
  \path[-stealth]
(m-1-1) edge node[above]{$\operatorname{Tw}$} node[below]{$\cong$} (m-1-3) edge (m-2-2)
(m-1-3) edge node[below right]{$\widetilde{\psi}$} (m-2-2)
;
\end{tikzpicture}
\end{center} 

Note that $Z_0$ is locally isomorphic to
\begin{align*}
U:=Z_0\cap \widetilde{W}=\{(x,y,\xi,\eta,0)\in \Delta_\epsilon^2\times (S^1)^2\times \R\mid x=|x|\xi,y=|y|\eta, \xi\eta=1,xy=0 \}.
\end{align*}
Denote by $W_+$ respectively $W_-$ the open subsets of $\widetilde{W}$ consisting of points with $|x|>|y|$ respectively $|x|<|y|$, and set $U_+:=U\cap W_+$ and $U_-:=U\cap W_-$.

\begin{remark}
Intuitively, the map $\operatorname{Tw}$ is going to twist the annular neighbourhood $U\subseteq Z_0$ of the curve $\{(0,0,\xi,\eta)\mid \xi\eta=1\}\subseteq Z_0$ according to the $\theta$-parameter of the given $(\theta,\delta)\in \widetilde{B}=\R\times [0,\epsilon)$ (see also \Cref{fig:picture of twist trivialisation}). We will in fact define its inverse $\operatorname{Untw}\colon \widetilde{Z}\rightarrow Z_0\times \widetilde{B}$: we ``untwist'' the annular neighbourhood in $\widetilde{Z}_{(\theta, \delta)}$ by $\theta$ while shrinking $\delta$ to $0$.
\end{remark}

We first of all consider a smooth trivialisation
\begin{center}
\begin{tikzpicture}
\matrix (m) [matrix of math nodes,row sep=2em,column sep=2em]
  {
\widetilde{Z}\smallsetminus \{|x|=|y|\} & & (Z_0\smallsetminus \{x=y=0\}) \times \widetilde{B} \\
 & \widetilde{B} & \\
  };
  \path[-stealth]
(m-1-1) edge node[above]{$H$} node[below]{$\cong$} (m-1-3) edge node[below left]{$\widetilde{\psi}$} (m-2-2)
(m-1-3) edge (m-2-2)
;
\end{tikzpicture}
\end{center} 
defined as in \cite[pp.147-148]{ArbarelloCornalbaGriffiths} to which we refer the reader for details. Locally on the subspace $\widetilde{W}\smallsetminus\{|x|=|y|\}=W_+\cup W_-$, the map $H$ is given by
\begin{align*}
W_+\xrightarrow{\ \cong \ } U_+\times \widetilde{B},\quad\quad (x,y,\xi,\eta,\theta)\mapsto \Big((|x|^2-|y|^2)\xi,0,\xi,\eta e^{-i\theta},\theta,|xy|\Big), \\
W_-\xrightarrow{\ \cong \ } U_-\times \widetilde{B},\quad\quad (x,y,\xi,\eta,\theta)\mapsto \Big(0,(|y|^2-|x|^2)\eta,\xi e^{-i\theta},\eta,\theta,|xy|\Big).
\end{align*}
Note that $(|x|^2-|y|^2)\xi=\tfrac{|x|^2-|y|^2}{|x|}x$ for points in $W_+$, and likewise $(|y|^2-|x|^2)\eta=\tfrac{|y|^2-|x|^2}{|y|}y$ for points in $W_-$; this may make it easier to understand what is going on.

We now define a trivialisation
\begin{center}
\begin{tikzpicture}
\matrix (m) [matrix of math nodes,row sep=2em,column sep=2em]
  {
\widetilde{W} & & U \times \widetilde{B} \\
 & \widetilde{B} & \\
  };
  \path[-stealth]
(m-1-1) edge node[above]{$\operatorname{Untw}_{\widetilde{W}}$} node[below]{$\cong$} (m-1-3) edge node[below left]{$\widetilde{\psi}$} (m-2-2)
(m-1-3) edge (m-2-2)
;
\end{tikzpicture}
\end{center} 
which agrees with $H$ on
\begin{align*}
\widetilde{W}\smallsetminus \{\big||x|^2-|y|^2\big|>\tfrac{\epsilon}{2}\}.
\end{align*}

Let $\chi\colon \R\rightarrow \R$ be a smooth decreasing function satisfying that $\chi(t)=1$ for all $t\leq 0$ and $\chi(t)=0$ for all $t\geq \tfrac{\epsilon}{2}$, and define $\operatorname{Untw}_{\widetilde{W}}$ as follows:
\begin{align*}
(x,y,\xi,\eta,\theta)\mapsto \begin{cases}
\Big(0,-s\eta e^{-i(1-\chi(s))\theta},\xi e^{-i\chi(s)\theta},\eta e^{-i(1-\chi(s))\theta},\theta,|xy|\Big) & |x|\leq|y| \\
\Big(s\xi e^{-i\chi(s)\theta},0,\xi e^{-i\chi(s)\theta},\eta e^{-i(1-\chi(s))\theta},\theta,|xy|\Big) & |x|\geq|y| \\
\end{cases}
\end{align*}
where for notational ease, we write $s=|x|^2-|y|^2$. We leave it to the reader to check that this is indeed a smooth trivialisation. Note that by our choice of $\chi$, we have
\begin{align*}
(0,-s\eta e^{-i(1-\chi(s))\theta},\xi e^{-i\chi(s)\theta},\eta e^{-i(1-\chi(s))\theta},\theta,|xy|)=(0,-s\eta,\xi e^{-i\theta},\eta,\theta,|xy|)
\end{align*}
for all points with $|x|\leq |y|$. So in fact, on $W_-$, we just have $\operatorname{Untw}_{\widetilde{W}}\vert_{W_-}=H\vert_{W_-}$, but on $W_+$ we ''untwist'' the $x$-parameter according to the $\theta$ so that the two maps extend smoothly to the locus $\{|x|=|y|\}$ (see \Cref{fig:picture of twist trivialisation} for an illustration). By our choice of $\chi$, we also see that $\operatorname{Untw}_{\widetilde{W}}$ agrees with $H$ on the subspace of $W_+$ given by $|x|^2-|y|^2>\tfrac{\epsilon}{2}$. It follows that these trivialisations glue together to define a smooth trivialisation:
\begin{center}
\begin{tikzpicture}
\matrix (m) [matrix of math nodes,row sep=2em,column sep=2em]
  {
\widetilde{Z} & & Z_0 \times \widetilde{B} \\
 & \widetilde{B} & \\
  };
  \path[-stealth]
(m-1-1) edge node[above]{$\operatorname{Untw}$} node[below]{$\cong$} (m-1-3) edge node[below left]{$\widetilde{\psi}$} (m-2-2)
(m-1-3) edge (m-2-2)
;
\end{tikzpicture}
\end{center}
We denote its inverse by $\operatorname{Tw}\colon Z_0\times \widetilde{B}\rightarrow \widetilde{Z}$. Given a marking of the central fibre $\alpha\colon (S,P)\rightarrow (Z_0,\mathbf{z}_0)$, we set
\begin{align*}
\widetilde{\alpha}:=\operatorname{Tw}\circ (\alpha,\operatorname{id})\colon S\times \widetilde{B}\rightarrow \widetilde{Z}.
\end{align*}

\begin{figure}[h]
\centering
\includegraphics[scale=0.2]{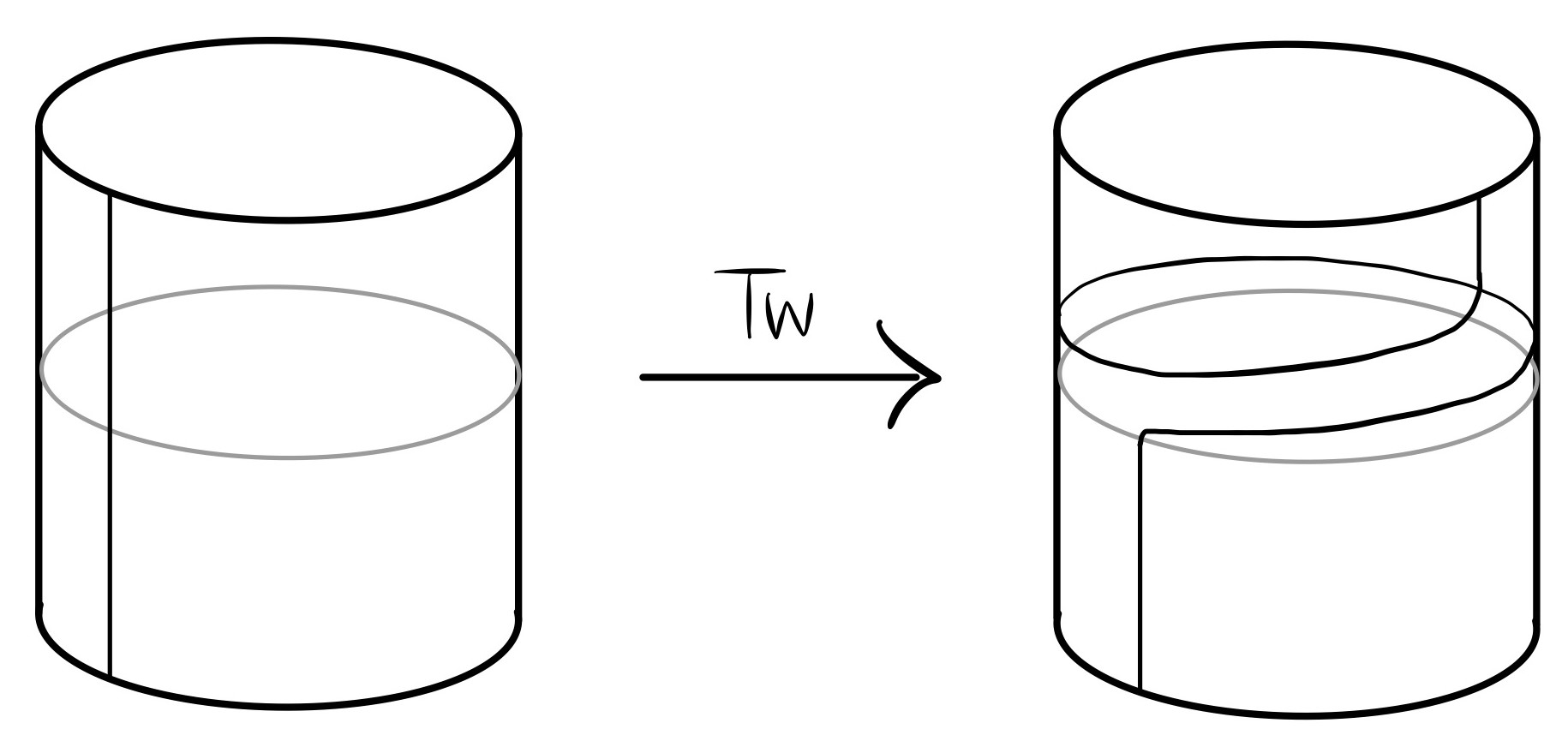}
\caption{
This is a picture of what the trivialisation $\operatorname{Tw}$ does fibrewise on annular neighbourhoods of the vanishing cycles. The vanishing cycles are drawn as faint horizontal cross sections and are given by the points $(x,y,\xi,\eta)$ for which $|x|=|y|$ in the respective annular neighbourhoods. The vertical line on the neighbourhood $U\subseteq Z_0$ on the left is given by the points $(x,y,1,1)$. The corresponding twisted line on the right is its image under $\operatorname{Tw}$. Explicitly, this is an attempt to depict the map $\operatorname{Tw}_{(\theta,0)}\colon Z_0\rightarrow Z_\tau$ over the point
\begin{align*}
(\theta,0)=\left(\tfrac{5\pi}{2},0\right)\in \R\times [0,\epsilon^2)=\widetilde{B}\qquad\qquad\qquad \ 
\end{align*}
with $\tau=\zeta(\theta,0)=(0,\sqrt{2}(1+i))\in \operatorname{Bl}_DB$ denoting the corresponding point in the real oriented blow-up.  
}
\label{fig:picture of twist trivialisation}
\end{figure}

\begin{observation}\label{twists of singular structures single node}
Let $\widetilde{\alpha}$ be a trivialisation as above and consider the associated map
\begin{align*}
F_{\pi,\alpha}\colon \widetilde{B}\rightarrow \overline{\mathcal{T}}_{\!\!S},\quad\quad t\mapsto s_t,
\end{align*}
as defined in the previous section. We continue with our simplified assumptions and notation, i.e.~the central fibre $X_0$ of $\pi$ has a single node and we omit the parameters $t_2,\ldots,t_n$.

For $(\theta,r)\in \R\times [0,\epsilon)= \widetilde{B}$, consider the singular structure $s_{(\theta,r)}=[L_{(\theta,r)},k_{(\theta,r)}]$. Explicitly, we have
\begin{align*}
k_{(\theta,r)}\colon S\xrightarrow{\ \alpha\ } Z_0\xrightarrow{\operatorname{Tw}_{(\theta,r)}} Z_{(re^{i\theta},e^{i\theta})}\xrightarrow{\lambda} X_{re^{i\theta}}.
\end{align*}

By definition of the trivialisation $\operatorname{Tw}$, it is not difficult to check that
\begin{align*}
s_{(\theta,r)}=s_{(0,r)}(\theta),
\end{align*}
where the right hand side is the result of twisting the singular structure $s_{(0,r)}$ by $\theta$ as in \cite[\S 4.3]{Ivanov}.
\end{observation}

In the general case where the central fibre has $k$ nodes, we perform such a twist in a neighbourhood of each node, but the procedure is completely analogous, and we also write $\operatorname{Tw}\colon \widetilde{Z}\rightarrow Z_0\times \widetilde{B}$ for the resulting trivialisation. Again, for a fixed a homeomorphism $\alpha\colon (S,P)\rightarrow (Z_0,\mathbf{z}_0)$, we set
\begin{align*}
\widetilde{\alpha}:=\operatorname{Tw}\circ (\alpha,\operatorname{id})\colon S\times \widetilde{B}\rightarrow \widetilde{Z}.
\end{align*}

The concrete advantage of writing down this explicit trivialisation is the following lemma in which we write
\begin{align*}
(\theta,r,t)=(\theta_1,\ldots,\theta_k,r_1,\ldots,r_k,t_{k+1},\ldots,t_n)\in \R^k\times [0,\epsilon)^k\times \Delta_{\epsilon^2}^{n-k}= \widetilde{B}
\end{align*}
for points of the universal cover $\widetilde{B}$.

\begin{lemma}\label{twists of singular structures general case}
Let $F\colon \widetilde{B}\rightarrow \overline{\mathcal{T}}_{\!\!S}$, $(\theta,r,t)\mapsto s_{(\theta,r,t)}$, be the map defined in \ref{going to the boundary} associated to a Kuranishi family $\pi\colon X\rightarrow B$ and any choice of trivialisation. Then for all $(\theta, r,t)\in \widetilde{B}$, we have $s_{(\theta,r,t)}=s_{(0,r,t)}(\theta)$, where the right hand side is the result of twisting the singular structure $s_{(0,r,t)}$ by $\theta=(\theta_1,\ldots,\theta_k)\in \R^k$ as in \cite[\S 4.3]{Ivanov}. In particular, the twist coordinates of \cite[\S 4.4]{Ivanov} satisfy
\begin{align*}
\Theta(s_{(\theta,r,t)})=\Theta(s_{(0,r,t)})+\theta.
\end{align*}
\end{lemma}
\begin{proof}
Firstly, by \Cref{dependence on single fibre}, we may assume that the trivialisation used to define $F$ is of the form $\widetilde{\alpha}$ above for some $\alpha\colon (S,P)\rightarrow (Z_0,\mathbf{z}_0)$. The claim then follows immediately by noting that \Cref{twists of singular structures single node} applies coordinatewise to the general case of $k$ nodes. The additional claim is a consequence of \cite[Lemma 3.7]{Ivanov}. 
\end{proof}

\subsection{An open cover of the Harvey bordification}\label{open cover}

We continue with the notation of the previous sections. We will show that the maps $F_{\pi,\alpha}\colon \widetilde{B}\rightarrow \overline{\mathcal{T}}_{\!\!S}$ are open embeddings and that they cover the Harvey bordification when varying the Kuranishi family $\pi\colon X\rightarrow B$ and the marking $\alpha\colon (S,P)\rightarrow (Z_0,\mathbf{z}_0)$. In order to do this, we need to compare with the concrete construction of the Harvey bordification. We refer the reader to \cite{Ivanov} for details and simply recall that the coordinate patches are given by length and twist parameters with respect to varying maximal curve systems on $S$:
\begin{align*}
X(\mu)\rightarrow \R_{\geq 0}^\mu\times \R^\mu, \quad s\mapsto (L_\mu^2(s),\Theta_\mu(s));
\end{align*}
here $\mu$ is a maximal curve system on $S$ and $X(\mu)\subseteq \overline{\mathcal{T}}_{\!\!S}$ is the open star neighbourhood of the $\mu$-stratum (\cite[\S 4.4 and Definition 5.1]{Ivanov}). We think of the length and twist coordinates as generalised Fenchel--Nielsen coordinates --- we are incorporating twist coordinates for the nodes, i.e.~the curves of length $0$ (see also \Cref{Remark: Fenchel-Nielsen coordinates on strata}).

In view of \Cref{isotopy classes of vanishing cycles}, the isotopy class of the curve system $L_t=\{L_{1,t},\ldots,L_{k,t}\}$ over the origin in $B$ is independent of the choice of $t\in (\nu\zeta)^{-1}(0)$. We denote this isotopy class by $\mathbf{L}=\{L_1,\ldots,L_k\}$ and choose some maximal curve system $\mu$ containing $\mathbf{L}$.

\begin{lemma}
The map $F_{\pi,\alpha}\colon \widetilde{B}\rightarrow \overline{\mathcal{T}}_{\!\!S}$ is smooth for any choice of $\pi$ and $\alpha$.
\end{lemma}
\begin{proof}
First of all note that the family of associated weak structures $\{\omega_t\}_{t\in \widetilde{B}}$ as constructed in \cite[\S 4.2]{Ivanov} is smooth as $\alpha$ is a smooth trivialisation. Consider the map
\begin{align*}
\widetilde{B}\rightarrow \R_{\geq 0}^\mu\times \R^\mu, \quad t\mapsto (L_\mu^2(s_t),\Theta_\mu(s_t))
\end{align*}
given by the length and twist parameters associated to singular structures as defined in \cite[\S 4.4]{Ivanov}. Since the family of weak structures is smooth, so is this map (\cite[Lemma 1.10 and \S 3.6, see also \S 5.4]{Ivanov}). Hence, by construction of $\overline{\mathcal{T}}_{\!\!S}$, $F_{\pi,\alpha}$ is smooth (\cite[Chapter XV, Definition 5.1]{Ivanov}).
\end{proof}

\begin{lemma}\label{injectivity}
The map $F_{\pi,\alpha}\colon \widetilde{B}\rightarrow \overline{\mathcal{T}}_{\!\!S}$ is injective for any choice of $\pi$ and $\alpha$.
\end{lemma}
\begin{proof}
Let $t,t'\in \widetilde{B}$ and assume that $s_t=s_{t'}$. Write $z$, respectively $z'$, for the images of $t$, respectively $t'$, in $B$. By the proof of \cite[Lemma 8.19]{ArbarelloCornalbaGriffiths}, we can conclude that $z=z'$. But then $t$ and $t'$ are of the form
\begin{align*}
t=(\theta_1,\ldots,\theta_k,r_1,\ldots,r_k,t_{k+1},\ldots,t_n),\quad\text{and}\quad t'=(\theta'_1,\ldots,\theta'_k,r_1,\ldots,r_k,t_{k+1},\ldots,t_n)
\end{align*}
and by \Cref{twists of singular structures general case}, we see that
\begin{align*}
\Theta(s_t)=\Theta(s_{t'}) + (\theta_i-\theta_i')=\Theta(s_t) + (\theta_i-\theta_i').
\end{align*}
Thus we conclude that $\theta_i=\theta_i'$ for all $i=1,\ldots,k$ and hence $t=t'$.
\end{proof}

\begin{corollary}
The map $F_{\pi,\alpha}\colon \widetilde{B}\rightarrow \overline{\mathcal{T}}_{\!\!S}$ is an open embedding for any choice of $\pi$ and $\alpha$.
\end{corollary}
\begin{proof}
By \cite[Chapter XI, Corollary 4.8]{ArbarelloCornalbaGriffiths}, the base $B$ is smooth of complex dimension $3g-3+n$, and thus both $\widetilde{B}$ and $\overline{\mathcal{T}}_{\!\!S}$ are smooth manifolds with corners of real dimension $6g-6+2n$. Hence, the claim follows by the previous two lemmas and invariance of domain.
\end{proof}

\begin{lemma}
For varying (small) Kuranishi families $\pi\colon X\rightarrow B$ and $\alpha$ running through all markings $(S,P)\xrightarrow{\cong} (Z_0,\mathbf{z}_0)$, the collection of open embeddings $F_{\pi,\alpha}\colon \widetilde{B}\hookrightarrow \overline{\mathcal{T}}_{\!\!S}$ cover $\overline{\mathcal{T}}_{\!\!S}$.
\end{lemma}
\begin{proof}
Indeed, let $s=[C,f\colon (S,P)\rightarrow (X_0,\mathbf{x}_0)]\in \overline{\mathcal{T}}_{\!\!S}$ be a singular structure and let $\pi\colon X\rightarrow B$ be a small Kuranishi family with central fibre $X_0$. Suppose $X_0$ has $k$ nodes $p_1,\ldots,p_k$. By definition, the local coordinates $x_i$ and $y_i$ on a neighbourhood $V_i$ of the node $p_i\in X_0$ are compatible with the parametrisation $d_i\colon S^1\times (-\epsilon,\epsilon)\rightarrow U_i$ of an annular neighbourhood $U_i\subseteq S$ of the corresponding curve $f^{-1}(p_i)\subseteq C$. It follows that there is a $\tau_i\in S^1$ such that the composite $f\circ d_i$ is given by
\begin{align*}
f\circ d\colon S^1\times (-\epsilon,\epsilon)\rightarrow V_i\subseteq X_0,\quad (\rho,s)\mapsto \begin{cases}
(s\rho,0) & s\geq 0 \\
(0,s\overline{\rho}\tau_i) & s \leq 0
\end{cases}
\end{align*}
in the local coordinates $(x_i,y_i)$ on $V_i$. It follows that for $\tau=(0,\tau_1,\ldots,\tau_k)\in \nu^{-1}(0)\subseteq \operatorname{Bl}_DB$, the marking $f$ lifts to a homeomorphism onto the fibre $(Z_\tau,\mathbf{z}_\tau)$ as in the following diagram
\begin{center}
\begin{tikzpicture}
\matrix (m) [matrix of math nodes,row sep=2em,column sep=2em]
  {
(S,P) & (Z_\tau,\mathbf{z}_\tau) \\
 & (X_0,\mathbf{x}_0) \\
  };
  \path[-stealth]
(m-1-1) edge node[above]{$\widehat{f}$}  (m-1-2) edge node[below left]{$f$} (m-2-2)
;
\path[->>]
(m-1-2) edge (m-2-2)
;
\end{tikzpicture}
\end{center}

Let $t\in \widetilde{B}$ such that $\zeta(t)=\tau$ and set $\alpha:=\operatorname{Untw}_t\circ \widehat{f}\colon (S,P)\rightarrow (Z_0,\mathbf{z}_0)$. The corresponding trivialisation $\widetilde{\alpha}=\operatorname{Tw}\circ (\alpha,\operatorname{id})\colon S\times \widetilde{B}\rightarrow \widetilde{Z}$ gives rise to an embedding $F_{\pi,\alpha}\colon \widetilde{B}\rightarrow \overline{\mathcal{T}}_{\!\!S}$ sending $t$ to $s$.
\end{proof}

Fix a marking $\alpha\colon (S,P)\rightarrow (Z_0,\mathbf{z}_0)$. Consider the mapping class group $\Gamma_S=\Gamma_{(S,P)}$ of $(S,P)$. The Picard--Lefschetz transformation defines a homomorphism (depending on our choice of marking $\alpha$)
\begin{align}\label{PL-transformation}
\pi_1(B\smallsetminus D)\longrightarrow \Gamma_S,
\end{align}
(see \cite[Chapter X, \S 9]{ArbarelloCornalbaGriffiths}). The real oriented blow-up $\operatorname{Bl}_DB$ contracts onto the central fibre $\nu^{-1}(0)\cong \mathbb{T}^k$, a $k$-dimensional torus (\cite[(9.13), p.150]{ArbarelloCornalbaGriffiths}) and we see that the inclusion $B\smallsetminus D=\operatorname{Bl}_DB\smallsetminus E\hookrightarrow \operatorname{Bl}_DB$ is a homotopy equivalence. Moreover, by the analysis in \cite[pp.157-160]{ArbarelloCornalbaGriffiths}, the generator $\gamma_i$ around the $i$'th coordinate circle is mapped to the Dehn twist around the corresponding vanishing cycle $L_i=k_0^{-1}(p_i)\subset S$. Letting $\Delta_{\mathbf{L}_\alpha}\subseteq \Gamma_S$ denote the free abelian subgroup generated by Dehn twists around the simple closed curves of $\mathbf{L}_\alpha$, it follows that we have identifications
\begin{align*}
\Delta_{\mathbf{L}_\alpha}\cong\pi_1(\operatorname{Bl}_D B)\cong\Z^k
\end{align*}
and that moreover, for the corresponding actions of $\Delta_{\mathbf{L}_\alpha}$ on $\widetilde{B}$ and $\overline{\mathcal{T}}_{\!\!S}$, we have
\begin{align*}
k_{\gamma t}=k_t\circ\gamma^{-1},\quad\text{for all }t\in \widetilde{B},\ \gamma\in \Delta_{\mathbf{L}_\alpha}.
\end{align*}
See also \cite[(8.14), p. 493]{ArbarelloCornalbaGriffiths} (they write $\Gamma(\mathbf{L}_\alpha)$ for our $\Delta_{\mathbf{L}_\alpha}$, but this clashes slightly with the notation of the main body of this paper, see also the proof of \Cref{SES of groups} below). Let us record this as a lemma in the following form.

\begin{lemma}\label{identification of Delta as Dehn twists}
The map $F_{\pi,\alpha}\colon \widetilde{B}\rightarrow \overline{\mathcal{T}}_{\!\!S}$ is equivariant with respect to the actions of $\Delta_{\mathbf{L}_\alpha}$.
\end{lemma}

We will need to know slightly more about the action of $\Gamma_S$ on the open subsets $\widetilde{B}_\alpha:=F_{\pi,\alpha}(\widetilde{B})$. In the following, we write $\operatorname{Aut}(X_0)$ for the automorphism group of the central fibre (omitting the marked points for notational ease). Denote by $\operatorname{Stab}_{\Gamma_S}(\widetilde{B}_\alpha)\subseteq \Gamma_S$ the subgroup preserving $\widetilde{B}_\alpha$ (not necessarily pointwise). We now restrict our attention to so-called standard Kuranishi families: recall that a standard (analytic) Kuranishi family $\pi\colon X\rightarrow B$ is one such that
\begin{enumerate}
\item $B$ is a connected complex manifold;
\item the family $\pi$ is Kuranishi at every point of $B$;
\item the action of $\operatorname{Aut}(X_0)$ on the central fibre extends to compatible actions on $X$ and $B$;
\item any isomorphism between fibres of $\pi$ is induced by an element of $\operatorname{Aut}(X_0)$.
\end{enumerate}
(see \cite[Chapter XI, Definition 6.8]{ArbarelloCornalbaGriffiths}).

\begin{lemma}\label{Stabiliser}
For a standard Kuranishi family $\pi\colon X\rightarrow B$ and any trivialisation $\alpha\colon S\times \widetilde{B}\rightarrow \widetilde{Z}$, we have
\begin{align*}
\operatorname{Stab}_{\Gamma_S}(\widetilde{B}_\alpha)= \{\gamma\in \Gamma_S\mid \gamma.\widetilde{B}_\alpha\cap \widetilde{B}_\alpha\neq \emptyset\}.
\end{align*}
\end{lemma}
\begin{proof}
The left to right inclusion is obvious, so we prove the reverse. Suppose $\gamma\in \Gamma_S$, $t,t'\in \widetilde{B}$ are such that $\gamma.F_\alpha(t)=F_\alpha(t')$. Then by assumption, there is an automorphism $\sigma$ of the entire family $\pi\colon X\rightarrow B$ such that $k_t\circ \gamma^{-1}$ is isotopic to $\sigma_{\nu\zeta(t)}\circ k_{t'}$.

We consider the induced automorphisms of the families constructed in \ref{going to the boundary}:
\begin{itemize}
\item Let $\sigma'\colon \operatorname{Bl}_DB\rightarrow \operatorname{Bl}_DB$ be the induced map of real oriented blow-ups and denote also by $\sigma'\colon X'\rightarrow X'$ the pullback of $\sigma\colon X\rightarrow X$ along $\operatorname{Bl}_DB\rightarrow B$;
\item Let $\widetilde{\sigma}\colon \widetilde{B}\rightarrow \widetilde{B}$ be the unique lift of $\sigma'$ sending $t$ to $t'$. By construction of the family $\psi\colon Z\rightarrow \operatorname{Bl}_D B$, there is a unique lift $\widehat{\sigma}\colon Z\rightarrow Z$ of $\sigma'\colon X'\rightarrow X'$ along $\lambda\colon Z\rightarrow X'$ (we leave the details to the reader), and we denote by $\widetilde{\sigma}\colon \widetilde{Z}\rightarrow \widetilde{Z}$ the pullback along $\rho\colon \widetilde{B}\rightarrow \operatorname{Bl}_D B$.
\end{itemize}

So we have automorphisms $\sigma'$, respectively $\widetilde{\sigma}$, of $\pi'\colon X'\rightarrow\operatorname{Bl}_DB$, respectively $\widetilde{\psi}\colon \widetilde{Z}\rightarrow \widetilde{B}$.

It is easily verified that the trivialisations
\begin{align*}
\beta:=\alpha\circ (\gamma^{-1}\times \operatorname{id}_{\widetilde{B}}), \quad \text{respectively}\quad \epsilon:=\widetilde{\sigma}^{-1}\circ\alpha\circ (\operatorname{id}_S\times \widetilde{\sigma}),
\end{align*}
evaluate to $\beta_s=\alpha_s\circ \gamma^{-1}$, respectively $\epsilon_s=\widehat{\sigma}_{\zeta(s)}^{-1}\circ\alpha_{\widetilde{\sigma}(s)}$, at $s\in \widetilde{B}$. It follows that $F_\beta=\gamma.F_\alpha$ and $F_\epsilon=F_\alpha\circ \widetilde{\sigma}$; for the latter equality we also use that $\sigma\circ \lambda=\lambda\circ \widehat{\sigma}$ and that the map $\sigma_z\colon X_z\rightarrow X_{\sigma(z)}$ is an isometry for all $z\in B$.

Now, simply note that by assumption, $\beta_t=\alpha_t\circ \gamma^{-1}$ is isotopic to $\epsilon_t=\widehat{\sigma}_{\zeta(t)}\circ\alpha_{t'}$ (given uniqueness of the lift $\widehat{\sigma}$), so these trivialisations define the same map into $\overline{\mathcal{T}}_{\!\!S}$ by \Cref{dependence on single fibre}, i.e.
\begin{align*}
\gamma.F_\alpha=F_\alpha\circ \widetilde{\sigma}.
\end{align*}
That concludes the proof.
\end{proof}

\begin{lemma}\label{SES of groups}
For any standard Kuranishi family $\pi\colon X\rightarrow B$ and any trivialisation $\alpha\colon S\times \widetilde{B}\rightarrow \widetilde{Z}$, we have a short exact sequence of groups
\begin{align*}
0\rightarrow \Delta_{\mathbf{L}_\alpha} \rightarrow \operatorname{Stab}_{\Gamma_S}(\widetilde{B}_\alpha) \rightarrow \operatorname{Aut}(X_0)\rightarrow 0.
\end{align*}
\end{lemma}
\begin{proof}
Write $F=F_\alpha$ and $\mathbf{L}=\mathbf{L}_\alpha=[L_0]$. We shall combine the following two facts:
\begin{enumerate}
\item The stabiliser of a point in Teichmüller space identifies with the automorphism group of the corresponding curve (see e.g.~\cite[p.451]{ArbarelloCornalbaGriffiths});
\item We have a short exact sequence of groups $0\rightarrow \Delta_\mathbf{L}\rightarrow \Gamma_S(\mathbf{L}) \rightarrow \Gamma_{S/L_0}\rightarrow 0$ where $\Gamma_S(\mathbf{L})$ is the subgroup of mapping classes $\gamma$ satisfying $\gamma(\mathbf{L})=\mathbf{L}$ (\Cref{SES of groups quotient by Dehn twists}).
\end{enumerate}

Consider the retraction $r_\mathbf{L}\colon X_\mathbf{L}\rightarrow \mathcal{T}_{S/L_0}$ from the $\mathbf{L}$-stratum in $\overline{\mathcal{T}}_{\!\!S}$ to the ``smaller'' Teichmüller space $\overline{\mathcal{T}}_{S/L_0}$ given by forgetting the twist parameters associated to the curves of $\mathbf{L}$. Concretely, given a singular structure $s=[L,k\colon S\rightarrow C]\in \overline{\mathcal{T}}_{\!\!S}$ with $[L]=\mathbf{L}$, the image of $s$ in $\mathcal{T}_{S/L_0}$ is given by the factorisation $\overline{k}\colon S/L_0\rightarrow C$ of $k$ through the nodal surface $S/L_0$ (potentially composing with the isometry witnessing $L\simeq L_0$). Note that the retraction is equivariant with respect to the action of $\Gamma_S(\mathbf{L})$. Let $s_0=F(0)$ be the singular hyperbolic structure associated to our chosen basepoint $0\in \widetilde{B}$ and consider the composite $\nu\zeta\colon \widetilde{B}\rightarrow B$ from the universal cover of the real oriented blow-up to the base space. Note that 
\begin{align}\label{retraction of central fibre}
r_\mathbf{L}(s_t)=r_\mathbf{L}(s_0)\quad \text{for any } t\in (\nu\zeta)^{-1}(0).
\end{align}
Write $\overline{s}_0:=r_\mathbf{L}(s_0)$.

It is a simple generalisation of item (1) above that the stabiliser $\operatorname{Stab}_{\Gamma_{S/L_0}}(\overline{s}_0)$ identifies with the automorphism group $\operatorname{Aut}(X_0)$. Thus to define the map
\begin{align*}
\operatorname{Stab}_{\Gamma_S}(\widetilde{B}_\alpha)\rightarrow \operatorname{Aut}(X_0),
\end{align*}
it suffices to note that any $\gamma\in \operatorname{Stab}_{\Gamma_S}(\widetilde{B}_\alpha)$ must belong to $\Gamma_S(\mathbf{L})$ (as $[L_t]\subset \mathbf{L}$ for all $t\in \widetilde{B}$) and that the corresponding element $\overline{\gamma}\in \Gamma_{S/L_0}$ must fix $\overline{s}_0$: indeed, as $\gamma\in \Gamma_S(\mathbf{L})$, we have $\gamma.s_0\in F((\nu\zeta)^{-1}(0))$, so this follows by \ref{retraction of central fibre} above.

The map is surjective, since if $\overline{\gamma}\in \Gamma_{S/L_0}$ fixes $\overline{s}_0$, then any lift $\gamma\in \Gamma_S(\mathbf{L})$ must satisfy
\begin{align*}
\gamma.s_0 \in F((\nu\zeta)^{-1}(0))\subseteq \widetilde{B}_\alpha
\end{align*}
(this can be verified by analysing the length and twist coordinates) and hence, by the previous lemma, $\gamma\in \operatorname{Stab}_{\Gamma_S}(\widetilde{B}_\alpha)$. Exactness at the middle term follows from the short exact sequence of item (2) above. That finishes the proof.
\end{proof}

\subsection{Identifying the groupoid presentations}

In this final section, we identify the real oriented blow-up of the Deligne--Mumford--Knudsen compactification with the Harvey compactification as stated in \Cref{Harvey compactification is real oriented blow-up}. Note that the stacks in question, $[\Gamma_S\backslash \overline{\mathcal{T}}_{\!\!S}]$ and $\operatorname{Bl}_\partial(\overline{\mathcal{M}}_{g,P}^{\operatorname{an}})$, are ``ordinary'' Deligne--Mumford stacks, that is, we can consider them as $1$-truncated $\infty$-stacks over smooth manifolds with corners admitting an étale atlas. In order to identify the two stacks it therefore suffices to show that they admit the same groupoid presentation.

Consider the groupoid presentation $\mathbf{I}\rightrightarrows X$ of $\overline{\mathcal{M}}_{g,P}$ constructed in \cite[Chapter XII, Theorem 8.3]{ArbarelloCornalbaGriffiths}. To summarise, $X=\coprod_{i=1}^N X_i$ is the Kuranishi atlas of $\overline{\mathcal{M}}$ defined in \S 3 of \cite[Chapter XII]{ArbarelloCornalbaGriffiths} (more specifically on page 272) where each $X_i$ is the base space of a standard algebraic Kuranishi family in the sense of \cite[Chapter XI, Definition 6.7]{ArbarelloCornalbaGriffiths}. Letting $\xi\colon C\rightarrow X$ denote the family of curves over $X$,
\begin{align*}
\mathbf{I}=\mathbf{Isom}_{X\times X}(p_1^*\xi,p_2^*\xi)
\end{align*}
is the scheme parametrising isomorphisms between fibres of the two families over $X\times X$ given by pulling back $\xi$ along the two projection maps, $p_1$ and $p_2$, to $X$ (see \cite[Chapter X, \S 5]{ArbarelloCornalbaGriffiths}).

For notational simplicity, we will also denote the analytification of this groupoid by $\mathbf{I}\rightrightarrows X$. The real oriented blow of $\overline{\mathcal{M}}_{g,P}^{\operatorname{an}}$ is then presented by the groupoid $\operatorname{Bl}_\partial \mathbf{I} \rightrightarrows \operatorname{Bl}_\partial X$ in the category of real analytic stacks, but we will consider it as a stack over smooth manifolds with corners.

We will define a map $\operatorname{Bl}_\partial X\rightarrow [\Gamma_S\backslash \overline{\mathcal{T}}_{\!\!S}]$ and identify the two groupoids
\begin{align*}
\operatorname{Bl}_\partial X\!\!\!\!\mathop{\times}_{[\Gamma_S\backslash \overline{\mathcal{T}}_{\!\!S}]}\!\!\!\!\operatorname{Bl}_\partial X\rightrightarrows \operatorname{Bl}_\partial X\quad\quad \text{ and }\quad\quad \operatorname{Bl}_\partial \mathbf{I} \rightrightarrows \operatorname{Bl}_\partial X.
\end{align*}
We do this by analysing the real oriented blow-up $\operatorname{Bl}_\partial X$ locally using coordinate patches; much the same way as it is constructed.

As mentioned, $X=\coprod X_i$ is a finite disjoint union of base spaces of standard (algebraic) Kuranishi families $\pi_i\colon C_i\rightarrow X_i$ (\cite[Chapter XI, Definition 6.7]{ArbarelloCornalbaGriffiths}).

\begin{assumption}[Basic coordinate patches]\label{basic coordinate patches}
Given $x\in X_i\subseteq X$, let $G_x=\operatorname{Aut}(C_x)$ denote the automorphism group of the fibre over $x$. There is a connected $G_x$-invariant analytic neighbourhood $U\subseteq X_i$ of $x$ such that any isomorphism of between fibres over $U$ is induced by an element of $G_x$ (part (e) of \cite[Chapter XI, Definition 6.7]{ArbarelloCornalbaGriffiths}). It follows that the restriction $\pi_U\colon C_U\rightarrow U$ is a standard (analytic) Kuranishi family (\cite[Chapter XI, Definition 6.8]{ArbarelloCornalbaGriffiths}). As in the previous sections, we may assume that $\pi_U\colon C_U\rightarrow U$ is \textit{small} in the sense that $U$ is a polydisc centered at $x=0$  with coordinates $t_1,\ldots, t_n$ and that the singular locus $D\subset U$ is given by a union of coordinate hyperplanes cut out by the equation $t_1\cdots t_k=0$ for some $k\leq n$ (\cite[Chapter XI, Theorem 3.17 and Corollary 4.7]{ArbarelloCornalbaGriffiths}).
\end{assumption}

As described in \ref{going to the boundary}, we consider the pullback of $\pi_U\colon C_U\rightarrow U$ to the real oriented blow-up $\operatorname{Bl}_D U$, the consequent ``resolution'' $Z\rightarrow \operatorname{Bl}_D U$, a family of non-singular curves, and the pullback $\widetilde{Z}\rightarrow \widetilde{U}$ of this to the universal cover $\widetilde{U}$ of $\operatorname{Bl}_D U$. For varying trivialisations $\alpha\colon S\times \widetilde{U}\rightarrow \widetilde{Z}$ (equivalently, varying markings of the central fibre $Z_0$), we consider the open embeddings $F_\alpha\colon \widetilde{U}\rightarrow \overline{\mathcal{T}}_{\!\!S}$. Denote by $\mathbf{L}_\alpha$ the isotopy class of the curve system $L_0=\{L_{1,0},\ldots,L_{k,0}\}=k_0^{-1}\{\operatorname{nodes\ of}X_0\})$ over the origin in $B$ and let $\Delta_{\mathbf{L}_\alpha}\subseteq \Gamma_S$ be the free abelian subgroup generated by Dehn twists around the simple closed curves of $\mathbf{L}_\alpha$.

Now, consider the following commutative diagram where the lower horizontal map on the right exists because of \Cref{identification of Delta as Dehn twists} above.

\begin{center}
\begin{tikzpicture}
\matrix (m) [matrix of math nodes,row sep=2em,column sep=3em]
  {
\widetilde{U} & \overline{\mathcal{T}}_{\!\!S} & \\
\operatorname{Bl}_D(U) & \Delta_{\mathbf{L}_\alpha}\backslash \widetilde{U} & \left[\Gamma_S\backslash \overline{\mathcal{T}}_{\!\!S}\right] \\
  };
  \path[-stealth]
(m-1-1) edge (m-2-1) 
(m-1-1) edge (m-2-2)
(m-1-1.355) edge node[above]{$F_\alpha$} (m-1-2) 
(m-2-1) edge node[below]{$\cong$} (m-2-2.182)
(m-2-2.357) edge (m-2-3)
(m-1-2) edge (m-2-3)
;
\end{tikzpicture}
\end{center}

We denote the lower horizontal composite by $\overline{F}=\overline{F}_\alpha=\overline{F}_{\alpha, U} \colon \operatorname{Bl}_DU\rightarrow [\Gamma_S\backslash \overline{\mathcal{T}}_{\!\!S}]$ (according to whether we want to stress the dependencies or not).

Explicitly, the composite $\widetilde{U}\rightarrow \overline{\mathcal{T}}_{\!\!S}\rightarrow [\Gamma_S\backslash \overline{\mathcal{T}}_{\!\!S}]$, which we will by slight abuse of notation also denote by $F_\alpha$, is given by the principal $\Gamma_S$-bundle
\begin{align*}
\pi_\alpha\colon \Gamma_S\times \widetilde{U}\rightarrow \widetilde{U},
\end{align*}
with $\Gamma_S$ acting from the left on the first factor, together with the $\Gamma_S$-equivariant map $\Gamma_S\times \widetilde{U}\rightarrow\overline{\mathcal{T}}_{\!\!S}$, $(\gamma,t)\mapsto \gamma.F_\alpha(t)$. The corresponding map $\overline{F}_\alpha$ is explicitly given by the principal $\Gamma_S$-bundle
\begin{align*}
\overline{\pi}_\alpha\colon \Gamma_S\times_{\Delta_{\mathbf{L}_\alpha}}\widetilde{U}\rightarrow \operatorname{Bl}_DU,
\end{align*}
where we mod out by the diagonal action of $\Delta_{\mathbf{L}_\alpha}$ on $\Gamma_S\times \widetilde{U}$ via the Picard--Lefschetz tranformation $\pi_1(\operatorname{Bl}_DU)\xrightarrow{\cong} \Delta_{\mathbf{L}_\alpha}$ (\ref{PL-transformation} in \ref{open cover}), and with $\Gamma_S$ acting from the left on the first factor, together with the $\Gamma_S$-equivariant map
\begin{align*}
\Gamma_S\times_{\Delta_{\mathbf{L}_\alpha}}\widetilde{U}\rightarrow \overline{\mathcal{T}}_{\!\!S},\quad \overline{(\gamma,t)}\mapsto \gamma.F_\alpha(t).
\end{align*}

\begin{lemma}\label{independence of trivialisation}
For any two choices of trivialisations $\alpha$ and $\beta$, the associated maps
\begin{align*}
F_\alpha,F_\beta\colon\widetilde{U} \rightarrow \left[\Gamma_S\backslash \overline{\mathcal{T}}_{\!\!S}\right]
\end{align*}
are canonically isomorphic. In particular, so are $\overline{F}_\alpha$ and $\overline{F}_\beta$.
\end{lemma}
\begin{proof}
As in \Cref{dependence on single fibre}, we see that $\alpha_t^{-1}\circ \beta_t$ is isotopic to $\alpha_0^{-1}\circ \beta_0$ for all $t\in \widetilde{U}$ and we let $\eta_{\alpha\beta}$ denote this mapping class. The map
\begin{align*}
\Gamma_S\times\widetilde{U}\rightarrow \Gamma_S\times\widetilde{U}\quad (\gamma,t)\mapsto (\gamma\eta_{\alpha\beta},t)
\end{align*}
defines an isomorphism $\pi_\alpha\cong\pi_\beta$ of fibre bundles commuting with the maps to $\overline{\mathcal{T}}_{\!\!S}$ and as such an isomorphism from $F_\alpha$ to $F_\beta$.  
\end{proof}

\begin{corollary}
For any  $U$ and $V$ as in \Cref{basic coordinate patches} with $V\subseteq U$ and any choice of trivialisations $\alpha\colon S\times \widetilde{U}\rightarrow \widetilde{Z}_U$ and $\beta\colon S\times \widetilde{V}\rightarrow \widetilde{Z}_V$, the maps
\begin{align*}
\overline{F}_{\beta,V}\colon \operatorname{Bl}_DV\rightarrow [\Gamma_S\backslash \overline{\mathcal{T}}_{\!\!S}]\quad \text{and}\quad\overline{F}_{\alpha,U}\vert_{\operatorname{Bl}_DV}\colon  \operatorname{Bl}_DV\hookrightarrow \operatorname{Bl}_DU\rightarrow [\Gamma_S\backslash \overline{\mathcal{T}}_{\!\!S}]
\end{align*}
are canonically isomorphic.
\end{corollary}
\begin{proof}
Consider the sequence of families used to define the map $F_\alpha\colon \widetilde{U}\rightarrow \overline{\mathcal{T}}_{\!\!S}$:
\begin{center}
\begin{tikzpicture}
\matrix (m) [matrix of math nodes,row sep=2em,column sep=3em]
  {
S\times \widetilde{U} & \widetilde{Z}_U & Z_U & C_U' & C_U \\
\widetilde{U} & \widetilde{U} & \operatorname{Bl}_DU & \operatorname{Bl}_DU & U \\
  };
  \path[-stealth]
(m-1-1.355) edge node[above]{$\alpha$} node[below]{$\cong$} (m-1-2.185)
(m-1-1) edge (m-2-1)
(m-1-2) edge node[left]{$\widetilde{\psi}_U$} (m-2-2)
(m-1-2.353) edge (m-1-3.177)
(m-1-3) edge node[right]{$\psi_U$} (m-2-3) edge (m-1-4)
(m-1-4) edge (m-1-5) edge node[left]{$\pi'_U$} (m-2-4)
(m-1-5) edge node[right]{$\pi_U$} (m-2-5)
(m-2-2.345) edge (m-2-3)
(m-2-4) edge (m-2-5)
;
\path[-]
(m-2-1.345) edge[double equal sign distance] (m-2-2.195)
(m-2-3) edge[double equal sign distance] (m-2-4)
(m-1-2) edge[white] node[black,near start]{\scalebox{1.5}{$\lrcorner$}} (m-2-3)
(m-1-4) edge[white] node[black,near start]{\scalebox{1.5}{$\lrcorner$}} (m-2-5)
;
\end{tikzpicture}
\end{center}
and the corresponding diagram for $V$. Let $\iota_{V,U}\colon\operatorname{Bl}_D V\rightarrow \operatorname{Bl}_D U$ denote the induced embedding of real oriented blow-ups (see \cite[p.150]{ArbarelloCornalbaGriffiths}). By construction, the real analytic family $Z_V\rightarrow \operatorname{Bl}_DV$ fits into a pullback diagram:
\begin{center}
\begin{tikzpicture}
\matrix (m) [matrix of math nodes,row sep=2em,column sep=3em]
  {
Z_V & Z_U \\
\operatorname{Bl}_DV & \operatorname{Bl}_DU \\
  };
  \path[-stealth]
(m-1-1) edge (m-1-2) edge node[left]{$\psi_V$} (m-2-1)
(m-2-1) edge node[below]{$\iota_{V,U}$} (m-2-2)
(m-1-2) edge node[right]{$\psi_U$} (m-2-2)
(m-1-1) edge[white] node[black,near start]{\scalebox{1.5}{$\lrcorner$}} (m-2-2)
;
\end{tikzpicture}
\end{center}

It follows that we can define a trivialisation $\alpha_V:=\alpha_U\circ (\operatorname{id}_S\times \widetilde{\iota})\colon S\times \widetilde{V}\rightarrow \widetilde{Z}_V$ of the family $\widetilde{\psi}_V\colon\widetilde{Z}_V\rightarrow \widetilde{V}$ for some lift $\widetilde{\iota}\colon \widetilde{V}\rightarrow \widetilde{U}$ of $\iota_{V,U}\colon \operatorname{Bl}_DV\rightarrow \operatorname{Bl}_DU$ to the universal covers. This trivialisation exactly recovers $\overline{F}_{\alpha_V,V}=\overline{F}_{\alpha,U}\circ \iota_{V,U}$. The claim then follows from \Cref{independence of trivialisation}.
\end{proof}

It follows from the above corollary, that the morphisms $\operatorname{Bl}_D U\rightarrow [\Gamma_S\backslash \overline{\mathcal{T}}_{\!\!S}]$ patch together to define a morphism
\begin{align*}
\operatorname{Bl}_\partial X\rightarrow [\Gamma_S\backslash \overline{\mathcal{T}}_{\!\!S}].
\end{align*}

In order to analyse the fibre product $\operatorname{Bl}_\partial X\mathop{\times}_{[\Gamma_S\backslash \overline{\mathcal{T}}_{\!\!S}]}\operatorname{Bl}_\partial X$, let us first of all consider the ``local'' fibre products
\begin{align*}
\operatorname{Bl}_DU\!\!\!\!\mathop{\times}_{[\Gamma_S\backslash \overline{\mathcal{T}}_{\!\!S}]}\!\!\!\!\operatorname{Bl}_DU.
\end{align*}
The following lemma should be compared with \cite[Chapter XII, Lemma 3.11]{ArbarelloCornalbaGriffiths} --- combining these two lemmas will be the key to our final identification.

\begin{lemma}\label{local picture}
For any $U\subset X$ is as in \Cref{basic coordinate patches}, there is an isomorphism 
\begin{align*}
G_x\times \operatorname{Bl}_DU\xrightarrow{\ \cong \ } \operatorname{Bl}_DU\mathop{\times}_{[\Gamma_S\backslash \overline{\mathcal{T}}_{\!\!S}]}\operatorname{Bl}_DU
\end{align*}
such that the composite with the natural projection to $\operatorname{Bl}_DU\times \operatorname{Bl}_DU$ is given by
\begin{align*}
G_x\times \operatorname{Bl}_DU\rightarrow \operatorname{Bl}_DU\times\operatorname{Bl}_DU,\quad (g,u)\mapsto (gu,u).
\end{align*}
\end{lemma}
\begin{proof}
Combining \Cref{Stabiliser} with the identifications $\operatorname{Bl}_\partial U\cong \Delta_{\mathbf{L}_\alpha}\backslash \widetilde{U}$ and $G_x\cong \Delta_{\mathbf{L}_\alpha}\backslash\operatorname{Stab}_{\Gamma_S}(\widetilde{U}_\alpha)$ (\Cref{SES of groups}) and the general fact that $[(H\backslash G)\backslash [H\backslash X]]\cong [G\backslash X]$, we see that we have canonical isomorphisms
\begin{align*}
\operatorname{Bl}_DU\!\!\!\!\mathop{\times}_{[\Gamma_S\backslash \overline{\mathcal{T}}_{\!\!S}]}\!\!\!\!\operatorname{Bl}_DU\xrightarrow{\ \cong\ }\operatorname{Bl}_DU\!\!\!\!\!\!\mathop{\times}_{[G_x\backslash \operatorname{Bl}_DU]}\!\!\!\!\!\!\operatorname{Bl}_DU \xleftarrow{\ \cong\ } G_x\times \operatorname{Bl}_DU
\end{align*}
where the latter sends $(g,u)$ to $(gu,u,g^{-1})$.
\end{proof}

We are now all set to identify the groupoid presentations and make the identification that we set out to in this appendix.

\begin{proof}[Proof of \Cref{Harvey compactification is real oriented blow-up}]
Consider a point $\phi$ in $\mathbf{I}=\mathbf{Isom}_{X\times X}(p_1^*\xi,p_2^*\xi)$ corresponding to an isomorphism $\phi\colon C_x\rightarrow C_y$ for $x,y\in X$. Let $U$, respectively $W$, be neighbourhoods of $x$, respectively $y$, as specified in \Cref{basic coordinate patches}. Consider the resulting analytic neighbourhood of $\phi\in \mathbf{I}$,
\begin{align*}
\mathbf{I}_{U\times W}:=\mathbf{Isom}_{U\times W}(p_1^*(\xi\vert_U),p_2^*(\xi\vert_W))\subseteq \mathbf{I},
\end{align*}
and the corresponding open $\operatorname{Bl}_\partial\mathbf{I}_{U\times W}\subset \operatorname{Bl}_\partial\mathbf{I}$. The projections to $\operatorname{Bl}_\partial X$ factor through the pullback $\operatorname{Bl}_\partial U\mathop{\times}_{[\Gamma_S\backslash \overline{\mathcal{T}}_{\!\!S}]} \operatorname{Bl}_\partial W$ as in the diagram below
\begin{center}
\begin{tikzpicture}
\matrix (m) [matrix of math nodes,row sep=2em,column sep=3em]
  {
\operatorname{Bl}_\partial\mathbf{I}_{U\times W} & \operatorname{Bl}_\partial X \\
\operatorname{Bl}_\partial U\mathop{\times}_{[\Gamma_S\backslash \overline{\mathcal{T}}_{\!\!S}]} \operatorname{Bl}_\partial W & {[\Gamma_S\backslash \overline{\mathcal{T}}_{\!\!S}]} \\
  };
  \path[-stealth]
(m-1-1) edge node[left]{$\Phi_{U\times W}$} (m-2-1)
(m-1-2) edge (m-2-2)
(m-1-1.354) edge (m-1-2.190)
(m-2-1.359) edge (m-2-2.188)
(m-1-1.8) edge (m-1-2.170)
(m-2-1.6) edge (m-2-2.171)
;
\end{tikzpicture}
\end{center}

We claim that the left vertical map $\Phi_{U\times W}$ is a homeomorphism. To see this, note first of all that since 
\begin{align*}
\xi_U\colon C_U\rightarrow U\quad\text{and}\quad\xi_W\colon C_W\rightarrow W
\end{align*}
are standard Kuranishi families with isomorphic central fibres, we may assume that there is an isomorphism $\gamma\colon U\rightarrow W$ such that $\gamma^*(\xi_W)=\xi_U$. As such, it suffices to consider the case when $U=W$, and the claim follows by combining \Cref{local picture} with the analogous \cite[Chapter XII, Lemma 3.11]{ArbarelloCornalbaGriffiths} identifying both the upper and lower line with the groupoid presentation
\begin{align*}
G_x\times \operatorname{Bl}_\partial U\rightrightarrows\operatorname{Bl}_\partial U
\end{align*}
of the quotient stack $[G_x\backslash \operatorname{Bl}_\partial U]$ with projection maps $s\colon (g,u)\mapsto gu$ and $t\colon(g,u)\mapsto u$. Moreover, we see that if $U'\subseteq U$ and $W'\subseteq W$, then $\Phi_{U'\times W'}=\Phi_{U\times W}\vert_{\operatorname{Bl}_\partial \mathbf{I}_{U'\times W'}}$. It follows that the $\Phi_{U\times W}$ patch together to define a surjective local homeomorphism
\begin{align*}
\Phi\colon \operatorname{Bl}_\partial \mathbf{I}\longrightarrow \operatorname{Bl}_\partial X
\mathop{\times}_{[\Gamma_S\backslash \overline{\mathcal{T}}_{\!\!S}]} \operatorname{Bl}_\partial X.
\end{align*}
It also follows from the local identification that this map commutes with the structure maps of the two groupoid presentations.

To see that it is in fact a homeomorphism, we note that it is also injective: indeed, if $a,b\in \operatorname{Bl}_\partial \mathbf{I}$ map to the same $s\in \operatorname{Bl}_\partial X
\mathop{\times}_{[\Gamma_S\backslash \overline{\mathcal{T}}_{\!\!S}]} \operatorname{Bl}_\partial X$, then we may choose opens $U,W\subset X$ such that $\nu p_1(s)\in U$ and $\nu p_2(s)\in W$ where $\nu\colon \operatorname{Bl}_\partial X\rightarrow X$ is the canonical map and $p_1,p_2$ are the two projection maps out of the fibre product. It follows that $a$ and $b$ belong to the same $\operatorname{Bl}_\partial \mathbf{I}_{U\times W}$ and thus we must have $a=b$ as $\Phi_{U\times W}$ is injective.

This finishes the proof: having identified the groupoid presentations, we conclude that
\begin{align*}
\operatorname{Bl}_\partial \overline{\mathcal{M}}_{g,P}\xrightarrow{\ \cong \ }[\Gamma_S\backslash \overline{\mathcal{T}}_{\!\!S}];
\end{align*}
the real oriented blow-up of the Deligne--Mumford--Knudsen compactification of the moduli stack $\mathcal{M}_{g,P}$ of smooth stable $P$-pointed genus $g$ curves identifies with the Harvey compactification of $\mathcal{M}_{g,P}$ as constructed by Ivanov.
\end{proof}

\bibliographystyle{alpha}
\bibliography{biblio}

\end{document}